\documentclass[reqno,10pt]{amsart}
\usepackage[english]{babel}
\usepackage{amsmath}

\usepackage{amssymb}
\usepackage[scr=boondoxo]{mathalfa}
\usepackage{graphicx}
\usepackage{epstopdf}

\usepackage{subfigure}
\usepackage{color}

\newcommand{\bpr}{\begin{trivlist} \item[]{\bf Proof. }}
\newcommand{\epr}{\hspace*{\fill} $\qed$\end{trivlist}}

\newcommand{\be}{\begin{eqnarray}}
\newcommand{\ee}{\end{eqnarray}}
\newcommand{\ba}{\begin{align}}
\newcommand{\ea}{\end{align}}
\newcommand{\bi}{\begin{itemize}}
\newcommand{\ei}{\end{itemize}}

\newcommand{\secref}[1]{Section~\ref{sec:#1}}

\newcommand{\seclab}[1]{\label{sec:#1}}
\newcommand{\eqlab}[1]{\label{eq:#1}}
\renewcommand{\eqref}[1]{(\ref{eq:#1})}

\newcommand{\figref}[1]{Fig.~\ref{fig:#1}}
\newcommand{\figlab}[1]{\label{fig:#1}}
\newcommand{\propref}[1]{Proposition~\ref{proposition:#1}}
\newcommand{\proplab}[1]{\label{proposition:#1}}
\newcommand{\corref}[1]{Corollary~\ref{corollary:#1}}
\newcommand{\corlab}[1]{\label{corollary:#1}}
\newcommand{\lemmaref}[1]{Lemma~\ref{lemma:#1}}
\newcommand{\lemmalab}[1]{\label{lemma:#1}}
\newcommand{\remref}[1]{Remark~\ref{remark:#1}}
\newcommand{\remlab}[1]{\label{remark:#1}}
\newcommand{\thmref}[1]{Theorem~\ref{theorem:#1}}
\newcommand{\thmlab}[1]{\label{theorem:#1}}
\newcommand{\tablab}[1]{\label{tab:#1}}
\newcommand{\tabref}[1]{Table~\ref{tab:#1}}

\newcommand{\appref}[1]{Appendix~\ref{app:#1}}

\newcommand{\applab}[1]{\label{app:#1}}

\newtheorem{theorem}{Theorem}[section]
\newtheorem{proposition}[theorem]{Proposition}
\newtheorem{corollary}[theorem]{Corollary}

\newtheorem{lemma}[theorem]{Lemma}

\newtheorem{remark}[theorem]{Remark}

\numberwithin{equation}{section}

\begin{document}

 \title[Unbounded time-reversible connection problems in $\mathbb R^3$]{On the pitchfork bifurcation of the folded node and other unbounded time-reversible connection problems in $\mathbb R^3$}

\author {K. Uldall Kristiansen} 
\date\today
\maketitle

\vspace* {-2em}
\begin{center}
\begin{tabular}{c}
Department of Applied Mathematics and Computer Science, \\
Technical University of Denmark, \\
2800 Kgs. Lyngby, \\
DK
\end{tabular}
\end{center}
 \begin{abstract}
%
In this paper, we revisit the folded node and the bifurcations of secondary canards at resonances $\mu\in \mathbb N$. In particular, we prove for the first time that pitchfork bifurcations occur at all even values of $\mu$. Our approach relies on a time-reversible version of the Melnikov approach in \cite{wechselberger2002a}, used in \cite{wechselberger_existence_2005} to prove the transcritical bifurcations for all odd values of $\mu$. It is known that the secondary canards produced by the transcritical and the pitchfork bifurcations only reach the Fenichel slow manifolds on one side of each transcritical bifurcation for all $0<\epsilon\ll 1$. In this paper, we provide a new geometric explanation for this fact, relying on the symmetry of the normal form and a separate blowup of the fold lines. We also show that our approach for evaluating the Melnikov integrals of the folded node -- based upon local characterization of the invariant manifolds by higher order variational equations and reducing these to an inhomogeneous Weber equation -- applies to general, quadratic, time-reversible, unbounded connection problems in $\mathbb R^3$. We conclude the paper by using our approach to present a new proof of the bifurcation of periodic orbits from infinity in the Falkner-Skan equation and the Nos\'e equations. 
 \end{abstract}
 
 \tableofcontents
\section{Introduction}
In slow-fast systems with one fast variable and two slow ones, the folded node $p$ is a singularity of the slow flow on the fold line of a critical manifold $C$. See an illustration in \figref{foldedC}. Upon desingularization $p$ corresponds to a stable node with eigenvalues $\lambda_s<\lambda_w<0$, and its strong stable manifold $\upsilon$, tangent to the eigenvector associated with $\lambda_s$, produces a funnel region on the critical manifold, where orbits approach the singularity tangent to the weak eigendirection (associated with the eigenvalue $\lambda_w$). Due to the contraction within the funnel region, the folded node -- upon composition with a global return mapping -- provides a mechanism for producing attracting limit cycles $\Gamma_\epsilon$, see \cite{brons-krupa-wechselberger2006:mixed-mode-oscil}. In fact, a blowup of the folded node reveals one orbit $\gamma$, along which extended versions $W^{cu}$ and $W^{cs}$ of the attracting and repelling critical manifolds, respectively, twist or rotate; the number of rotations being described by $\mu:=\lambda_s/\lambda_w>1$. The twisting is such that these manifolds intersect transversally whenever $\mu\notin \mathbb N$. As a consequence, for these values of $\mu$, there exists a `weak canard' connecting extended versions of the Fenichel slow manifolds. This orbit acts, due to the twisting of $W^{cu}$ along $\gamma$, as a `center of rotation' and trajectories on either side will therefore experience small oscillations before they leave a neighborhood of $p$ by following its unstable set; in \figref{foldedC} this unstable set coincides with the positive $z$-axis. Consequently, the limit cycles $\Gamma_\epsilon$ will be of mixed-mode type where small oscillations are followed by larger ones. Such oscillations appear in many applications, perhaps most notably in chemical reaction dynamics, and the folded node has therefore gained glory as a (relatively) simple mathematical model of this phenomenon, see e.g. the review article \cite{desroches2012a}. 


Bifurcations of the weak canard occurs whenever $\mu\in \mathbb N$; in this case, for $\epsilon=0$, the twisting of $W^{cu}$ and $W^{cu}$ is such that these manifolds intersect tangentially along $\gamma$. These bifurcations were described for $\epsilon=0$ by the reference \cite{wechselberger_existence_2005}, working on the `normal form'
   \begin{align}\eqlab{fnnfeps20a}
 \begin{split}
 \dot x &=\frac12 \mu y-(\mu+1)z,\\
 \dot y &=1 ,\\
 \dot z &=x+z^2,
 \end{split}
\end{align}
and using the Melnikov approach developed in \cite{wechselberger2002a}, following \cite{vander}.
The system \eqref{fnnfeps20a} is related to the blowup of the folded node $p$ for $\epsilon=0$. In particular for \eqref{fnnfeps20a},
%
$\gamma$ takes the following form
\begin{align}
 \gamma:\,(x,y,z) = \left(-\frac{1}{4}t^2+\frac12,t,\frac12 t\right),\eqlab{gammaherea}
\end{align}
   For each odd $n$, it was shown that there is a transcritical bifurcation of `canards' connecting $W^{cu}(\mu)$ and $W^{cs}(\mu)$. For all $0<\epsilon\ll 1$, this bifurcation produces additional transversal intersections of the slow manifold. The resulting new canards -- the `secondary canards' -- produce bands on the attracting slow manifold where different number of small oscillations occur, see \cite{brons-krupa-wechselberger2006:mixed-mode-oscil,desroches2012a,wechselberger_existence_2005}.  For any even $n$, it was conjectured that a pitchfork bifurcation occurs. This was supported by numerical computations. Furthermore, in \cite[App. A]{mitry2017a} a way was found to compute a `third order' Melnikov integral using Mathematica for all even $n$ and explicit computations demonstrated that the integral was nonzero for even values of $n$ up to $20$. Following the work of \cite{wechselberger_existence_2005} this also shows that a pitchfork bifurcation occurs, at least for these values.

In this paper, we prove the pitchfork bifurcation for every even $n$ by evaluating the third order Melnikov integral analytically. Our approach is based upon a time-reversible version of the Melnikov theory of \cite{wechselberger2002a}. However, the most important insight of this paper is to characterize the manifolds $W^{cs}(\mu)$ and $W^{cu}(\mu)$ locally by solutions to higher order variational equations; this is in contrast to \cite{wechselberger_existence_2005} which uses an integral representation of these manifolds. We show that this approach, relying on reducing these variational equations to an inhomogeneous Weber equation, extends to a very general class of time-reversible, quadratic, unbounded connection problems in $\mathbb R^3$. Regardless, for the folded node the `time-reversible approach' also allows us to provide a more detailed blowup picture of the folded node, including a rigorous description of the additional transverse intersections of $W^{cs}$ and $W^{cu}$ that arise due to the pitchfork bifurcation. 

The bifurcations of canards for \eqref{fnnfeps20a}, is closely related to bifurcations of periodic orbits from heteroclinic cycles at infinity. 
The Falkner-Skan equation 
  \begin{align}
 x''' + x'' x + \mu (1-x'^2)=0,\eqlab{falkerskan0}
\end{align}
and the Nos\'e equation:
\begin{equation}\eqlab{nose}
\begin{aligned}
 \dot x &=-y-xz,\\
 \dot y&=x,\\
 \dot z&=\mu (1-x^2),
\end{aligned}
\end{equation}
are well-known examples of systems (without equilibria) possessing such bifurcations, see e.g. \cite{swinnerton-dyer2008a}, and \cite{llibre2007a} for other examples. 
The Falkner-Skan equation \eqref{falkerskan0} initially appeared in the study of boundary layers in fluid dynamics, see \cite{falkner1931a}. In this context, the physical relevant parameter regime is $\mu \in (0,2)$. However, the equation has subsequently been studied by other authors \cite{swinnertondyer1995a,sparrow2002a,swinnerton-dyer2008a} for all $\mu>0$ on the basis of the rich dynamics it possesses (including chaotic dynamics and a novel bifurcation of periodic orbits from infinity). On the other hand, the Nos\'e equations \eqref{nose} model the interaction of a particle with a heat bath \cite{nose1984a}. The system also has interesting dynamics without any equilibrium and possesses many similar properties to \eqref{fnnfeps20a} and \eqref{falkerskan0}. Nevertheless, the description of the bifurcating periodic orbits in both of these systems, is -- as noted by \cite{swinnerton-dyer2008a} -- long and cumbersome, and to a large extend, independent of standard methods of dynamical systems theory. Therefore, although the folded node will be our primary focus, a subsequent aim of the paper, is to apply the Melnikov theory, and our classification of $W^{cs}$ and $W^{cu}$ through solutions of an inhomogeneous Weber equation, to these bifurcations and present a simpler description of the emergence of periodic orbits, based on normal form theory and invariant manifolds and therefore more in tune with dynamical systems theory.

\subsection{The folded node: Further background}

Following \cite[Proposition 2.1]{wechselberger_existence_2005}, any folded node can be brought into the `normal form':
\begin{align}\eqlab{fnnfeps}
\begin{split}
 \dot x &=\epsilon \left(\frac12 \mu y-(\mu+1)z+\mathcal O(x,\epsilon,(y+z)^2)\right),\\
 \dot y &=\epsilon ,\\
 \dot z &=x+z^2+\mathcal O(xz^2,z^3,xyz)+\epsilon\mathcal O(x,y,z,\epsilon),
 \end{split}
\end{align}
by only using scalings, translations and a regular time transformation. Here $\mu:=\lambda_s/\lambda_w> 1$, and the critical manifold $C$ is approximately given by the parabolic cylinder $x=-z^2$, $z<0$ ($C_a$) being stable  and $z>0$ ($C_r$) being unstable.  See  \figref{foldedC}. Here $C_a$ is in blue, $C_r$ is in red, whereas the degenerate line $F:\,x=z=0$, being the fold line, is in green. For \eqref{fnnfeps}, the folded node $p$ (pink), on $F$, is at the origin. Furthermore, if we for simplicity ignore the $\mathcal O$-terms in \eqref{fnnfeps}, then the reduced problem on $C$ is given by
\begin{align*}
 y'&=1,\\
 2z z' &= -\frac12 \mu y + (\mu +1)z.
\end{align*}
Consider $C_a$ where $z<0$. Then multiplication of the right hand side by $-2z$ gives the topologically equivalent system
\begin{align}\eqlab{yzDS}
\begin{split}
 y' &=-2z,\\
 z' &= \frac12 \mu y-(\mu+1)z,
 \end{split}
\end{align}
on $C_a$, see \cite{wechselberger_existence_2005}.
The point $(y,z)=(0,0)$ is then a stable node of these equations with eigenvalues $-1$ and $-\mu$ and associated eigenvectors: 
\begin{align}
 (2,1)^T,\eqlab{weakEigenvector}
\end{align}
and $(2,\mu)^T$, respectively. See illustration of the reduced flow in \figref{reducedC}. Notice that the orbits on $C_r$, where $z>0$, are also orbits of \eqref{yzDS}, but their directions have to be reversed.

\begin{figure}[h!]
\begin{center}
{\includegraphics[width=.795\textwidth]{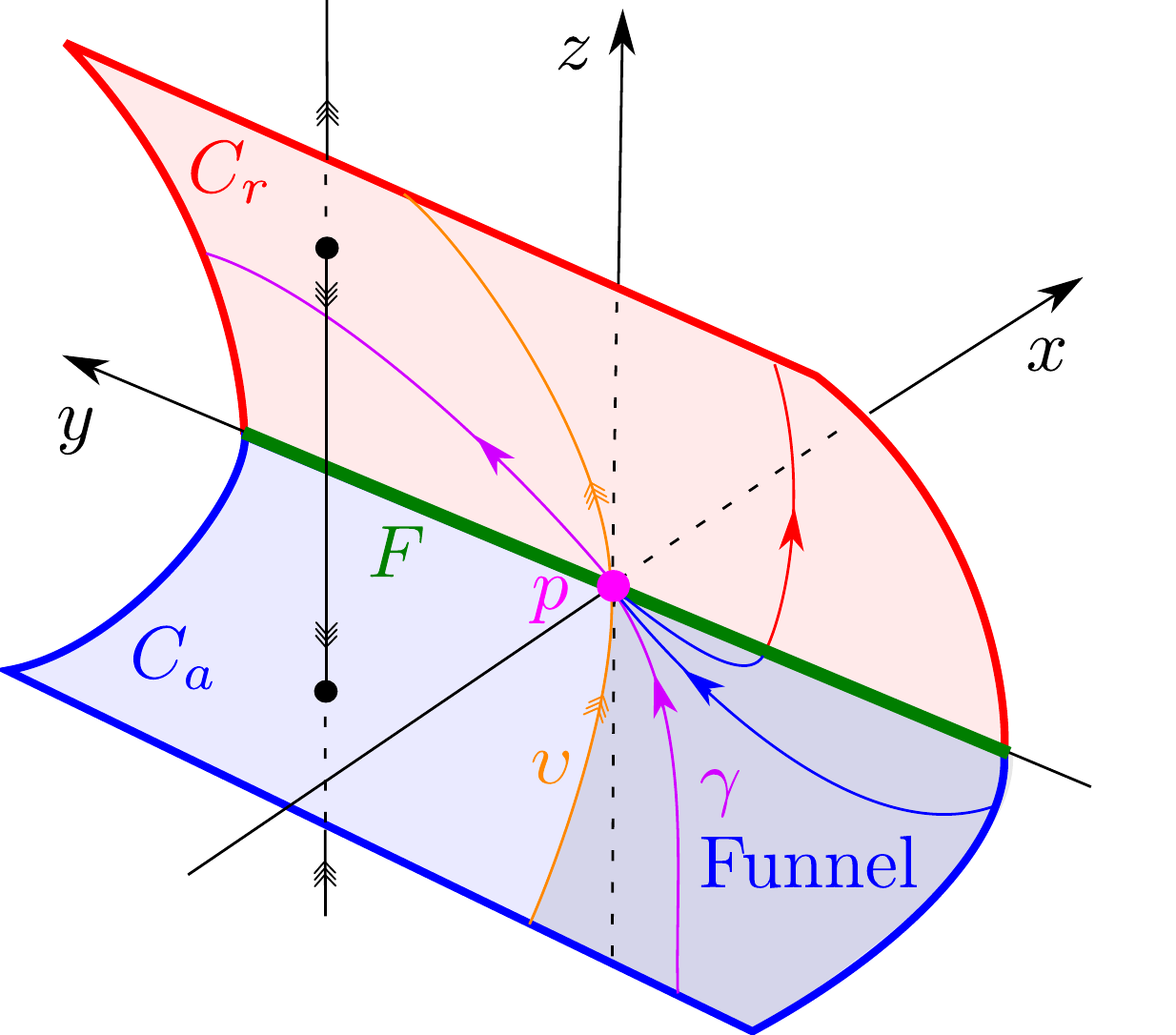}}
\end{center}
\caption{The folded node singularity $p$. Upon desingularization of the reduced problem, the folded node singularity becomes a stable node. The strong eigenvector associated with the node, gives rise to a strong stable manifold $\upsilon$ (orange) that forms a boundary of a funnel region (shaded), bounded on the other side by $F$, where trajectories approach the folded node $p$ (in finite time before desingularization), tangent to a weak eigendirection. For the system \eqref{fnnfeps} without the $\mathcal O$-terms, the weak eigenvector also produces an invariant space and an orbit $\gamma$, which we show in purple.  }
\figlab{foldedC}
\end{figure}

\begin{figure}[h!]
\begin{center}
{\includegraphics[width=.795\textwidth]{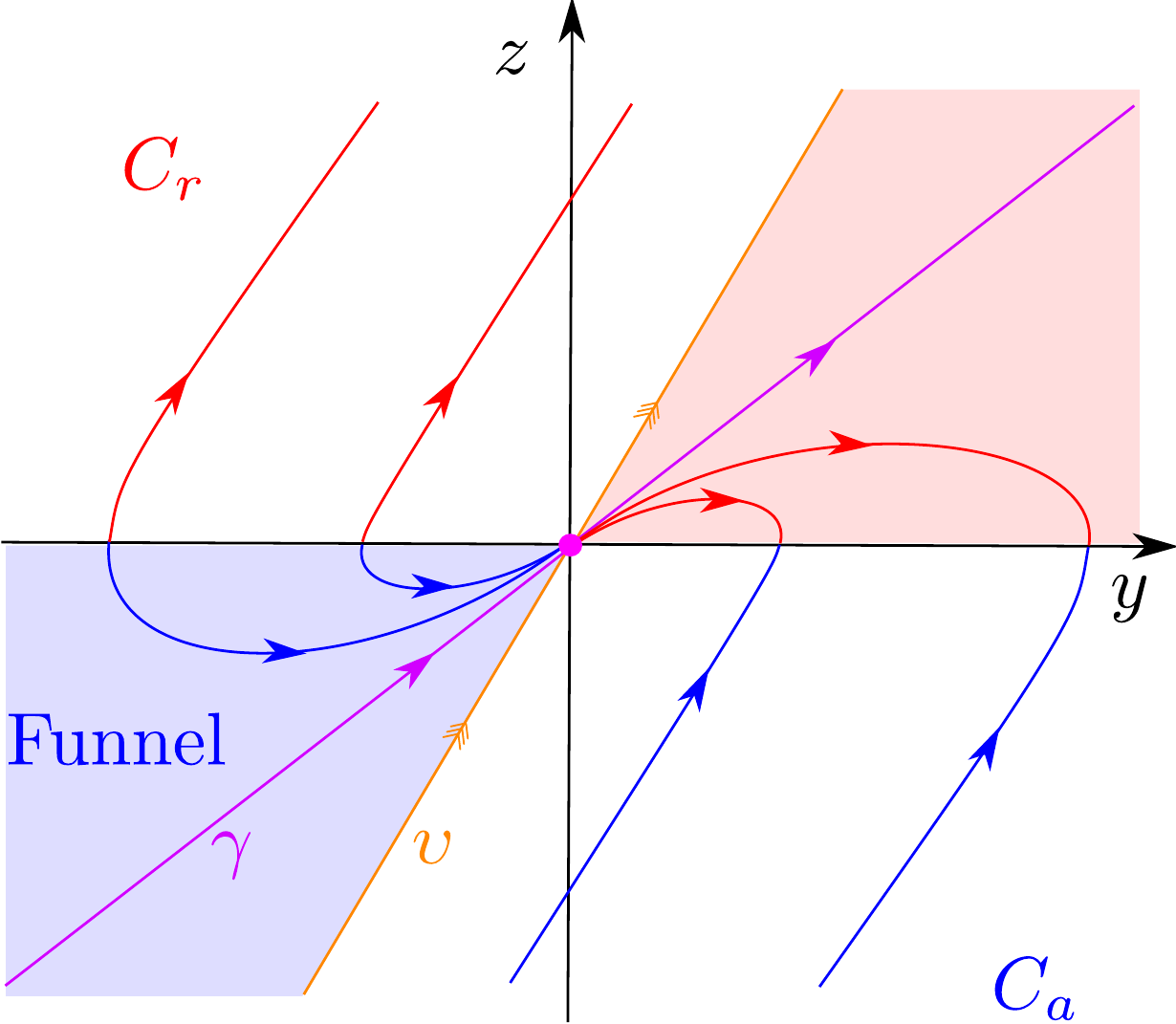}}
\end{center}
\caption{The reduced flow on $C$, recall \figref{foldedC}, projected onto the $(y,z)$-plane. The strong canard $\upsilon$ is shown in purple, whereas the weak canard $\gamma$, obtained from \eqref{fnnfeps} upon ignoring the $\mathcal O$-terms, is shown in orange. }
\figlab{reducedC}
\end{figure}
\subsubsection*{Blowup analysis}
Compact submanifolds (with boundaries) $S_a$ and $S_r$ of $C_a$ and $C_r$, respectively, bounded away from the fold line, perturb by Fenichel's theory to attracting and repelling slow manifolds $S_{a,\epsilon}$ and $S_{r,\epsilon}$ for all $0<\epsilon\ll 1$, see \cite{fen1,fen2,fen3,jones_1995}. We will refer to these manifolds as `Fenichel's (slow) manifolds'. They are nonunique but $\mathcal O(e^{-c/\epsilon})$-close. Extended versions of these invariant manifolds up close to the folded node $p$ is obtained in \cite{szmolyan_canards_2001} by blowing up the point $(x,y,z)=0$ for $\epsilon=0$. In further details, the authors apply the following blowup transformation $\mathcal B:[0,r_0)\times S^3\rightarrow \mathbb R^4$, given by
\begin{align}
r\in [0,r_0),\,(\bar x,\bar y,\bar z,\bar \epsilon)\in S^3 \mapsto \begin{cases}
                                                x&=r^2\bar x,\\
                                                y&=r\bar y,\\
                                                z&=r\bar z,\\
                                                \epsilon &=r^2\bar \epsilon,
                                               \end{cases}\eqlab{blowup}
%
\end{align}
to the extended system $(\mbox{\eqref{fnnfeps}},\dot \epsilon=0)$. For this extended system, $x=y=z=\epsilon=0$ is fully nonhyperbolic -- its linearization having only zero eigenvalues -- but upon blowup \eqref{blowup}, we gain hyperbolicity of $r=0$ after desingularization through division of the resulting right hand sides by $r$. 
In particular, setting $\bar x=-1$ in \eqref{blowup} produces the following local form of \eqref{blowup}
\begin{align}
 (r_1,y_1,z_1,\epsilon_1) \mapsto \begin{cases}
                                                x&=-r_1^2,\\
                                                y&=r_1y_1,\\
                                                z&=r_1z_1,\\
                                                \epsilon &=r_1^2\epsilon_1.
                                               \end{cases}\eqlab{barXNeg1}
\end{align}
The local coordinates $(r_1,y_1,z_1,\epsilon_1)$ provide a coordinate chart `$\bar x=-1$', covering $[0,r_0)\times S_{\bar x<0}^3$ where $S_{\bar x<0}^3:=S^3 \cap \{\bar x<0\}$. Here $r_1=0$ corresponds to $r=0$. In this chart, 
one gains hyperbolicity of $C_a$ and $C_r$ for $r=0$ upon division of the right hand side by $r_1$. By center manifold theory, this then enables an extension of the Fenichel slow manifolds $S_{a,\epsilon}$ and $S_{r,\epsilon}$ as the $\mathcal B$-image of $\epsilon=$const. sections of three-dimensional invariant manifolds $M_{a}$ and $M_{r}$, respectively, for all $0<\epsilon\ll 1$. Following \cite{krupa2010a}, we shall abbreviate these extended manifolds in the $(x,y,z)$-space by $S_{a,\sqrt{\epsilon}}$ and $S_{r,\sqrt{\epsilon}}$, respectively; see \cite{szmolyan_canards_2001} for further details. 
\begin{remark}\remlab{martinrem}
It is important to highlight that, due to the contraction towards the weak canard, the forward (backward) flow of the Fenichel manifold $S_{a,\epsilon}$ ($S_{r,\epsilon}$, respectively) is only a subset of $S_{a,\sqrt{\epsilon}}$ ($S_{r,\sqrt{\epsilon}}$). Therefore when we intersect $S_{a,\sqrt{\epsilon}}$ and $S_{r,\sqrt{\epsilon}}$, extended by the forward and backward flow, it does not follow directly that the Fenichel manifolds $S_{a,\epsilon}$ and $S_{r,\epsilon}$ also intersect.  
%
\end{remark}

Notice that the blowup approach, following \eqref{barXNeg1} and the conservation of $\epsilon$, provide control of $S_{a,\sqrt{\epsilon}}$ and $S_{r,\sqrt{\epsilon}}$ up to $\mathcal O(\sqrt{\epsilon})$-close to the folded node $p$, justifying the use of the subscripts. To describe these manifolds beyond this, we have to look at the (scaling) chart obtained by setting $\bar \epsilon=1$. This produces the following local blowup transformation
\begin{align}
 (r_2,x_2,y_2,z_2) \mapsto \begin{cases}
                                                x&=r_2^2x_2,\\
                                                y&=r_2y_2,\\
                                                z&=r_2z_2,\\
                                                \epsilon &=r_2^2.
                                               \end{cases}\eqlab{scaledK2}
\end{align}
using the chart-specified coordinates $(x_{2},y_{2},z_{2},r_2)$. The corresponding coordinate chart `$\bar \epsilon=1$'  covers $[0,r_0)\times S_{\bar \epsilon>0}^3$ where $S_{\bar \epsilon>0}^3:=S^3 \cap \{\bar \epsilon>0\}$. By inserting \eqref{scaledK2} into \eqref{fnnfeps}, dividing the right hand side by $r_2$ and subsequently setting $r_2=\sqrt{\epsilon}=0$, we obtain \eqref{fnnfeps20a}, repeated here for convenience:
  \begin{align}\eqlab{fnnfeps20}
 \begin{split}
 \dot x &=\frac12 \mu y-(\mu+1)z,\\
 \dot y &=1 ,\\
 \dot z &=x+z^2.
 \end{split}
\end{align}
In \eqref{fnnfeps20}, we have also dropped the subscripts on $(x_2,y_2,z_2)$. Two explicit algebraic solutions are known for this unperturbed system, one:
\begin{align*}
 \upsilon:\,\,(x,y,z) = \left(-\frac{\mu^2}{4}t^2+\frac{\mu}{2},t,\frac{\mu}{2} t\right)
\end{align*}
 corresponding to the `strong canard', 
while $\gamma$ in \eqref{gammaherea}, repeated here for convenience:
\begin{align}
 \gamma:\,(x,y,z) = \left(-\frac{1}{4}t^2+\frac12,t,\frac12 t\right),\eqlab{gammahere}
\end{align}
corresponds to the `weak canard', which we will focus on in this paper. 
\begin{remark}
Notice that the projection of \eqref{gammahere} onto the $(y,z)$-plane coincides with the span of the weak eigenvector \eqref{weakEigenvector}, explaining the use of `weak' in `weak canard'. Also, the orbit \eqref{gammahere} is unique as an orbit on the blowup sphere with these properties. This is obviously in contrast with reduced flow on $C_a$ where all trajectories within the funnel is assumption to the weak canard.

On a related issue, notice we abuse notation slightly: Most often, $\gamma$ will refer to \eqref{gammahere} as an orbit of  \eqref{fnnfeps20}. But by the coordinate chart `$\bar \epsilon=1$', this orbit also becomes a heteroclinic connection on $r=0,\,S^3_{\bar{\epsilon}\ge 0}$, connecting partially hyperbolic points on the equator $\bar \epsilon=0$ for the blowup system. We will use the same symbol for this orbit. At the same time, in \figref{reducedC} we also use the symbol $\gamma$ to highlight the weak eigendirection of the folded node as an attracting node of the desingularized reduced problem on $C_a$. A similar misuse of notation occurs for $\upsilon$. 
 \end{remark}

Restricting the center manifolds $M_{a}$ and $M_{r}$, obtained in the chart $\bar x=-1$, to $r=0$ we obtain, when writing the result in the chart $\bar \epsilon=1$, center-stable $W^{cs}(\mu)$ and center-unstable manifolds $W^{cu}(\mu)$ of \eqref{fnnfeps20} and $z\rightarrow \pm \infty$, respectively, consisting of solutions that grow algebraically as $t\rightarrow \pm \infty$, respectively. Following \cite{szmolyan_canards_2001}, a simple calculation shows that $W^{cs}(\mu)$ takes the local form:
\begin{align}
 W_{loc}^{cs}(\mu): x=-z^2+\frac12 (\mu+1) -\frac14 \mu yz^{-1} +z^{-2} m(yz^{-1},z^{-2}),\eqlab{WCSLOC}
\end{align}
for all $z$ sufficiently large and some smooth $m:I\times [0,\delta]\rightarrow \mathbb R$ for an appropriate interval $I\subset \mathbb R$ and $\delta>0$ sufficiently small. Due to the invariance of $\upsilon$ and $\gamma$, $m$ also satisfies $m(2,z^{-2})=m(2/\mu,z^{-2})=0$. 
By using the time-reversible symmetry $(x,y,z,t)\mapsto (x,-y,-z,-t)$ of \eqref{fnnfeps20}, a simple calculation
shows that the manifold $W^{cu}(\mu)$ takes an identical form, with the expression in \eqref{WCSLOC} 
now valid for all $z$ sufficiently negative. We illustrate the results of the blowup analysis in \figref{blowup0}. See figure caption for further description.
%
%
Guiding these manifolds along $\upsilon$ and $\gamma$ one obtains the global manifolds $W^{cs}(\mu)$ and $W^{cu}(\mu)$.  In particular, by considering the variational equations of \eqref{fnnfeps20} along $\gamma$ 
%
%
the following was shown in \cite{szmolyan_canards_2001}.
\begin{lemma}\lemmalab{lemmabif}
$W^{cs}(\mu)$ and $W^{cu}(\mu)$ intersect transversally along $\gamma$ if and only if $\mu \notin \mathbb N$. 
\end{lemma}
 By regular perturbation theory, the extension of the slow manifolds $S_{a,\sqrt{\epsilon}}$ and $S_{r,\sqrt{\epsilon}}$ by the flow are therefore smoothly $\mathcal O(r_2=\sqrt{\epsilon})$-close to $W^{cu}(\mu)$ and $W^{cs}(\mu)$, respectively, in compact subsets of the chart $\bar \epsilon=1$. Hence, as a consequence of \lemmaref{lemmabif}, for every $\mu \notin \mathbb N$ there exist a transverse intersection of $S_{a,\sqrt{\epsilon}}$ and $S_{r,\sqrt{\epsilon}}$ which is $\mathcal O(\sqrt{\epsilon})$-close to $\gamma$ in fixed compact subsets of the scaling chart. In general, recall \remref{martinrem}, it seems that there do not exist any results on how far this perturbed `weak canard' extends and whether `it' (being nonunique) actually reaches the true Fenichel slow manifolds $S_{a,\epsilon}$ and $S_{r,\epsilon}$.  
 The situation is different for $\upsilon$. First and foremost, $W^{cs}(\mu)$ and $W^{cu}(\mu)$ always intersect transversally along this orbit. Consequently, $\upsilon$ always perturbs as a `strong maximal canard' for all $0<\epsilon\ll 1$, and this perturbed version always reaches the Fenichel manifolds. This latter property is a consequence of the repelling nature of $\upsilon$, `forcing' $S_{a,\sqrt{\epsilon}}$ ($S_{r,\sqrt{\epsilon}}$) and the forward (backward) flow $S_{a,\epsilon}$ ($S_{r,\epsilon}$, respectively) to coincide near this object. 
 \begin{figure}[h!]
\begin{center}
{\includegraphics[width=.795\textwidth]{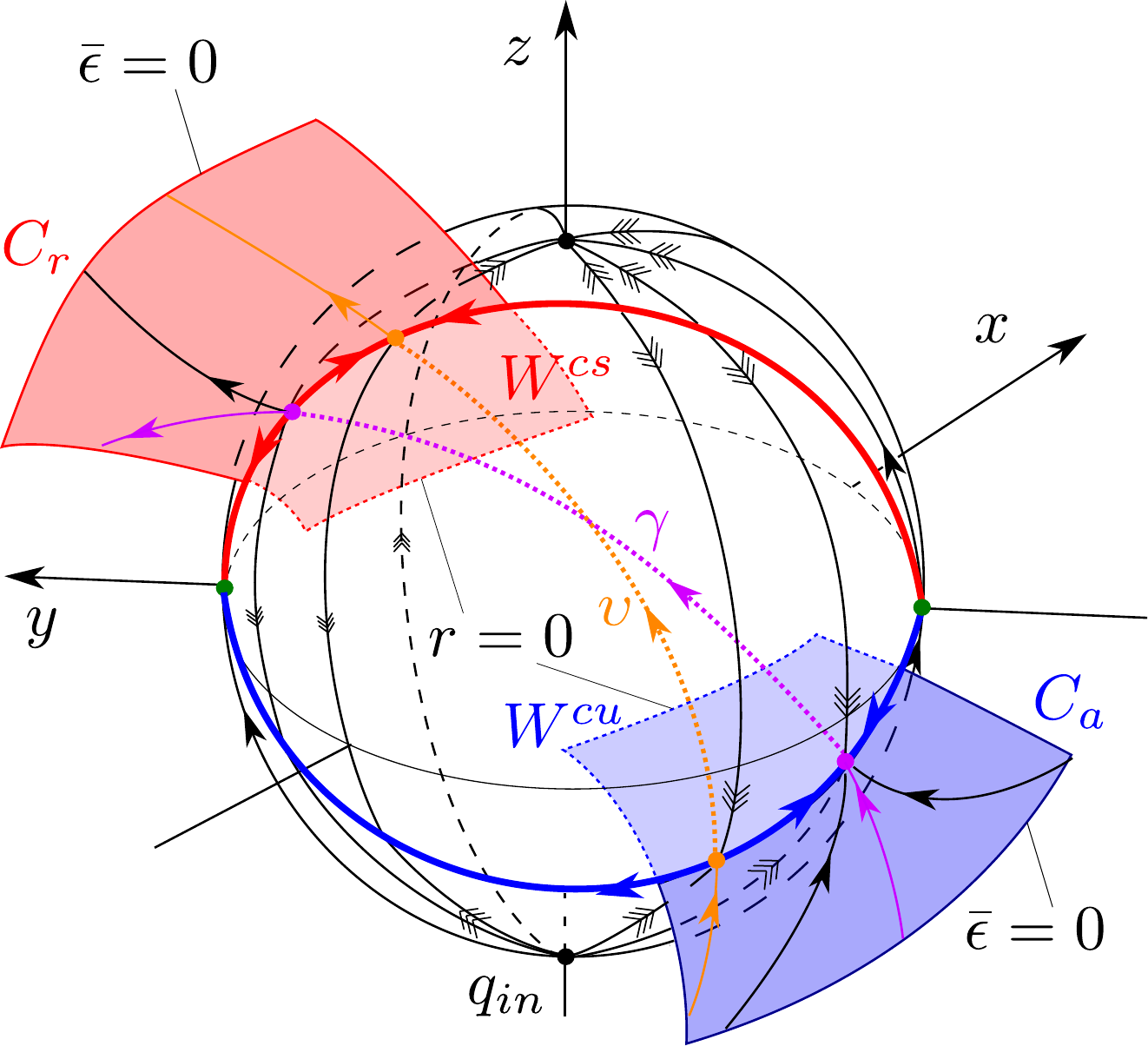}}
\end{center}
\caption{Illustration of the blowup of $p$ for $\epsilon=0$ to a hemisphere $S^3_{\bar \epsilon\ge 0}:=S^3\cap \{\bar \epsilon\ge 0\}$. In this figure, we represent $r=0$, $S^r_{\bar \epsilon\ge 0}$ -- by projection -- as a solid `ball' in the $(\bar x,\bar y,\bar z)$-space, emphasizing those objects that are inside by using dotted lines. The $S^2$ sphere, being the boundary of the ball, corresponds to $r=\bar \epsilon=0$, whereas everything inside of the ball corresponds to $r=0,\bar \epsilon>0$. Outside of the ball, we represent $r>0,\,\bar \epsilon=0$, highlighting, in particular, the critical manifolds $C_a$ and $C_r$ and their reduced flow. Through the blowup we gain hyperbolicity of $C_a$ and $C_r$ for $r=0$ (indicated by triple-headed arrows) along the lines (in blue and red, respectively) of partially hyperbolic equilibria. By center manifold theory, these lines produce two three-dimensional manifolds, $M_{a}$ and $M_r$ (not shown), having submanifolds within $r=0$, denoted by $W^{cu}$ and $W^{cs}$. These local two-dimensional manifolds are shown in lighter blue and red, respectively, since they extend inside the sphere. Also, within the sphere $r=0,\,\bar \epsilon>0$ we illustrate the orbits $\upsilon$ (orange) and $\gamma$ (purple), the `singular canards', connecting partially hyperbolic points within $r=\bar \epsilon=0$ on $W^{cu}$ and $W^{cs}$, respectively. The transversality of $W^{cu}$ and $W^{cs}$ along $\upsilon$, and along $\gamma$ for any $\mu\notin \mathbb N$, produce, transverse intersections of $S_{a,\sqrt{\epsilon}}$ and $S_{r,\sqrt{\epsilon}}$, since these objects, obtained as $\epsilon=\text{const}.$ sections of $M_{a}$ and $M_r$, respectively,  are smoothly $\mathcal O(r_2=\sqrt{\epsilon})$-close on $\bar \epsilon>0$ to $W^{cu}$ and $W^{cs}$, respectively. }
\figlab{blowup0}
\end{figure}
%
\subsection{Main result}
Using a Melnikov approach, it was shown in \cite[Theorem 3.1]{wechselberger_existence_2005} that a transcritical bifurcation of the intersection of $W^{cs}(\mu)$ and $W^{cu}(\mu)$  occurs along $\gamma$ for any odd $\mu=2k-1$, $k\in \mathbb N$. 
As a result, additional (secondary) `canards', connecting $S_{a,\sqrt{\epsilon}}$ with $S_{r,\sqrt{\epsilon}}$, exist near $\mu = 2k-1$, for all $0<\epsilon\ll 1$ by regular perturbation theory. In this paper, we prove the existence of a pitchfork bifurcation for $\mu=2k$. 
%
We then have the following complete result regarding the bifurcations of `canards' for \eqref{fnnfeps20}:
\begin{theorem}\thmlab{main1}
Consider any $n\in \mathbb N$ and let $k\in \mathbb N$ be so that 
\begin{align*}
 n = \begin{cases}
      2k-1 & n=\textnormal{odd}\\
      2k & n=\textnormal{even}
     \end{cases}.
\end{align*}
Set $\mu=n+\alpha$ and let 
\begin{align}D(v,\alpha)=0,\eqlab{bifeqn}\end{align} be the bifurcation equation (to be defined formally below in \eqref{MelnikovD} locally near $(v,\alpha)=(0,0)$) where each solution $(v,\alpha)$ corresponds to an intersection of $W^{cs}(\mu)$ and $W^{cu}(\mu)$. In particular, $D(0,\alpha)=0$ for all $\alpha$ due to the existence of the connection $\gamma$. Then 
\begin{enumerate}
 \item \label{nodd} For $n=\textnormal{odd}$, \eqref{bifeqn} is locally equivalent with the transcritical bifurcation:
 \begin{align}
  \tilde v(\tilde \alpha+(-1)^k \tilde v)=0.\eqlab{nftranscrit}
 \end{align}
\item \label{neven} For $n=\textnormal{even}$, \eqref{bifeqn} is locally equivalent with the pitchfork bifurcation:
\begin{align*}
  \tilde v (\tilde \alpha +\tilde v^2) = 0.
 \end{align*}
\end{enumerate}
In each case, the local conjugacy $\phi:(v,\alpha)\mapsto (\tilde v,\tilde \alpha)$ satisfies $\phi(0,0)=(0,0)$ and
\begin{align}
 D\phi(0,0) = \textnormal{diag}\,(d_1(n),d_2(n)) \quad \mbox{with $d_i(n)>0$ for every $k$.}\eqlab{conjugacy}
\end{align}

%
%
 
\end{theorem}

\thmref{main1} item (\ref{nodd}) is covered by \cite{wechselberger_existence_2005}. In particular, it is shown (see \cite[Propositions 3.2 \& 3.3]{wechselberger_existence_2005}) that 
\begin{align*}
 \textnormal{sign}\,\frac{\partial^2 D}{\partial v^2}(0,0) &= \textnormal{sign}\,(-1)^k,\\
 \frac{\partial^2 D}{\partial v\partial \alpha}(0,0) &> 0,
\end{align*}
which produces \eqref{conjugacy} by singularity theory \cite{golubitsky1988a}. We will therefore only prove \thmref{main1} item (\ref{neven}) in the following. Notice, however, that in \cite{wechselberger_existence_2005}, the Melnikov function is defined for all $r_2=\sqrt{\epsilon}$ sufficiently small, measuring the intersection of $S_{a,\sqrt{\epsilon}}$ and $S_{r,\sqrt{\epsilon}}$ directly (rather than measuring the $\epsilon=0$ objects $W^{cu}(\mu)$ and $W^{cs}(\mu)$). Nevertheless, seeing that the bifurcations in \thmref{main1} are for the $r_2=\sqrt{\epsilon}=0$ system, we will in this paper just focus on $\epsilon=0$ (and will only describe the perturbation of transverse intersection points into $0<\epsilon\ll 1$, see e.g. \remref{secondarycanards1} and \remref{remsecondarcanards2k}).

\subsection{Overview}
The remainder of the paper is organized as follows: In the next \secref{melnikov}, we review the Melnikov theory in \cite{wechselberger_existence_2005} in further details and extend this approach to time-reversible systems, see also \cite{knobloch1997a}. The result is collected in \thmref{Melnikov}. This is relevant for \eqref{fnnfeps20},
since this equation is time-reversible with respect the following symmetry 
\begin{align}
\sigma=\textnormal{diag}\,(1,-1,-1):\quad \mbox{If $(x,y,z)(t)$ is a solution of \eqref{fnnfeps20} then so is $\sigma (x,y,z)(-t)$.} \eqlab{sigmafnnf}
\end{align} 
This reduces the proof of our main result, \thmref{main1} item (\ref{neven}) on the pitchfork bifurcation, to evaluating two integrals; one of which is already covered by \cite{wechselberger_existence_2005}, while the other one is the `third order' Melnikov integral mentioned above. We evaluate these integrals by characterizing the manifolds $W^{cu}(\mu)$ and $W^{cs}(\mu)$ locally through solutions of `higher order variational equations' rather than, how it is done in \cite{mitry2017a}, their (implicit) formulation through integral equations. We describe our approach further in \secref{recipe} and how these variational equations can be solved upon reduction to an inhomogeneous Weber equation. In this section, we also present a general class of systems, that include the folded node normal form, the Falkner-Skan equation and the Nos\'e equations, for which we can show, see \thmref{recipe}, that our method produces closed form expressions of the appropriate Melnikov integrals. These results rely on properties of Hermite polynomials $H_n$, $n\in \mathbb N_0$. All information on these orthogonal polynomials that is relevant to the present manuscript is available in \appref{Hermite}. 

The proof of \thmref{main1} is presented in \secref{foldednode} below, see \lemmaref{main1} proven towards the end of the section. We show that our expressions agree with the computations in \cite[App. A]{mitry2017a} in \appref{compare}. Next, in \secref{blowup} we will use the time-reversible setting and the blowup approach to show that the additional intersections produced by the pitchfork bifurcation do not reach the actual Fenichel slow manifolds for any $0<\epsilon\ll 1$, see \propref{secondarcanards2k} and \remref{remsecondarcanards2k}. They do therefore not produce `true' canards. This is in contrast to the case $\mu=2k-1$ where it is known that true `secondary' canards are produced for every $\mu>2k-1$ and $0<\epsilon\ll 1$. We also provide a new geometric explanation for this property in \secref{blowup}, see \propref{gammasctranscrit}. Although these statements about canards are probably known to most experts in the field, we believe that we present the first rigorous proofs of these facts.  
In our final \secref{more}, we consider some other equations: the two-fold, the Falkner-Skan equation \eqref{falkerskan0}, and the Nos\'e equations \eqref{nose}, for which our time-reversible Melnikov approach in \secref{melnikov} is also applicable. In particular, for the Falkner-Skan and for the Nos\'e equations we provide a new geometric proof of the emergence of symmetric periodic orbits from infinity in these systems. For the Nos\'e equations, we show that for $\mu>1$ periodic orbits only emerge for $\mu\in \mathbb N$, a result that escaped \cite{swinnerton-dyer2008a}. 
We conclude the paper in \secref{conclusion}.
\section{A Melnikov theory for time-reversible systems}\seclab{melnikov}

The reference \cite{wechselberger2002a} describes a Melnikov theory for connection problems of nonhyperbolic points at infinity. In this section, we will review this approach in the context of time-reversible systems. 
%
For simplicity, we restrict to $\mathbb R^3$ and consider a general smooth ODE
\begin{align}
 \dot x &=f(x,\alpha),\eqlab{xf}
\end{align}
for $x=x(t)\in \mathbb R^3$, depending on a parameter $\alpha\in \mathbb R$,
and assume the following:
\begin{itemize}
\item[(H1)] There exists a time-reversible symmetry
\begin{align}
 (x,t)\mapsto (\sigma x,-t),\eqlab{Tsym}
\end{align}
with $\sigma\in \mathbb R^{3\times 3}$ being an involution: $\sigma^2=\textnormal{id}$, $\textnormal{id}
=\text{diag}(1,1,1)$ being the identity matrix in $\mathbb R^{3\times 3}$, such that $$f(\sigma x,\alpha) = -\sigma f(x,\alpha),$$ for all $x$ and all $\alpha$. 
\end{itemize}
Therefore: 
\begin{align*}
\mbox{If $x(t)$ is solution of \eqref{xf} then so is $\sigma x(-t)$.}
\end{align*}As is standard, we say that an orbit $x$ with parametrization $x(t)$, which is a fix-point of the symmetry: $x(t) = \sigma x(-t)$ for all $t$, is `symmetric'. On the other hand, in general two orbits $x_1\ne x_2$, for which $x_2(t)=\sigma x_1(-t)$, is said to be `symmetrically related'.

Furthermore, we say that a solution $x(t)$ of \eqref{xf} has algebraic growth for $t\rightarrow \infty$ if there exists a $\nu>0$ such that $\sup_{t\ge 0}\vert x(t)\vert (\vert t\vert +1)^{-\nu}<\infty$ for $\nu>0$ large enough. Specifically, we define the Banach space
\begin{align*}
C_{b,+}(\nu): = \{x\in C([0,\infty),\mathbb R^3)\vert \sup_{t\ge 0}\{\vert x(t)\vert(\vert t\vert +1)^{-\nu} <\infty\},
\end{align*}
 for $\nu>0$ fixed, see \cite{wechselberger2002a}. Similarly, a solution $x(t)$ of \eqref{xf} has algebraic growth for $t\rightarrow - \infty$ if $\sup_{t\le 0}\vert x(t)\vert (\vert t\vert +1)^{-\nu}<\infty$ for $\nu>0$ large enough. Accordingly, we define 
 \begin{align*}
 C_{b,-} (\nu):= \{x\in C((-\infty,0],\mathbb R^3)\vert \sup_{t\le 0}\{\vert x(t)\vert(\vert t\vert+1)^{-\nu} <\infty\}.
\end{align*}
We will suppress $\nu$ in $C_{b,+}(\nu)$ and $C_{b,-}(\nu)$ whenever it is convenient to do so.

Next, we assume
\begin{itemize}
 \item[(H2)] For $\alpha=0$ there exists a symmetric solution $\gamma$ with parametrization $\gamma(t)$, $t\in \mathbb R$, of at most algebraic growth for $t\rightarrow \pm \infty$, i.e. $\gamma(t)\in C_{b,+}(\nu)$ with $\nu>0$ large enough. Without loss of generality we suppose that 
 \begin{align*}
  \gamma(0)=0.
 \end{align*}

 \item[(H3)] There exists a three-dimensional smooth invariant manifold $W^{cs}$ in the extended system
  \begin{align}\eqlab{xfext}
  \begin{split}
   \dot x &=f(x,\alpha),\\
   \dot \alpha &=0.
   \end{split}
  \end{align}
Here $W^{cs}$ denotes the center-stable manifold consisting of all solution curves $(x(t),\alpha)$ of \eqref{xf} near (and including) $(x,\alpha)=(\gamma(t),0)$ (in a sense specified below) for which $x(t)\in C_{b,+}(\nu)$, for $\nu>0$ large enough.
\end{itemize}
$W^{cs}$ is foliated by two-dimensional invariant manifolds $W^{cs}(\alpha)$ of \eqref{xf} for fixed values of $\alpha$, sufficiently small.

By (H1) and (H3), there exists a center-unstable manifold 
\begin{align}\eqlab{Wcu}
W^{cu}(\alpha):=\sigma W^{cs}(\alpha):= \{\sigma q\vert q\in W^{cs}(\alpha)\},\end{align} consisting of all solution curves $(z(t),\alpha)$ of \eqref{xf} near (and including) $(x,\alpha)=(\gamma(t),0)$ for which $x(t)\in C_{b,-}(\nu)$.
\begin{itemize}
 \item[(H4)] Let $U:=\textnormal{span}(\dot \gamma(0))$. Then for $\alpha=0$ there exists a one-dimensional linear space $V$ such that
 $$T_{\gamma(0)} W^{cs}(0)\cap T_{\gamma(0)} W^{cu}(0) =U \oplus  V,$$
 is a two-dimensional subspace.
\end{itemize}
(H4) implies that the manifolds $W^{cs}(0)$ and $W^{cu}(0)$ intersect tangentially along $\gamma$ for $\alpha=0$. In fact, seeing that $W^{cu}=\sigma W^{cs}$ we have
\begin{lemma}\lemmalab{eqvH4}
 The following statements are equivalent:
 \begin{enumerate}
  \item (H4) holds.
  \item \label{3ii} $T_{\gamma(0)} W^{cs}(0)=T_{\gamma(0)} W^{cu}(0)$ and the intersection of $W^{cs}(0)$ and $W^{cu}$ along $\gamma$ is tangential.
  \item \label{3iii} $T_{\gamma(0)} W^{cs}(0)$ is an invariant subspace for $\sigma$: $x\in T_{\gamma(0)} W^{cs}(0)$ $\Longrightarrow$ $\sigma x\in T_{\gamma(0)} W^{cs}(0)$. 
  \item \label{3iv} The variational equation along $\gamma$ for $\alpha=0$:
\begin{align}
\dot z &=A(t) z,\eqlab{vargen}
\end{align}
where $A(t)=D_xf(\gamma(t),0)$, has two linearly independent solutions $z_1(t)=\dot \gamma(t)$ and $z_2(t)$ for which $z_i\in C_{b,+}\cap C_{b,-}$. 
 \end{enumerate}
\end{lemma}
\begin{proof}
 (1) $\Leftrightarrow$ (2) is trivial, seeing that $W^{cs}(0)$ and $W^{cu}(0)$ are two-dimensional manifolds. (\ref{3iii}) $\Leftrightarrow$ (\ref{3ii}) follows from the following computation:
$T_{\gamma(0)} W^{cu}(0) = T_{\gamma(0)} \sigma W^{cs}(0) = T_{\gamma(0)} W^{cs}(0)$ by (H1), recall \eqref{Wcu}. 
Finally, (1) $\Leftrightarrow$  (\ref{3iv}) is standard, see \cite[Proposition 4.4]{szmolyan_canards_2001}. Indeed, variations along the two-dimensional space $T_{\gamma(0)} W^{cs}(0)\cap T_{\gamma(0)} W^{cu}(0)$ correspond to algebraic growth as $t\rightarrow \pm \infty$. 
\end{proof}
%
%

%

Next, following \cite{wechselberger_existence_2005} let 
\begin{align}
W = T_{\gamma(0)} W^{cs}(0)^\perp.\eqlab{WDefinition} 
\end{align}
Then $\mathbb R^3 = U\oplus V\oplus W$. Let $e_u$, $e_v$ and $e_w$ be unit vectors spanning $U$, $V$ and $W$, respectively, and denote the coordinates of any $x\in \mathbb R^3$ with respect to this basis $\{e_u,e_v,e_w\}$ by $(u,v,w)$. 
 Fix $r>0$ small and let $B_r$ be the ball of radius $r$ centered at $\gamma(0)$. We then define a local section $\Sigma$ transverse to $\gamma$ at $\gamma(0)=0$ by
 \begin{align*}
  \Sigma = \{V\oplus W\} \cap B_r.
 \end{align*}
 Notice that $\Sigma$ -- in the $(u,v,w)$-coordinates -- is contained within the $(v,w)$-plane.
Next, we write $x=z+\gamma(t)$ following \cite{wechselberger2002a} such that 
\begin{align}
 \dot z &=A(t)z+g(t,z,\alpha),\eqlab{zg}
\end{align}
where $g(t,z,\alpha) = f(\gamma(t)+z,\alpha)-f(\gamma(t),0)-A(t)z$. Also $g(t,0,0)=0,D_z g(t,0,0)=0$ and notice that \eqref{vargen} is the variational equation along $\gamma(t)$. Furthermore:
\begin{lemma}\lemmalab{AProperty}
\begin{align*}
 \sigma A(-t) &= -A(t) \sigma,\\
 \sigma g(-t,z,\alpha) & = -g(t,\sigma z,\alpha),
\end{align*}
for all $t,z,\alpha$.
\end{lemma}
 \begin{proof}
 Follows directly from the time-reversible symmetry of $f$, recall (H1).
 \end{proof}

Let $\Phi(t,s)$ be the state-transition matrix of \eqref{vargen}. Then by (H3) and (H4) there exists a continuous projection $P:[0,\infty) \rightarrow\mathbb R^3$ such that
 \begin{align}
  \textnormal{Range}\,P(0) = U\oplus V,\quad
  \textnormal{ker}\, P(0) &=W,\nonumber
 \end{align}
 and
\begin{align*}
 P(t) \Phi(t,s) = \Phi(t,s) P(s),
\end{align*}
for all $t,s\ge 0$. Furthermore, if $Q(s)=I-P(s)$ then 
\begin{align}
  \textnormal{ker}\, Q(0) = U\oplus V,\quad  \textnormal{Range}\,Q(0) = W.\eqlab{rangeQ}
\end{align}
and it follows that
\begin{align}\eqlab{PhiProps}
\begin{split}
 \Vert \Phi(t,s) P(s)\Vert &\le K(t-s+1)^\theta,\\
 \Vert \Phi(s,t) Q(t)\Vert &\le K e^{-\eta(t-s)},
 \end{split}
\end{align}
for some $K\ge 1$, $\theta,\eta\ge 0$ and all $0\le s \le t$, see e.g. \cite{wechselberger2002a,vander}. By assumption (H1), \lemmaref{AProperty} and \eqref{PhiProps} we also have:
\begin{lemma}\lemmalab{psi}
$\Phi$ is symmetric in the following sense:
\begin{align*}
 \sigma \Phi(t,s) = \Phi(-t,-s) \sigma.
\end{align*}
Also, $t\mapsto \sigma P(-t)\sigma^{-1}$ and $t\mapsto \sigma Q(-t)\sigma^{-1}$ are continuous projection operators such that 
\begin{align*}
 \Vert \Phi(t,s) \sigma P(-s)\sigma^{-1}\Vert& = \Vert \sigma \Phi(-t,-s) P(-s)\sigma^{-1}\Vert \le K (s-t+1)^\theta,\\
 \Vert \Phi(s,t) \sigma Q(-t)\sigma^{-1}\Vert &= \Vert \sigma \Phi(-s,-t) Q(-t)\sigma^{-1}\Vert \le K e^{-\eta(s-t)}.
\end{align*}
for all $t\le s\le 0$.  
\end{lemma}
\begin{proof}
 Straightforward calculation.
\end{proof}

Consider the adjoint equation of \eqref{vargen}:
 \begin{align}
  \dot \psi + A(t)^T \psi = 0,\eqlab{adjvargen}
 \end{align}
 and notice that 
 \begin{align}
 (\psi,t)\mapsto (\sigma^T \psi,-t),\eqlab{adjvargensym}
 \end{align}is a time-reversible symmetry for \eqref{adjvargen} by (H1). 
  Then 
 \begin{lemma}\lemmalab{wvector}
 Let $\psi_*(t)$ be a solution of \eqref{adjvargen}. Then $\psi_*(t)$ decays exponentially for $t\rightarrow \pm \infty$ if and only if $\phi_*(0)\in W$. 
 \end{lemma}
 \begin{proof}
  Standard, see \cite{wechselberger2002a}. 
 \end{proof}
In the following, we fix a specific $\psi_*(t)$ by setting $\psi_*(0)=e_w$. Since $\Phi^T(s,t)=\Phi^{-T}(t,s)$ is a state-transition matrix of \eqref{adjvargen}, we can then write $\psi_*(t)$ as
 \begin{align}
 \psi_*(t)=\Phi^T(0,t)e_w.\eqlab{psiStarHere}
 \end{align}
  We now have the following important result.
 \begin{lemma}\lemmalab{VWinvariant}
  $V$ and $W$ are one-dimensional invariant subspaces of  $\sigma$ and $\sigma^T$, respectively. Hence; there exists $\sigma_i\in \{\pm 1\}$, $i=v,w$, such that $$\sigma\vert_V=\sigma_v \textnormal{id}, \quad \sigma^T\vert_{W}=\sigma_w \textnormal{id},$$ where $\sigma_i =\pm 1$ for $i=v,w$.
 \end{lemma}
 \begin{proof}
 First, regarding the $\sigma^T$-invariance of $W$: The solution $\psi_*(t)$ of \eqref{adjvargensym} is exponentially decaying for $t\rightarrow \pm \infty$. Clearly, the symmetrically related solution $\sigma^T\psi_*(-t)$ satisfies the same properties, and hence $\sigma^T\psi_*(0)\in W$ by \lemmaref{wvector}; the $\sigma^T$-invariance of $W$ therefore follows. Next, since $\gamma(t)$ is symmetric it follows by differentiation with respect to $t=0$ that $\sigma\vert_{U}=-\textnormal{id}$. But then since $U\oplus V$ is invariant with respect to $\sigma$, recall \lemmaref{eqvH4} item (\ref{3iii}), and $\sigma^2=\textnormal{id}$, the statement about $\sigma\vert_{V}$ also follows from a straightforward calculation. %
 \end{proof}
 In fact, in the $(u,v,w)$-coordinates 
 \begin{align}
 \sigma = \text{diag}\,(-1,\sigma_v,\sigma_w).\eqlab{sigmauvw}
 \end{align}

Next, for \eqref{zg}, it can by variation of constants -- following \cite{wechselberger2002a} -- be shown that $z(t)\in C_{b,+}$, with $z(0) \in \Sigma$, if and only if there exists a $v \in V$ such that 
\begin{align}
 z(t) =  \Phi(t,0)v +\int_0^t P(t) \Phi(t,s) g(s,z(s),\alpha)ds + \int_{\infty}^t Q(t) \Phi(t,s)g(s,z(s),\alpha)ds.\eqlab{ztVOC}
\end{align}
This enables an analytic characterization of the (nonunique) invariant manifold $W^{cs}(\mu)$, which is essential for the Melnikov approach, as follows.
Let the mapping $z\mapsto T(z)$ be defined on $C_{b,+}$ so that $T(z)(t)$ is the right hand side of \eqref{ztVOC} and consider a sufficiently small neighborhood $N$ of $(v,\alpha)=(0,0)$. Then, upon possible modification (or cut-off) of $f$ (and therefore of $g$ in \eqref{zg}), as in center manifold theory \cite{car1}, we obtain for each $(v,\alpha)\in N$, a unique fix-point $z_*(v,\alpha)$ of $T$: $T(z_*)=z_*$, see \cite{wechselberger2002a}. Henceforth we will assume that such a modification of $f$ (and therefore of $g$) has been made. It is also standard, see also \cite{wechselberger2002a}, to show that $z_*$ is smooth, with each partial derivative belonging to $C_{b,+}(\nu)$ for $\nu$ large enough. In this way, 
\begin{align*}
 W^{cs}(\alpha)  = \{z_*(v,\alpha)(t)\vert (v,\alpha)\in N,\,t\in \mathbb R\}.
\end{align*}
\begin{remark}
 In our examples, including \eqref{fnnfeps20}, the invariant manifolds $W^{cs}(\mu)$ will be obtained, not as fix-points of \eqref{ztVOC}, but as center manifolds upon appropriate Poincar\'e compactification. For the analysis of the implications of the bifurcations of $\gamma$, it will be important to study the reduced problem on such center manifolds. For the Falkner-Skan equation and the Nos\'e equations the `selection' of these nonunique manifolds will be crucial to our analysis. 
\end{remark}

%
%
\subsection{The Melnikov function}
For $\alpha$ sufficiently small, we write $W_0^{cs}(\alpha)=W^{cs}(\alpha)\cap \Sigma$ and $W_0^{cu}(\alpha)=W^{cu}(\alpha)\cap \Sigma$ locally within the $(v,w)$-plane as smooth graphs
\begin{align}
 w = h_{cs}(v,\alpha),\eqlab{graphh}
\end{align}
and 
\begin{align}
 w = h_{cu}(v,\alpha),\eqlab{graphhu}
\end{align}respectively, over $v$. 
Following \cite{wechselberger2002a}, we then define the Melnikov function as
\begin{align}
 D(v,\alpha) &= h_{cu}(v,\alpha)-h_{cs}(v,\alpha).\eqlab{MelnikovD}
\end{align}
Clearly, a root of $D$ corresponds to an intersection of $W_0^{cs}(\alpha)$ and $W_0^{cu}(\alpha)$ and, hence, to an intersection of $W^{cs}(\alpha)$ and $W^{cu}(\alpha)$.  Furthermore, the intersection is transverse if and only if the root is simple. 
%
We now prove the following.
\begin{lemma}\lemmalab{MelnikovDNew}
%
\begin{align}
 D(v,\alpha) &= \sigma_w h_{cs}(\sigma_v v,\alpha) - h_{cs}(v,\alpha).\eqlab{MelnikovDNew}
 \end{align}
\end{lemma}
\begin{proof}
Following \eqref{MelnikovD}, we simply have to show that 
\begin{align}
 h_{cu}(v,\alpha) = \sigma_w h_{cs}(\sigma_v v,\alpha),\eqlab{hcuExpr}
\end{align}
for all $v$ and $\alpha$. We will show this using the integral representation \eqref{ztVOC} as follows.
Let $z_*(v,\alpha)\in C_{b,+}$ be the fix-point of the mapping $T$, with $T(z)(t)$ being defined as the right hand side of \eqref{ztVOC}, so that $z_*(v,\alpha)(t)\in W^{cs}(\alpha)$ for all $t$ and $z_*(v,\alpha)(0)=(v,h_{cs}(v,\alpha))$ within $\Sigma$ in the $(v,w)$-coordinates. Then $\sigma z_*(v,\alpha)(-\cdot)\in C_{b-}$ so that $\sigma z_*(v,\alpha)(-t)\in W^{cu}$ with $$\sigma z_*(v,\alpha)(0)=(\sigma_v v,\sigma_w h_{cs}(v,\alpha)),$$ 
writing the right hand side in the $(v,w)$-coordinates, recall \eqref{sigmauvw}.
Since $W^{cu}=\sigma W^{cs}$, we conclude from \eqref{graphhu}$_{v=\sigma_v v}$ that 
\begin{align*}
 h_{cu}(\sigma_v v,\alpha) = \sigma_w  h_{cs}(v,\alpha).
\end{align*}
This shows \eqref{hcuExpr}, seeing that $\sigma_v^2=1$.
\end{proof}

Now, we are ready to present the following, final result on the Melnikov integral, which is a translation of \cite[Theorem 1]{wechselberger2002a} to the time-reversible setting in $\mathbb R^3$.
\begin{theorem}\thmlab{Melnikov}
Let $t\mapsto z_*(v,\alpha)(t) \in C_{b,+}$ be the solution of \eqref{zg} with initial conditions 
\begin{align}
z_*(v,\alpha)(0) = (v,h_{cs}(v,\alpha)),\eqlab{zStarInit}
\end{align}
with respect to the $(v,w)$-coordinates, on $W^{cs}_0(\alpha)$ within $\Sigma\subset \{u=0\}$. Then
%
\begin{align}
 D(v,\alpha) = \int_0^\infty \langle \psi_*(s), g(s,z_*(v,\alpha)(s),\alpha)-\sigma_w g(s,z_*(\sigma_v v,\alpha)(s),\alpha)\rangle ds. \eqlab{DvalphaThm}
\end{align}
\end{theorem}
\begin{proof}
%
%
The result follows from \cite[Theorem 1]{wechselberger2002a}, upon setting  $h_{cs}e_w=h_+$ and $\sigma^T h_{cs}e_w = h_-$. In further details, we simply set
$t=0$ in \eqref{ztVOC}$_{z=z_*(v,\alpha)}$:
\begin{align}
 z_*(v,\alpha)(0) = v - Q(0) \int_0^\infty  \Phi(t,s)g(s,z_*(v,\alpha)(s),\alpha)ds.\eqlab{zStar0}
\end{align}
The last term on the right hand side -- by \eqref{rangeQ} -- belongs to $W$ for all $\vert v\vert,\vert \alpha\vert \le \delta$; whence, 
\begin{align}
 h_{cs}(v,\alpha)=\langle e_w,-Q(0) \int_0^\infty \Phi(0,s) g(s,z_*(v,\alpha)(s),\alpha)ds\rangle.
\end{align}
Then upon using \eqref{psiStarHere}, we obtain the desired form
\begin{align*}
 h_{cs}(v,\alpha)& = -\int_0^\infty \langle \psi_*(s),g(s,z_*(v,\alpha)(s),\alpha)\rangle ds. 
\end{align*}
The result then follows from \eqref{MelnikovDNew}.
\end{proof}

\subsection{A recipe for computing appropriate Melnikov integrals}\seclab{recipe}
We can use \eqref{DvalphaThm} to describe the bifurcations of heteroclinic connections of $W^{cs}(\alpha)$ and $W^{cs}(\alpha)$, provided that we can determine the partial derivatives of $D$. We now describe a set of assumptions, covering all of the cases we study below, where we can describe a specific procedure for doing this. We start with the following.
\begin{itemize}
\item[(H5)] Suppose that $\gamma=(0,0,t)$ for all $t\in \mathbb R$ and all $\alpha$.
\end{itemize}
Upon a linear change of coordinates, we may also suppose without loss of generality that $\sigma$ is diagonal, recall \eqref{sigmauvw}. 
\begin{itemize}
\item[(H6)] Suppose that $\sigma=\text{diag}(1,-1,-1)$.
\end{itemize}
Notice that $\gamma$ is symmetric with respect to this $\sigma$. 
\begin{itemize}
\item[(H7)] Suppose that $f(\cdot,\alpha)$ is quadratic.
\end{itemize}
We can then show that all relevant Melnikov integrals can be evaluated in closed form.
\begin{theorem}\thmlab{recipe}
 Suppose (H3) and (H5)-(H7). 
 \begin{enumerate}
 \item Then $f$, upon scaling $x_3$ and $t$ if necessary, takes the following form
 \begin{align}
  f(x,\alpha) = \begin{pmatrix}
          x_2 a_2+x_1 x_2 a_{12} + x_1x_3\\
          x_1b_1+x_1^2 b_{11}+x_2^2 b_{22} \\
          1+x_1c_1+ x_1^2 c_{11}+x_2^2c_{22}+x_2x_3 c_{23} 
         \end{pmatrix}.\eqlab{fGeneral}
 \end{align}
 with each coefficient $a_2,a_{12},b_1,b_{11},b_{22},c_1,c_{11},c_{22},c_{23}$ depending smoothly upon the parameter $\alpha$. 
%
\item 
 Furthermore, let
 \begin{align*}
  \beta := -(a_{2}b_1+1).
 \end{align*}
Then (H4) is satisfied if and only if $\beta \in \mathbb N_0$. 
\item 
Next, let $\alpha=0$ be so that $\beta\in \mathbb N_0$ and consider $D(v,\alpha)$ as in \eqref{DvalphaThm}. Then $D(0,\alpha)=0$ for all $\alpha$ and we have the following:
\begin{enumerate}
\item Suppose $\beta=\text{odd}$. Then $\sigma_v=-1$, $\sigma_w=1$ and $v\mapsto D(v,\alpha)$, see \eqref{DvalphaThm}, is an even function, so that $\frac{\partial^{2j} D}{\partial v^{2j}}(0,0)=0$ for all $j\in \mathbb N_0$. Furthermore, the following partial derivatives of $D$ can be evaluated in closed form:
\begin{align*}
 \frac{\partial^2 D}{\partial v\partial \alpha}(0,0),\quad \frac{\partial^3 D}{\partial v^3}(0,0).
\end{align*}
If these quantities are nonzero, then the bifurcation equation $D(v,\alpha)=0$ is locally equivalent with the pitchfork normal form. 
\item Suppose $\beta=\text{even}$. Then $\sigma_v=1$, $\sigma_w=-1$. In this case, the following partial derivatives can be evaluated in closed form
\begin{align*}
 \frac{\partial^2 D}{\partial v\partial \alpha}(0,0),\quad \frac{\partial^2 D}{\partial v^2}(0,0).
\end{align*}
If these quantities are nonzero, then the bifurcation equation $D(v,\alpha)=0$ is locally equivalent with the transcritical normal form. 
\end{enumerate}
%
\end{enumerate}
\end{theorem}
\begin{proof}
Regarding (1): The general form in \eqref{fGeneral} is a simple consequence of $f(0,0,t,\alpha)=(0,0,1)^T$ by (H5), $\sigma f(x,\alpha) + f(\sigma x,\alpha)=0$ for all $x$ and all $\alpha$ by (H6). Furthermore, we conclude, using a Poincar\'e compactification and center manifold theory, that (H3) implies that the term $x_1x_3a_{13}$ in the first component of $f$ satisfies $a_{13}>0$ and subsequently that there is no term $x_2x_3 b_{23}$ in the second component of $f$. Seeing that $a_{13}>0$, we finally obtain \eqref{fGeneral} by scaling $x_3$ and $t$ as follows: $\tilde x_3 = \sqrt{a_{13}}x_3$, $\tilde t=\sqrt{a_{13}} t$.

Next regarding (2): 
%
The general form \eqref{fGeneral} produces 
\begin{align}
 A(t) = \begin{pmatrix}
         t  & a_2 & 0\\
         b_{1} & 0 & 0\\
         c_{1} & t c_{23}&0
        \end{pmatrix}.\eqlab{AGen}
\end{align}
%
%
%
%
But then, upon differentiating the first equation for $z_1$ in the variational equation \eqref{vargen} one more time with respect to $t$, we can write the equation for $z_1$ as a Weber equation:
 \begin{align}
  L_\beta z_1 = 0 ,\eqlab{weber}
 \end{align}
 where the second order operator $L_\beta$ is defined by the general expression:
 \begin{align}
  L_\beta q := \ddot q-t\dot q+\beta q,\eqlab{Lmexpr}
 \end{align}
 for any $q\in C^2$. From $z_1$, $z_2$ and $z_3$ can be determined by successive integration of
 \begin{equation}\eqlab{z2z3red}
 \begin{aligned}
  \dot z_2 &=b_{1}z_1,\\
  \dot z_3 &=c_{1}z_1+tc_{23} z_2.
 \end{aligned}
 \end{equation}
We have.
 \begin{lemma}\lemmalab{LmHmHl}
Let $H_n$, $n\in \mathbb N_0$, denote the $n$th degree Hermite polynomial. Then for all non-negative integers $m$ and $l$, the following holds 
\begin{align}
L_m H_l(t/\sqrt{2}) = (m-l) H_l(t/\sqrt{2}).\eqlab{rangeLm}
\end{align}
In particular, 
\begin{align}
H_m(\cdot/\sqrt{2})\in \textnormal{ker}\,L_m.\eqlab{KerLm}
\end{align}
\end{lemma}
\begin{proof}
 Follows from \eqref{prop1} and \eqref{prop2} in \appref{Hermite}.
\end{proof}
 Therefore for $\beta\in \mathbb N_0$ there exists by \eqref{KerLm} an algebraic solution $x(t)=H_\beta(t/\sqrt{2})$ of \eqref{weber}. Inserting this solution into \eqref{z2z3red} produces, upon using \eqref{prop1} and \eqref{prop2} in \appref{Hermite}, an algebraic solution of \eqref{vargen}:
 \begin{align}
 z(t)= \begin{pmatrix}
 H_\beta(t/\sqrt{2})\\
 \frac{b_{1}}{2(\beta+1)} H_{\beta+1}(t/\sqrt{2})\\
 \frac{c_{1}}{2(\beta+1)} H_{\beta+1}(t/\sqrt{2})+\frac{b_1c_{23}}{2(\beta+1)} \left(\frac{1}{2(\beta+3)}H_{\beta+3}(t/\sqrt{2})+H_{\beta+1}(t/\sqrt{2})\right)
 \end{pmatrix}.\eqlab{algsol}
\end{align} 
 $\beta \in \mathbb N_0\Rightarrow (H4)$ is then a consequence of \lemmaref{eqvH4}, see item \ref{3iv}. On the other hand, if $\beta\notin \mathbb N_0$ then there are no algebraic solutions of \eqref{weber}, see e.g. \cite{AbramowitzStegun1964},  and therefore (H4) does not hold, see again \lemmaref{eqvH4} item \ref{3iv}. This completes the proof of (2). 


To finish the proof, we just need to verify the claims about $\sigma_v$, $\sigma_w$ and the partial derivatives of $D$ at $v=\alpha=0$ in item (3). For this we first determine $\psi_*$. Suppose that $\beta\in \mathbb N_0$. Then a simple computation shows that
the adjoint equation can be written as a second order equation for $z_1$:
\begin{align*}
 \ddot z_1 &=-tz_1 -(\beta +2)z_1.
\end{align*}
Substituting $z_1=e^{-t/2}\tilde z_1$ gives 
\begin{align*}
L_{\beta+1}z_1 =0,
\end{align*}
recall \eqref{Lmexpr},
upon dropping the tilde. Using \eqref{prop1}, we obtain the following expression for $\psi_*$: 
 \begin{align}
  \psi_*(t) = e^{-t^2/2} c \begin{pmatrix}
               H_{\beta+1}(t/\sqrt{2}) \\
               -\frac{a_{2}}{\sqrt{2}} H_{\beta}(t/\sqrt{2})\\
               0
              \end{pmatrix},\eqlab{psiStarGen}
 \end{align}
for some constant $c$, ensuring that $\psi_*(0)=e_w$ has length $1$. The statements regarding $\sigma_v$ and $\sigma_w$ are then simple consequences of \eqref{algsol} and \eqref{psiStarGen}, recall \lemmaref{VWinvariant}. 

Regarding the partial derivatives of $D$, we focus on $\beta=\text{odd}$ and the closed form expression for $\frac{\partial^3 D}{\partial v^3}(0,0)$ in (3a). Both $\frac{\partial^2 D}{\partial v\partial \alpha}(0,0)$ and $\beta=\text{even}$ in (3b) are similar, but simpler, and therefore left out.  Notice also that the statements about the local equivalence with the pitchfork and the transcritical normal form follows from singularity theory, see e.g. \cite{golubitsky1988a}. 

Let
\begin{align}
z'=(z'_1,z'_2,z'_3)^T:=\frac{\partial z_*}{\partial v}(0,0),\quad z''=(z''_1,z''_2,z''_3)=\frac{\partial^2 z_*}{\partial v^2}(0,0).\eqlab{zp}
\end{align}
Notice, by linearization of \eqref{ztVOC}$_{z=z_*(v,\alpha)}$ using $z_*(0,0)=0$, it follows that $z'$ is determined by an appropriate normalization of \eqref{algsol}. 
 Then by \eqref{psiStarGen} we conclude that $\frac{\partial^3 D}{\partial v^3}(0,0)$ is a linear combination of terms of the following form
\begin{align}
 \int_0^\infty e^{-t^2/2} H_\beta(t/\sqrt{2}) z_m'(t) z_n''(t) dt, \eqlab{Dv3Gen}
\end{align}
for $n,m=1,2,3$. Obviously, this linear combination can be stated explicitly in terms of \eqref{fGeneral}. However, the details are not important here. Next, suppose that $z''$ is a finite sum of Hermite polynomials:
\begin{align}
 z''= \sum_{i\in I} v_i H_{i}(t/\sqrt{2}),\eqlab{zppGen}
\end{align}
with $I\subset \mathbb N_0$ being a finite index set and $v_i\in \mathbb R^3$, for any $i\in I$. Then upon inserting \eqref{zppGen} into \eqref{Dv3Gen} we obtain a linear combination of terms of the following form
\begin{align}
 \int_0^\infty e^{-t^2/2} H_\beta(t/\sqrt{2})H_{i}(t/\sqrt{2})H_j(t/\sqrt{2})dt,\eqlab{D3Terms}
\end{align}
with $i,j\in \mathbb N_0$ and $\beta\in \mathbb N_0$. Again, the details are not important. However, each term of the form \eqref{D3Terms} can be determined in closed form using \eqref{productRuleHnm} and the statement of the theorem therefore follows once we have shown \eqref{zppGen}. 
For this we insert $z=z_*(v,\alpha)$ into \eqref{zg} and differentiate the resulting equation twice with respect to $v$. We then evaluate the resulting expression at $(v,\alpha)=(0,0)$ and again use that $z_*(0,0)=0$. This gives a linear inhomogeneous equation for $z''$ of the following form
\begin{align*}
 \dot z'' = A(t)z'' + \begin{pmatrix}
                       2(a_{12}z_1'z_2'+a_{13} z_1'z_3')\\
                       2(b_{11} (z_1')^2+b_{22} (z_2')^2)\\
                       2(c_{11}(z_1')^2+c_{22} (z_2')^2+c_{23}z_2'z_3')                       
                      \end{pmatrix},
\end{align*}
By \eqref{AGen} we can therefore write the equation for $z_1''$ as a second order equation and obtain \textit{an inhomogeneous Weber equation}:
\begin{align}
 L_\beta z_1'' = 2\frac{d}{dt} \left(a_{12}z_1'z_2'+a_{13} z_1'z_3'\right)+2a_2 \left(b_{11}(z_1')^2 +b_{22} (z_2')^2\right).\eqlab{z1pp}
\end{align}
Notice that by \eqref{algsol} the right hand side of \eqref{z1pp} is a sum of products of Hermite polynomials. The product rule in \eqref{productRuleHnm} in \appref{Hermite} allows us to write this sum of products as a sum of Hermite polynomials only. A simple calculation, using \eqref{prop2}, then shows that this sum only consists of Hermite polynomials of even degree. We can then solve \eqref{z1pp} using \lemmaref{LmHmHl}. In particular,
by linearity and the fact that $m=\beta=\text{odd}$ and $l=\text{even}$, it follows from \eqref{rangeLm} that there exists an algebraic solution of \eqref{z1pp} of the form 
\begin{align*}
 z_1''(t)=\sum_{j\in J} d_j H_{2j}(t/\sqrt{2}),
\end{align*}
for a finite index set $J\subset \mathbb N_0$ and $d_j\in \mathbb R$, $j\in J$.
This is $z_1''$, since (a) it has the desired algebraic growth and (b) $z_1''(0)=0$. The latter property (b) is a consequence of $H_{2i}(0)=0$ for each $i\in \mathbb N_0$. 
Upon integrating the equations for $z_2''$ and $z_3''$ we obtain a solution of the form
\eqref{zppGen}. This completes the proof of \thmref{recipe}.

\end{proof}
\begin{remark}\remlab{procedure}
 The general procedure for evaluating $\frac{\partial^3 D}{\partial v^3}(0,0)$, described in the proof above, can essentially be summarized as follows:
\begin{itemize}
\item \textbf{Step (a)}. Insert $z_*(v,\alpha)$ into \eqref{zg} and differentiate the resulting equation twice with respect to $v$. This characterizes $z''$ as a solution of a higher order variational equation.
\item \textbf{Step (b)}. This equation can by (H5)-(H7) be reduced to an inhomogeneous Weber equation, see \eqref{z1pp}, with right hand side as a finite sum of products of Hermite polynomials. 
\item \textbf{Step (c)}. We can then use \eqref{productRuleHnm} in \appref{Hermite} to reduce these products to sums of Hermite polynomials and solve the resulting equation for $z''$ using \lemmaref{LmHmHl}. 
\item \textbf{Step (d)}. Finally, we insert the resulting expression for $z''$ in step (c) into $\frac{\partial^3 D}{\partial v^3}(0,0)$ producing a sum of integrals of the form \eqref{D3Terms}. Finally, these integrals can be evaluated using \eqref{intHnHmHl}.
\end{itemize}
In principle, this procedure can be extended to cases where $f$ is a polynomial of higher degree. It is still possible to reduce the equation for $z''$ to \textnormal{an inhomogeneous Weber equation} with a right hand side consisting of a sum of Hermite polynomials. However, I have not managed to find an appropriate general setting for this, where one can show that this right hand side does not involve $H_\beta(t/\sqrt{2})$. These terms belong to the kernel of $L_\beta$, recall \eqref{KerLm}, and we can therefore not apply \lemmaref{LmHmHl}, as described in step (c), to inhomogeneous terms of this form. There are similar issues related to formalising the procedure iteratively to obtain closed form expressions for higher order derivatives of $z_*$ and $D$. 
\end{remark}

\begin{remark}\remlab{mitryrem}
 In \cite{mitry2017a}, in the case of the folded node normal form, a more implicit, integral representation of $z''$ is obtained by differentiating the fix-point equation $z_*=T(z_*)$, with $T$ defined by the right hand side of \eqref{ztVOC}. Inserting this expression into the expression for $\frac{\partial^3 D}{\partial v^3}(0,0)$ gives a double integral, which (presumable) can be evaluated in Mathematica. The reference \cite{mitry2017a} computes all values up to $n=20$ for the folded node. (
 We compare these values with our closed form expressions in \appref{compare}.) However, it is unclear if it is possible to evaluate such double integral directly. 
 \end{remark}


\section{Application of \thmref{Melnikov} to the folded node: Proof of \thmref{main1}}\seclab{foldednode}

In this section, we now prove \thmref{main1}. For this, we follow the recipe in \secref{recipe}, summarized in \remref{procedure}. But first we follow \cite{wechselberger_existence_2005} and rectify $\gamma$, recall \eqref{gammahere}, to the $x_3$-axis by introducing
\begin{align}\eqlab{rectify}
\begin{split}
 x_1 &:= x+z^2-\frac12,\\
 x_2 &:=y-2z,\\
 x_3 &:=2z,
 \end{split}
\end{align}
so that 
 \begin{align}
 \gamma:\, x(t) = \left(0,0,t\right),\, t\in \mathbb R,\eqlab{gammaSol}
\end{align}
using -- for simplicity -- the same symbol for the same object in the new variables. Notice that \cite{mitry2017a} rectifies $\gamma$ in a slightly different way, see \appref{compare}.
Inserting \eqref{rectify} into \eqref{fnnfeps20}
produces
 \begin{align}\eqlab{fnnf}
 \begin{split}
\dot x_1 &=\frac12 \mu x_2+x_1x_3,\\
\dot x_2&=-2x_1,\\
\dot x_3 &=2x_1+1,
\end{split}
\end{align}
which we, see also \cite[Eq. (2.18)]{wechselberger_existence_2005}, will study in the following. 
The system \eqref{fnnf} is time-reversible with respect to the same symmetry as in \eqref{sigmafnnf}. It is easy to see that \eqref{fnnf} satisfies the assumptions (H5)-(H7) and \thmref{recipe} applies. In particular, we have
\begin{align*}
 \beta=\mu-1.
\end{align*}
We will therefore apply the procedure used in the proof of this result, specifically see \thmref{recipe} item (3a) and \remref{procedure}, to \eqref{fnnf}. This will enable a proof of \thmref{main1}.

In the following, let $$\mu=n+\alpha.$$ Then by \lemmaref{lemmabif} and the analysis in \cite{wechselberger_existence_2005} the assumptions (H1)-(H4) are satisfied. At this stage, we keep $n\in \mathbb N$ general. For $n=\textnormal{odd}$ the results in the following are therefore covered by \cite{wechselberger_existence_2005}; the only exception is that we exploit the symmetry to simplify some of the expressions. 
%

Writing $x=\gamma(t)+z$ gives
\begin{align}\eqlab{veqns}
\begin{split}
 \dot z_1 &= t z_1+ \frac{n}{2} z_2 +g(z,\alpha),\\
 \dot z_2 &=-2z_1,\\
 \dot z_3 &=2z_1,
 \end{split}
\end{align}
where
\begin{align}
 g(z,\alpha) = \frac12 \alpha z_2+z_1z_3.\eqlab{gFunc}
\end{align}
In comparison with \eqref{zg}, `$g$' for \eqref{veqns} is really $(g(z,\alpha),0,0)$, but it is useful to allow for a slight abuse of notation and let $g$ here refer to the first nontrivial coordinate function only. 
Setting $g=0$ (ignoring the nonlinear terms) in \eqref{veqns} produces the variational equations about $\gamma$
\begin{align}
 \dot z = A(t) z\quad \mbox{with}\quad 
 A(t) = \begin{pmatrix}
         t & \frac{n}{2} & 0\\
         -2 & 0 &0\\
         2 &0 &0
        \end{pmatrix}.\eqlab{vvar}
\end{align}
%
By differentiating the first equation for $z_1$ in \eqref{vvar} with respect to $t$ , we obtain a Weber equation for $z_1$:
\begin{align*}
L_{n-1}z_1=0,
\end{align*}
recall \eqref{Lmexpr}. For $n\in \mathbb N$, it has an algebraic solution:
\begin{align}
 z_1 = H_{n-1}(t/\sqrt{2}).\eqlab{z1Algebraic}
\end{align}
Inserting \eqref{z1Algebraic} into the remaining equations for $z_2$ and $z_3$, we obtain the following state-transition matrix $\Phi(t,s)$ of \eqref{vvar}:
\begin{align}
 \Phi(t,0) &= \begin{pmatrix}
              \frac{1}{H_{n-1}(0)}H_{n-1}(t/\sqrt{2}) & * & 0\\
              -\frac{\sqrt{2} }{n H_{n-1}(0)}H_{n}(t/\sqrt{2}) & * & 0\\
              \frac{\sqrt{2} }{n H_{n-1}(0)}H_{n}(t/\sqrt{2}) & * & 1
             \end{pmatrix},\quad n=\textnormal{odd}, \eqlab{Psi2k_1}\\
             \Phi(t,0)& = \begin{pmatrix}
              * & -\frac{n }{\sqrt{2} H_{n}(0)}H_{n-1}(t/\sqrt{2}) & 0\\
             * & \frac{1}{H_{n}(0)}H_{n}(t/\sqrt{2})& 0\\
              * & 1-\frac{1}{H_{n}(0)}H_{n}(t/\sqrt{2}) &  1
             \end{pmatrix},\quad n=\textnormal{even},\eqlab{Psi2k}
\end{align}
see also \cite[Eqns. (3.14)-(3.15)]{wechselberger_existence_2005}. Following the notation in \cite{wechselberger_existence_2005}, the asterisks denote a separate linearly independent solution that we do not specify and which will play no role in the following. 
Setting $t=0$ in the expressions for $\Phi$ above, it follows that the $z_1z_2$-plane has the following decomposition:
\begin{align*}
 V \oplus W,
\end{align*}
where
\begin{align}
 V& = \textnormal{span}\, e_v,\quad \left\{\begin{array}{cc}
             e_v = (1,0,0)^T  & n=\textnormal{odd}\\
            e_v = (0,1,0)^T & n=\textnormal{even}
            \end{array}\right., \nonumber\\
W&= \textnormal{span}\, e_w,\quad \left\{\begin{array}{cc}
             e_w=(0,1,0)^T & n=\textnormal{odd}\\
             e_w=  (1,0,0)^T & n=\textnormal{even}
            \end{array}\right.,\eqlab{Wscap}
\end{align}
recall (H4) and \eqref{WDefinition}. Also $U=\textnormal{span}(0,0,1)^T$ for all $n\in \mathbb N$. Therefore by \eqref{MelnikovDNew}:
\begin{proposition}\proplab{Dgeneral}
For \eqref{fnnf}, 
\begin{align*}
\sigma_v &= \begin{cases} 
         1 &n=\textnormal{odd}\\
         -1 &n=\textnormal{even}
           \end{cases},\\
\sigma_w &= \begin{cases} 
         -1 &n=\textnormal{odd}\\
         1 &n=\textnormal{even}
           \end{cases},
\end{align*}
and
\begin{align}
 D(v,\alpha) = \left\{\begin{array}{cc }
         -2h_{cs}(v,\alpha) &n=\textnormal{odd}\\
         h_{cs}(-v,\alpha)-h_{cs}(v,\alpha)&n=\textnormal{even}
        \end{array}\right..\eqlab{DvalphaH}
\end{align}In particular,
\begin{enumerate}
\item \label{Dalpha} $D(0,\alpha)=0$ for all $\alpha$ and any $n$. 
 \item \label{Dodd} For $n=\textnormal{even}$, $v\mapsto D(v,\alpha)$ is an odd function for every $\alpha$. 
\end{enumerate}

\end{proposition}
\begin{proof}
 Follows from the definition of $\sigma_i$, $i=v,w$ in \lemmaref{VWinvariant} and from \eqref{MelnikovDNew}, see also \thmref{recipe}. 
\end{proof}
As a corollary, we have the following.
\begin{corollary}\corlab{symmetryLemma}
 Consider $n=\textnormal{odd}$. Then solutions of $D(v,\alpha)=0$ bifurcating from $v=\alpha=0$ correspond
to symmetric solutions of \eqref{veqns}, i.e. they are fix-points of the time-reversible symmetry $\sigma$.
 \end{corollary}
\begin{proof}
 Follows from \eqref{DvalphaH}$_{n=\textnormal{odd}}$ and the fact that any solution of $D(v,\alpha)$ in this case lies within $w=0$, corresponding to $x_2=x_3=0$, being the fix-point set of the symmetry $\sigma$. \end{proof}
In contrary, when $n=\textnormal{even}$ bifurcating solutions come in pairs (as a pitchfork bifurcation) that are related by the symmetry. 

Now, finally by \thmref{Melnikov} we have.
\begin{lemma}\lemmalab{hallo}
Let $z_*(v,\alpha)(\cdot)$ be the solution with $z_*(v,\alpha)(0)=(v,h_{cs}(v,\alpha))\in W_0^{cs}(\alpha)\subset \Sigma$ in the $(v,w)$-coordinates. Then $z_*(v,\alpha)\in C_{b,+}$ and
\begin{align}
 D(v,\alpha)& = \int_0^\infty e^{-t^2/2} \times \nonumber\\
 &\left\{\begin{array}{cc} 
                                        \frac{2\sqrt{2} }{n H_{n-1}(0)}H_{n}(t/\sqrt{2})  g(z_*(v,\alpha)(t),\alpha) & n=\textnormal{odd}\\
                                        \frac{1}{H_{n}(0)}H_{n}(t/\sqrt{2})  (g(z_*(v,\alpha)(t),\alpha)-g(z_*(-v,\alpha)(t),\alpha)& n=\textnormal{even}
                      \end{array}\right. dt.\eqlab{Dformula}
\end{align}

\end{lemma}
\begin{proof}
 We have in these expressions used that 
 \begin{align*}
  \psi_*(t) &= e^{-t^2/2}\begin{pmatrix}
             \frac{\sqrt{2} }{nH_{n-1}(0)}H_{n}(t/\sqrt{2})\\
             \frac{1}{H_{n-1}(0)}H_{n-1}(t/\sqrt{2})\\
             0
            \end{pmatrix},\quad \text{for $n$ odd},\\
  \psi_*(t) &= e^{-t^2/2}\begin{pmatrix}
             \frac{1}{H_n(0)}H_{n}(t/\sqrt{2})\\
             \frac{n }{\sqrt{2}H_n(0)}H_{n-1}(t/\sqrt{2})\\
             0
            \end{pmatrix},\quad \text{for $n$ even},
 \end{align*}
 see \cite[Eq. (3.12)]{wechselberger_existence_2005}, is the solution of the adjoint equation \eqref{adjvargen} with $\psi_*(0)=e_w$ which decays exponentially for $t\rightarrow \pm \infty$, recall \lemmaref{wvector}. 
\end{proof}
%

We are now ready to prove \thmref{main1} item (\ref{neven}).
\begin{proof}[Proof of \thmref{main1} item (\ref{neven})]

We now write $n=2k$, $k\in \mathbb N$. By \propref{Dgeneral}, items (\ref{Dalpha}) and (\ref{Dodd}), it follows that
\begin{align}
 D(0,\alpha) = \frac{\partial^{2i} D}{\partial v^{2i}}(0,\alpha) = 0,\eqlab{Dsimple}
\end{align}
for all $\alpha$ and all $i\in \mathbb N$.
%
%
Next, we have the following lemma.
\begin{lemma}\lemmalab{main1}
For $n=2k$ with $k\in \mathbb N$ the following expressions hold:
\begin{align}
\frac{\partial^2 D}{\partial v\partial \alpha}(0,0)  &=\frac{\sqrt{\pi}(2k)!!}{\sqrt{2}(2k-1)!!},\eqlab{DvalphaFinal}\\
 \frac{\partial^3 D}{\partial v^3}(0,0) &={3\sqrt{2\pi} (2k+1) (2k)!!^4} \nonumber\\
 &\times \sum_{j=0}^{2k-1} \frac{(4k-1-2j)!}{(2k-1-2j)j! (j+1)! (2k-1-j)!^2(2k-j)!^2}.\eqlab{Dv3Final}
\end{align}
Let $c_{jk}$ be the elements of the sum in \eqref{Dv3Final}:
\begin{align*}
 c_{kj} :=  \frac{(4k-1-2j)!}{(2k-1-2j)j! (j+1)! (2k-1-j)!^2(2k-j)!^2},
\end{align*}
for $j=0,\ldots,2k-1$. 
Then for every $k\in\mathbb N$
\begin{align*}
 \left\{\begin{array}{ccc}
                 c_{kj}  >0 & \textnormal{for} &j=0,\ldots,k-1,\\
                 c_{kj}  <0 & \textnormal{for}&j=k,\ldots,2k-1,
                 \end{array} \right.
\end{align*}
and
\begin{align}
 \left|\frac{c_{k(k-l)}}{c_{k(k+l-1)}}\right|>2^{2l-1}\ge 2 \eqlab{ckprop},
\end{align}
for all $l=1,\ldots,k$. 
\end{lemma} 
We turn to the proof of \lemmaref{main1} once we have shown that \lemmaref{main1} implies \thmref{main1}. 
For this, we first estimate the negative terms (where $j=k,\ldots,2k-1$) of the sum in \eqref{Dv3Final} using \eqref{ckprop} to obtain the following positive lower bound,
\begin{align}
 \frac{\partial^3 D}{\partial v^3} (0,0) > {3\sqrt{2\pi} (2k+1) (2k)!!^4}\sum_{j=0}^{k-1} \frac12 c_{kj},\eqlab{Dv3Bound}
\end{align}
of $\frac{\partial^3 D}{\partial v^3}(0,0)$, with the right hand side being the sum of only positive terms. Consequently, the expressions \eqref{Dsimple}, \eqref{DvalphaFinal}, \eqref{Dv3Final} -- together with singularity theory \cite{golubitsky1988a} -- proves our main result \thmref{main1} item (\ref{neven}) on the pitchfork bifurcation.

%
%



\end{proof}
\begin{proof}[Proof of \lemmaref{main1}]
Let $z_*(v,\alpha)$ be as described. Recall, that it has algebraic growth as $t\rightarrow \infty$, and that $z_*(0,\alpha)=0$ for all $\alpha$ since $\gamma$ is a solution for all $\alpha$. Furthermore, by differentiating \eqref{veqns} with respect to $v$ and setting $v=\alpha=0$, we obtain the following equation
\begin{align*}
 \dot z'  = A(t)z',
\end{align*}
with $z'=\frac{\partial z_*}{\partial v}(0,0)$, recall \eqref{zp}.
Here $A(t)$ is given in \eqref{vvar} with $n=2k$.
Consequently, by \eqref{Psi2k} we have
\begin{align}
 z'(t) = \begin{pmatrix}
       -\frac{\sqrt{2} k }{H_{2k}(0)}H_{2k-1}(t/\sqrt{2})\\
             \frac{1}{H_{2k}(0)}H_{2k}(t/\sqrt{2}) \\
              1-\frac{1}{H_{2k}(0)}H_{2k}(t/\sqrt{2})
      \end{pmatrix},\eqlab{zprime}
\end{align}
see also \cite{wechselberger_existence_2005}.
Let $z''(t)=\frac{\partial^2 z_*}{\partial v^2}(0,0)$, recall \eqref{zp},
denote the second partial derivative of $z_*$. We now follow the steps in \remref{procedure}.

\textbf{Step (a)}. 
By differentiating \eqref{veqns} once more with respect to $v$ and setting $v=\alpha=0$ we obtain a `higher order variational equation'
\begin{align}
 \dot z'' = A(t)z'' + \begin{pmatrix}
                       z_1'z_3'\\
                       0\\
                       0
                      \end{pmatrix}.\eqlab{zpp}
\end{align}
%
%
We have the following.
\begin{lemma}

\begin{align}
\frac{\partial^2 D}{\partial v\partial \alpha}(0,0)&=\frac{1}{H_{2k}(0)^2} \int_0^\infty e^{-t^2/2} H_{2k}(t/\sqrt{2})^2 dt \eqlab{Dvalpha}\\
 \frac{\partial^3 D}{\partial v^3}(0,0) &= \frac{3}{\sqrt{2}H_{2k}(0)}\int_{0}^\infty e^{-t^2/2} H_{2k+1}(t/\sqrt{2}) z_3'(t) z_3''(t) dt,\eqlab{Dv3}
\end{align}
 where $z_3'$ and $z_3''$ in \eqref{Dv3} are defined by \eqref{zp}, respectively.
\end{lemma}
\begin{proof}
We use \eqref{Dformula} with $n=2k$:
\begin{align*}
 D(v,\alpha) = \frac{1}{H_{2k}(0)}\int_0^\infty e^{-t^2/2}H_{2k}(t/\sqrt{2})  (g(z_*(v,\alpha)(t),\alpha)-g(z_*(-v,\alpha)(t),\alpha)dt,
\end{align*}
recall \eqref{gFunc}. To obtain \eqref{Dvalpha} we differentiate this expression partially with respect to $v$ and $\alpha$. This gives
\begin{align*}
 \frac{\partial^2 D}{\partial v\partial \alpha}(0,0)&= \frac{1}{H_{2k}(0)}\int_0^\infty e^{-t^2/2} {H_{2k}(t/\sqrt{2})} z_2'(t) dt\\
 &=\frac{1}{H_{2k}(0)^2} \int_0^\infty e^{-t^2/2} H_{2k}(t/\sqrt{2})^2 dt,
\end{align*}
by \eqref{zprime} 
upon setting $v=\alpha=0$.

For \eqref{Dv3}, we also perform a direct calculation to obtain
 \begin{align*}
\frac{\partial^3 D}{\partial v^3}(0,0) = \frac{6}{H_{2k}(0)} \int_{0}^\infty e^{-t^2/2} H_{2k}(t/\sqrt{2}) \left(z_1'(t)z_3''(t)+z_1''(t)z_3'(t)\right)dt.
 \end{align*}
Following \eqref{veqns}, 
\begin{align*}
 z_1^{(i)} = \frac12 \dot z_3^{(i)},
\end{align*}
for $i=1,2$, and hence 
 \begin{align*}
\frac{\partial^3 D}{\partial v^3}(0,0) = \frac{3}{H_{2k}(0)} \int_{0}^\infty e^{-t^2/2} H_{2k}(t/\sqrt{2}) \frac{d}{dt}\left(z_3'(t)z_3''(t)\right)dt.
 \end{align*}
 By integration by parts, using $z_3'(0)=0$ and \eqref{prop1} in \appref{Hermite}, we then obtain the result. 
\end{proof}
Using that $H_{2k}$ is an even function, the formula in \eqref{HOrtho} in \appref{Hermite} then produces the desired expression \eqref{DvalphaFinal} for $\frac{\partial^2 D}{\partial v\partial \alpha}(0,0)$ in \lemmaref{main1}. 

To prove the remaining expression \eqref{Dv3Final} in \lemmaref{main1} for $\frac{\partial^3 D}{\partial v^3}(0,0)$, we determine $z_3''$,  which is the only remaining unknown in the expression \eqref{Dv3}.  

\textbf{Step (b)}. We do so by first writing \eqref{zpp} as \textit{an inhomogeneous Weber equation} for $z_1''$: 
\begin{align}
 L_{2k-1}z_1'' = 2\frac{d}{dt}\left(z_1'z_3'\right),\eqlab{v1pp}
\end{align}
recall the definition of second order linear differential operator $L_{2k-1}$ in \eqref{Lmexpr}. 

\textbf{Step (c)}.  We then use \lemmaref{LmHmHl} to solve the linear, inhomogeneous equation \eqref{v1pp} for the algebraic solution $z_1''$ with $z_1''(0)=0$, once we have written the right hand side of \eqref{v1pp} as a finite sum of Hermite polynomials. For this we use \eqref{productRuleHnm}:
\begin{lemma}\lemmalab{z1z3p}
The following holds true for any $k\in \mathbb N$:
\begin{align*}
 z_1' z_3' = \frac{\sqrt{2} k}{H_{2k}(0)^2}\sum_{j=0}^{2k-1} \begin{pmatrix}
                                                              2k-1\\
                                                              j
                                                             \end{pmatrix}\begin{pmatrix}
                                                              2k\\
                                                              j
                                                             \end{pmatrix}2^j j! H_{4k-1-2j}(t/\sqrt{2})-\frac{\sqrt{2}k}{H_{2k}(0)}H_{2k-1}(t/\sqrt{2}).
\end{align*}
\end{lemma}
\begin{proof}
 Calculation.
\end{proof}
Consequently, we have
\begin{lemma}
 The following holds true for any $k\in \mathbb N$:
 \begin{align}
  z_1''(t) &= -2 \frac{d}{dt}\bigg(\frac{\sqrt{2}k}{H_{2k}(0)^2}\sum_{j=0}^{2k-1}  \frac{1}{2k-1-2j}\begin{pmatrix}
                                                              2k-1\\
                                                              j
                                                             \end{pmatrix}\begin{pmatrix}
                                                              2k\\
                                                              j
                                                             \end{pmatrix}2^j j! H_{4k-1-2j}(t/\sqrt{2})\nonumber\\
                                                             &\quad \quad  \quad +\frac{\sqrt{2}k}{H_{2k}(0)}H_{2k-1}(t/\sqrt{2})\bigg).\eqlab{z1ppExpr}\\
                                                             z_3''(t) &= -4 \bigg(\frac{\sqrt{2}k}{H_{2k}(0)^2}\sum_{j=0}^{2k-1}  \frac{1}{2k-1-2j}\begin{pmatrix}
                                                              2k-1\\
                                                              j
                                                             \end{pmatrix}\begin{pmatrix}
                                                              2k\\
                                                              j
                                                             \end{pmatrix}2^j j! H_{4k-1-2j}(t/\sqrt{2})\nonumber\\
                                                             &\quad \quad \quad  +\frac{\sqrt{2}k}{H_{2k}(0)}H_{2k-1}(t/\sqrt{2})\bigg).\eqlab{z3ppExpr}
 \end{align}
\end{lemma}
\begin{proof}
 The expression in \eqref{z1ppExpr} follows from a simple calculation using \lemmaref{LmHmHl}, \lemmaref{z1z3p} and \eqref{prop2}. \eqref{z3ppExpr} is then obtained by integrating $\dot z_3''=2z_1''$, recall \eqref{zpp} and using $z_3''(0)=0$. 
\end{proof}

\textbf{Step (d)}.  We then have.
\begin{lemma}
The following holds for any $k\in \mathbb N$:
\begin{align*}
 \frac{\partial^3 D}{\partial v^3}(0,0) &= \frac{6k}{H_{2k}(0)^4}\sum_{j=0}^{2k-1}\frac{1}{2k-1-2j}\begin{pmatrix}
                                                              2k-1\\
                                                              j
                                                             \end{pmatrix}\begin{pmatrix}
                                                              2k\\
                                                              j
                                                             \end{pmatrix}2^j j! \\
                                                        & \times \int_{-\infty}^\infty e^{-t^2/2} H_{2k+1}(t/\sqrt{2}) H_{2k}(t/\sqrt{2})  H_{4k-1-2j}(t/\sqrt{2})dt.
 \end{align*}
\end{lemma}
\begin{proof}
 We simply insert the expressions for $z_3'$ and $z_3''$ in \eqref{zprime} and \eqref{z3ppExpr}, respectively, into \eqref{Dv3}. We then use \eqref{HOrtho} and the fact that the integrand is an even function of $t$ to simplify the expression. 
\end{proof}
The expression \eqref{Dv3Final} for $\frac{\partial^3 D}{\partial v^3}(0,0)$ in \lemmaref{main1} then follows from \eqref{intHnHmHl} and \eqref{prop3}. 

To show \eqref{ckprop} we simply expand the binomial coefficients in the expression for $c_{kj}$ and obtain
\begin{align*}
\frac{c_{k(k-l)}}{\vert c_{k(k+l-1)}\vert} = \frac{(2k+2l) (2k+2l-1)\cdots (2k+4-2l)(2k+3-2l)}{(k+l)^2(k+l-1)^2\cdots (k+3-l)^2(k+2-l)^2},
\end{align*}
where the 
numerator and denominator both consist of $2(2l-1)$ factors. We simplify half of these factors by dividing up
\begin{align*}
 \frac{c_{k(k-l)}}{\vert c_{k(k+l-1)}\vert} = 2^{2l-1}\frac{(2k+2l-1)(2k+2l-3)\cdots (2k+5-2l)(2k+3-2l)}{(k+l)(k+l-1)\cdots (k+3-l)(k+2-l)}.
\end{align*}
We can write the last fraction as a product
\begin{align*}
\left(2-1/(k+l)\right)\left(2-3/(k+l-1)\right)\cdots \left(2-1/(k+2-l)\right),
\end{align*}
where each factor is $>1$ for every $l=1,\ldots,k$. This shows \eqref{ckprop} and
we have therefore completed the proof of \lemmaref{main1}. 
\end{proof}


\section{Secondary canards: a complete picture}\seclab{blowup}
In \figref{blowup1} we present a sketch of the compactified version of \eqref{fnnfeps20} using the Poincar\'e compactification induced by \eqref{blowup}. The diagram is therefore identical to \figref{blowup0}, but with $r=0$ (and therefore $\epsilon=0$). Recall also that the three-dimensional hemisphere $S^3_{\bar \epsilon\ge 0}+ = \{(\bar x,\bar y,\bar z,\bar \epsilon)\in S^3\vert \bar \epsilon\ge 0\}$, is ``flattened out''  by projection onto the $(\bar x,\bar y,\bar z)$-space, so that the sketched two-dimensional-sphere $(\bar x,\bar y,\bar z)\in S^2$ corresponds to the ``equator'' $\bar \epsilon=0$ of $S_{\bar \epsilon\ge 0}^3$. On the other hand, everything inside is $\bar \epsilon>0$. In the following, we let $\sigma$ act on $S^3_{\bar \epsilon\ge 0}$ as follows $\sigma:\,(\bar x,\bar y,\bar z,\bar \epsilon)\mapsto (\bar x,-\bar y,-\bar z,\bar \epsilon)$. This action is consistent with \eqref{sigmafnnf}. The red and blue curves on the equator sphere $\bar \epsilon=0$ correspond to  the intersection with the critical manifold: $\bar x = \bar z^2$, which away from $\bar x=\bar z=0$ has gained hyperbolicity, recall \figref{blowup0}. Applying center manifold theory to these points gives rise to the local center manifolds $W^{cs}(\mu)$ and $W^{cu}(\mu)$ also illustrated as shaded surfaces extending into $\bar \epsilon>0$. (Recall, that these manifolds are (a) the ones obtained by restricting the $3D$ manifolds $M_{r}$ and $M_{a}$ to the sphere $r=0$ and therefore (b) the `extensions' of the critical manifolds $C_r$ and $C_a$, respectively, onto the blowup sphere, recall \figref{blowup0}. ) The manifolds $W^{cs}(\mu)$ and $W^{cu}(\mu)$  contain the strong and weak canards (orange and purple dotted lines, respectively), being heteroclinic orbits, within this framework, connecting partially hyperbolic points $\sigma p_{s}$ and $\sigma p_w$, given by
\begin{align*}
 (\bar x,\bar y,\bar z,\bar \epsilon)=(-1,-2,-1,0),\quad  (\bar x,\bar y,\bar z,\bar \epsilon)=(-1,-2/\mu,-1,0),
\end{align*}
with $p_s$ and $p_w$, respectively, on the equator sphere with $\bar \epsilon=0$. 
Another simple calculation in the `$\bar z=1$' chart shows that the points $q_{\textnormal{out}}$ and $q_{\textnormal{in}}=\sigma q_{\textnormal{out}}$ are hyperbolic attracting and repelling nodes, respectively. They correspond to the intersection of the nonhyperbolic critical fiber of the folded node $p$ (in \figref{foldedC} this fiber coincides with the $z$-axis) with the blowup sphere. On the other hand, by working in the chart `$\bar y=1$' it follows that the points $q_\pm:\,\bar x=\bar z=\bar \epsilon=0, \bar y=\pm 1$ are fully nonhyperbolic. Notice also $q_+=\sigma q_-$.

In the following, we write $a \prec b$ to mean that $a<b$ while $b-a$ is `sufficiently small'. We define $a\succ b$ similarly to mean that $b \prec a$. Finally, $a\sim b$ will mean that $\vert b-a\vert$ is `sufficiently small'.

Recall that for \eqref{fnnfeps20}, $\gamma$ in \eqref{gammahere} is the `weak canard' written in the $\bar \epsilon=1$ chart. This special orbit divides the center manifolds into unique and nonunique subsets. To see this, notice that the local center manifold for $\bar z>0$ and $\bar \epsilon\sim 0$ is unique around the strong canard all up to the weak canard since these points coincide with the stable set of $p_s$, see \figref{blowup1} and \cite[Fig. 9]{wechselberger_existence_2005}. We collect this result -- using the $x$-variables, recall \eqref{rectify} -- as follows:
\begin{lemma}\lemmalab{uniquenessWLOC}
 The local center manifold $W_{loc}^{cs}(\mu)$ is unique on the side $x_2\le 0$ as the stable set of $p_s$ but nonunique for $x_2>0$. Indeed, every point on the nonunique side of $W_{loc}^{cs}(\mu)$ with $\bar \epsilon>0$ is forward asymptotic to the hyperbolic and attracting node $q_{\textnormal{out}}$. 
\end{lemma}
\begin{proof}
Regarding the unique side of $W_{loc}^{cs}(\mu)$, we proceed as follows.
In terms of the coordinates $(x_1,y_1,\epsilon_1)$ specified by the chart `$\bar z=1$', the point $\sigma p_s$ is $(x_1,y_1,\epsilon_1)=(0,2\mu^{-1},0)$ whereas $\sigma p_w$ is $(x_1,y_1,\epsilon_1) = (0,2,0)$. The center manifold $W^{cs}_{loc}$ in \eqref{WCSLOC} is therefore only unique for $y_1\le 2$ which upon coordinate transformation becomes $y_2 \le 2z_2$, seeing that $z_2\gg 1$. The result then follows from the definition of $x_2$ in \eqref{rectify}.

On the other hand, to verify the statement about $x_2>0$, we blowup each $q_\pm$ to a sphere by setting 
\begin{align}
 \bar x = \rho^2 \bar{\bar x},\,\bar z=\rho\bar{\bar z},\,\bar \epsilon = \rho^3 \bar{\bar \epsilon},\eqlab{blowupqpm}
\end{align}
leaving $\bar y$ untouched,
where $\rho\ge 0$, $(\bar{\bar x},\bar{\bar z},\bar{\bar \epsilon})\in S^2$. Only $S^2_{\bar{\bar \epsilon}\ge 0}:= S^2\cap \{\bar{\bar \epsilon}\ge 0\}$ is relevant. Notice that these weights are the same as those used for blowing up the fold in $\mathbb R^3$, see \cite{szmolyan2004a}. The calculations are also essentially identical to those in \cite{szmolyan2004a}, so we skip the details and just present the resulting diagrams, see \figref{blowup2} and \figref{blowup3} for the blowup of $q_-$ and $q_+$, respectively. In these figures, the spheres $S^2_{\bar{\bar \epsilon}\ge 0}$, obtained from the blowup \eqref{blowupqpm}, are shown in green. The consequence of these blowups are then that each point on $W^{cs}_{loc}(\mu)$ with $x_2>0$, $\bar \epsilon>0$, is forward asymptotic to $q_{out}$. Seeing that $q_{out}$ is a hyperbolic and attracting node, this means that $W^{cs}_{loc}(\mu)$ is nonunique on this side of $\gamma$. 
\end{proof}

\noindent Using the symmetry, we obtain a similar result for $W^{cu}$. In particular, every point on the nonunique side of $W_{loc}^{cu}(\mu)$ with $\bar \epsilon>0$ is backwards asymptotic to $q_{\textnormal{in}}$. 
%

\subsection{The transcritical bifurcation}
Now, consider the transcritical bifurcation near any odd integer $n=2k-1$. Then by \thmref{main1} item (\ref{nodd}) and \corref{symmetryLemma}, we have a symmetric secondary canard $\gamma^{sc}(\mu)$ for any $\mu \sim 2k-1$. For $\mu=2k-1$, $\gamma^{sc}(2k-1)=\gamma$. Furthermore
\begin{proposition}\proplab{gammasctranscrit}
The following holds for any $k\in \mathbb N$:
\begin{enumerate}
 \item \label{case1} For any $\mu\prec 2k-1$, $\gamma^{sc}(\mu)$ is backwards asymptotic to $q_{\textnormal{in}}=\sigma q_{\textnormal{out}}$ and forward asymptotic to $q_{\textnormal{out}}$. In this case, $\gamma^{sc}(\mu)$ is nonunique.
 \item \label{case2} For any $\mu\succ 2k-1$, $\gamma^{sc}(\mu)$ is backwards asymptotic to $\sigma p_{s}$ and forward asymptotic to $p_s$. In this case, $\gamma^{sc}(\mu)$ is unique as a heteroclinic connection.
\end{enumerate}
\end{proposition}
\begin{proof}
Firstly, the fact that $\gamma^{sc}(\mu)$ is either (\ref{case1}): backwards asymptotic to $q_{\textnormal{in}}$ and forward asymptotic to $q_{\textnormal{out}}=\sigma q_{\textnormal{in}}$ or (\ref{case2}): backwards asymptotic to $\sigma p_{s}$ and forward asymptotic to $p_s$, is a consequence of $\gamma^{sc}(\mu)$ being symmetric, recall \corref{symmetryLemma}. Similarly, the uniqueness of $\gamma^{sc}(\mu)$ is a consequence of the uniqueness of the center manifolds on one side of $\gamma$ only, see discussion above and \lemmaref{uniquenessWLOC}. 
To complete the proof, suppose that $\mu\succ 2k-1$. ($\mu\prec 2k-1$ is similar and therefore left out.) Therefore $\alpha \succ 0$ and by working with the normal form \eqref{nftranscrit}, recall also \eqref{conjugacy}, we realise that $\gamma^{sc}\subset W^{cs}\cap W^{cu}$ intersects $\Sigma$ along $x_2=0$. Let $(x_1(\mu),0,0)$ denote the intersection point. Then by \eqref{nftranscrit}
\begin{align}
\textnormal{sign}\,x_1 = \textnormal{sign}\,(-1)^{k+1}.\eqlab{signx1}
\end{align} 
We will now show that the $x_2$-component of $\gamma^{sc}(\mu)$ is negative for all $t$ sufficiently large.
For this purpose, consider the first column of the state-transition matrix $\Phi$ in \eqref{Psi2k_1}$_{n=2k-1}$ and multiply this column by the nonzero number $H_{2k-2}(0)$. This gives the following solution
\begin{align}
 \begin{pmatrix}
  H_{2k-2}(t/\sqrt{2})\\
  -\frac{1}{2k-1}H_{2k-1}(t/\sqrt{2})\\
  \frac{\sqrt{2}}{2k-1}H_{2k-1}(t/\sqrt{2})
 \end{pmatrix},\eqlab{solution}
\end{align}
of the variational equations \eqref{veqns} with an initial condition 
\begin{align}
(H_{2k-2}(0),0,0)^T,\eqlab{initialCondition}
\end{align}
along $V$; recall that $U\oplus V$ is $T_{\gamma(0)}W^{cs}(\mu)$ for $\mu=2k-1$. Using \eqref{prop3} we realise that the first component of \eqref{initialCondition} has the same sign as \eqref{signx1}. Fix therefore $T>0$ large enough so that $H_{2k-1}(t)\ge 1$, say, for all $t\ge T$. Such $T$ exists since $H_{2k-1}(t)$ is polynomial with positive coefficient of the leading order term $t^{2k-1}$. Then specifically, the $x_2$-component of \eqref{solution} is negative for all $t\ge T$, and consequently for $\mu\succ 2k-1$, by regular perturbation theory, the $x_2$-component of the time $t\ge T$ forward flow of $\gamma^{sc}\cap \Sigma$ is also negative.  This completes the proof since by \lemmaref{uniquenessWLOC}, $\gamma^{sc}(\mu)$ then belongs to the unique side of $W_{loc}^{cs}(\mu)$ with $x_2<0$ being forward asymptotic to $p_s$. See also \figref{blowup1}.
\end{proof}
\begin{remark}\remlab{secondarycanards1}
Here we recall some basic facts about canards from \cite{brons-krupa-wechselberger2006:mixed-mode-oscil,szmolyan_canards_2001,wechselberger_existence_2005}. Whereas the strong canard always persists as a true (`maximal') canard for any $0<\epsilon \ll 1$, connecting the Fenichel slow manifolds $S_{a,\epsilon}$ and $S_{r,\epsilon}$, the perturbation of the weak canard for $0<\epsilon\ll 1$ to a true (`maximal') canard is clearly more involved. In particular, there is no candidate weak canard on the critical manifold, but rather a funnel of trajectories tangent at $p$ to the weak eigenvector at the folded node. However, seeing that $\gamma^{sc}(\mu)$ on the blowup sphere is asymptotic to $p_s$ and $\sigma p_s$ for fixed $\mu\succ 2k-1$, see \propref{gammasctranscrit} (\ref{case2}), this secondary canard has the same asymptotic properties as $\upsilon$ and it therefore also perturbs into a true  (`maximal') canard connecting the Fenichel slow manifolds $S_{a,\epsilon}$ and $S_{r,\epsilon}$ for $0<\epsilon\ll 1$, see also \cite{wechselberger_existence_2005}. 

In fact, the secondary canards appearing for $\mu\succ 2k-1$ do not undergo additional bifurcations for $\mu>2k-1$. Therefore if $\mu$ satisfies $2k-1<\mu<2k+1$ for some $k$, then there exists $k$ secondary canards for all $0<\epsilon\ll 1$, see \cite[Proposition 4.1]{wechselberger_existence_2005}. These canards divide the Fenichel slow manifold into bands $o(1)$-close (with respect to $\epsilon\rightarrow 0$) to the strong canard with different rotational properties  \cite{brons-krupa-wechselberger2006:mixed-mode-oscil}. These bands provide an explanation for mixed-mode oscillations, see also \cite{desroches2012a}. 
\end{remark}

\subsection{The pitchfork bifurcation}
Next, we consider $n=2k$ and the pitchfork bifurcation. Then by \thmref{main1} item (\ref{neven}) there exists two secondary canards $\gamma^{sc}(\mu)$ and $\sigma\gamma^{sc}(\mu)$ for any $\mu\prec 2k$ (or $\alpha\prec 0$). For $\mu=2k$, $\gamma^{sc}(2k)=\gamma$.
\begin{proposition}\proplab{secondarcanards2k}
 The secondary canards $\gamma^{sc}(\mu)$ and $\sigma\gamma^{sc}(\mu)$ for $\mu\prec 2k$ are nonunique heteroclinic connections. One connects $\sigma p_s$ with $q_{\textnormal{out}}$ while the other one connects $q_{\textnormal{in}}=\sigma q_{\textnormal{out}}$ with $ p_s$. 
\end{proposition}
 \begin{proof}
  Straightforward working from the diagrams in \figref{blowup1}, \figref{blowup2} and \figref{blowup3}. These canards are nonunique since they intersect the nonunique parts of the local center manifolds; recall \lemmaref{uniquenessWLOC} and that $\gamma^{sc}(\mu)$ is not symmetric in this case. 
 \end{proof}
 
 Together \propref{gammasctranscrit} and \propref{secondarcanards2k} provide a rigorous and geometric explanation of \cite[Fig. 17]{wechselberger_existence_2005}. In this figure, `$\rho=v$' and the `TPB's are  points beyond which $\gamma^{sc}(\mu)$ does not reach the fixed local version of $W^{cs}(\mu)$, see \eqref{WCSLOC}.

 \begin{remark}\remlab{remsecondarcanards2k}
 As in \remref{secondarycanards1}, we will now describe the implications of \propref{secondarcanards2k} for $0<\epsilon \ll 1$. Fix $\mu\prec 2k$ and suppose without loss of generality that $\gamma^{sc}(\mu)$ is the connection from $\sigma p_s$ to $q_{\textnormal{out}}$. For all $0<\epsilon\ll 1$, seeing that $W^{cs}(\mu)$ and $W^{cu}(\mu)$ are transverse along $\gamma^{sc}(\mu)$, this secondary canard produces, as for the transcritical bifurcation above, a connection between the extended manifolds $S_{a,\sqrt \epsilon}$ and $S_{r,\sqrt \epsilon}$. But since $\gamma^{sc}(\mu)$ for $\epsilon=0$ is asymptotic to $q_{\textnormal{out}}$ in forward time, the perturbed `canard' never reaches the Fenichel slow manifold $S_{r,\epsilon}$. Instead it follows, upon blowing down, the nonhyperbolic critical fiber as $\epsilon\rightarrow 0$. However, since $\gamma^{sc}(\mu)$ is close to the strong canard for all $t$ sufficiently negative, we can flow the perturbed version backwards and conclude that it does in fact originate from the Fenichel slow manifold $S_{a,\epsilon}$. Here it also divides the subset of $S_{a,\epsilon}$ between the secondary canard due to the bifurcation at $2k-1$ and the rest of the funnel into regions of separate rotational properties through the folded node, see also \cite[Proposition 2.5]{wechselberger_existence_2005}.
 \end{remark}


\begin{figure}[h!]
\begin{center}
{\includegraphics[width=.795\textwidth]{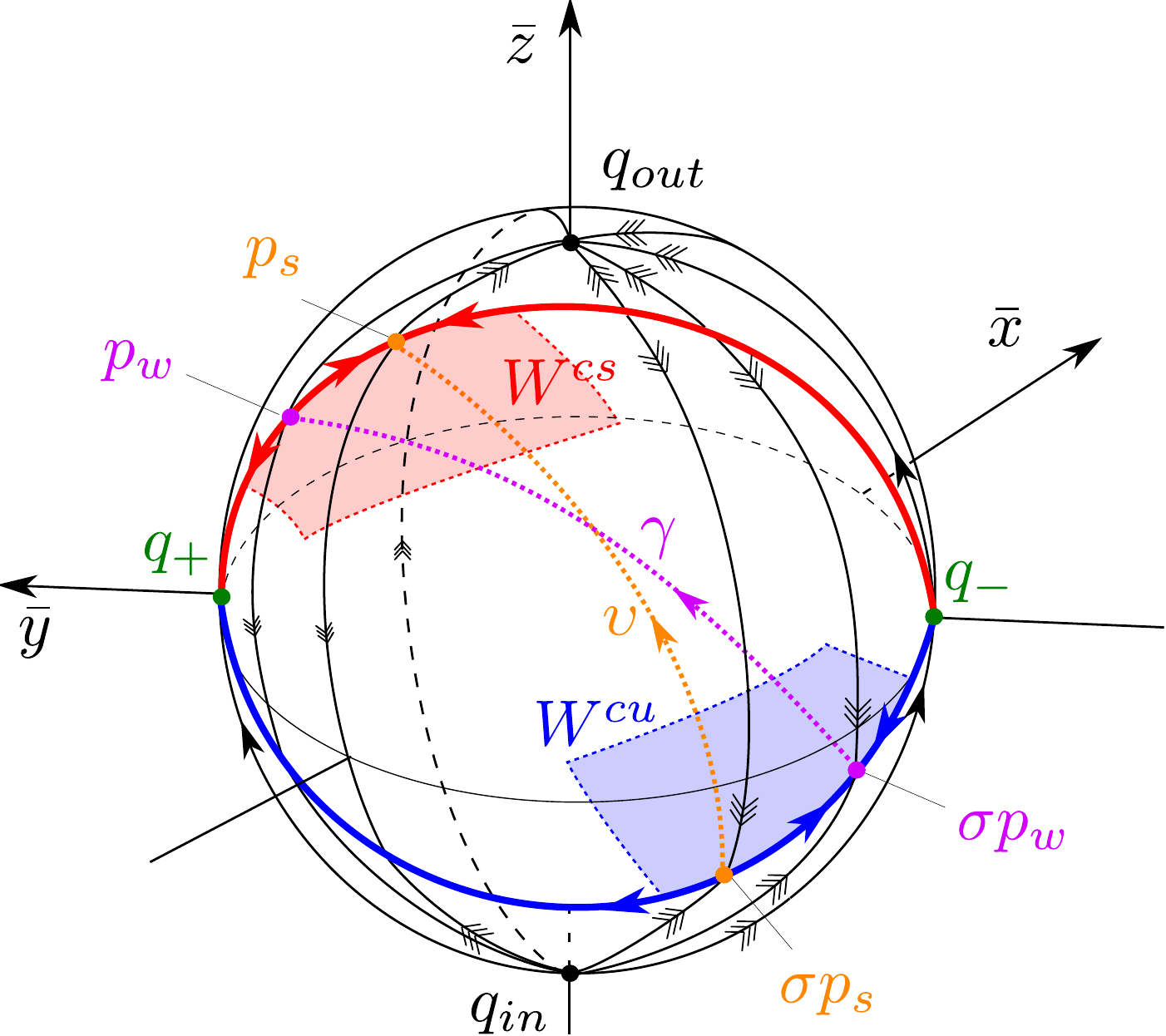}}
\end{center}
\caption{The global dynamics on the sphere $S^3$, representing $S^3_{\bar \epsilon\ge 0}:=S^3\cap \{\bar \epsilon\ge 0\}$ -- by projection -- as a solid `ball' in $(\bar x,\bar y,\bar z)$-space. Here the unit sphere, being the boundary of the ball, corresponds to $\bar \epsilon=0$, whereas everything inside of the ball corresponds to $\bar \epsilon>0$. Within this framework, the strong and weak canard, $\upsilon$ and $\gamma$, respectively, are symmetric heteroclinic connections of points on the sphere. These orbits belong to $\bar \epsilon>0$, i.e. inside the sphere, and are therefore indicated in orange and purple, recall also \figref{foldedC}, using dotted lines. Indicated are also the invariant manifolds $W^{cu}(\mu)$ and $W^{cs}(\mu)$ (suppressing the $\mu$-dependency in the figure), which are locally center manifolds of normally hyperbolic lines of equilibria (blue and red half-circles, respectively). These lines end in nonhyperbolic points, $q_-$ and $q_+$ in green which correspond to the intersection of the fold line $F$, see \figref{foldedC}, with the sphere obtained by blowing up the folded node $p$. The manifolds $W^{cu}(\mu)$ and $W^{cs}(\mu)$ intersect along $\gamma$, doing so tangentially for any $\mu\in \mathbb N$. This `bifurcation' produces secondary canards through transcritical and pitchfork bifurcations, see \thmref{main1}.  }
\figlab{blowup1}
\end{figure}

\begin{figure}[h!]
\begin{center}
{\includegraphics[width=.795\textwidth]{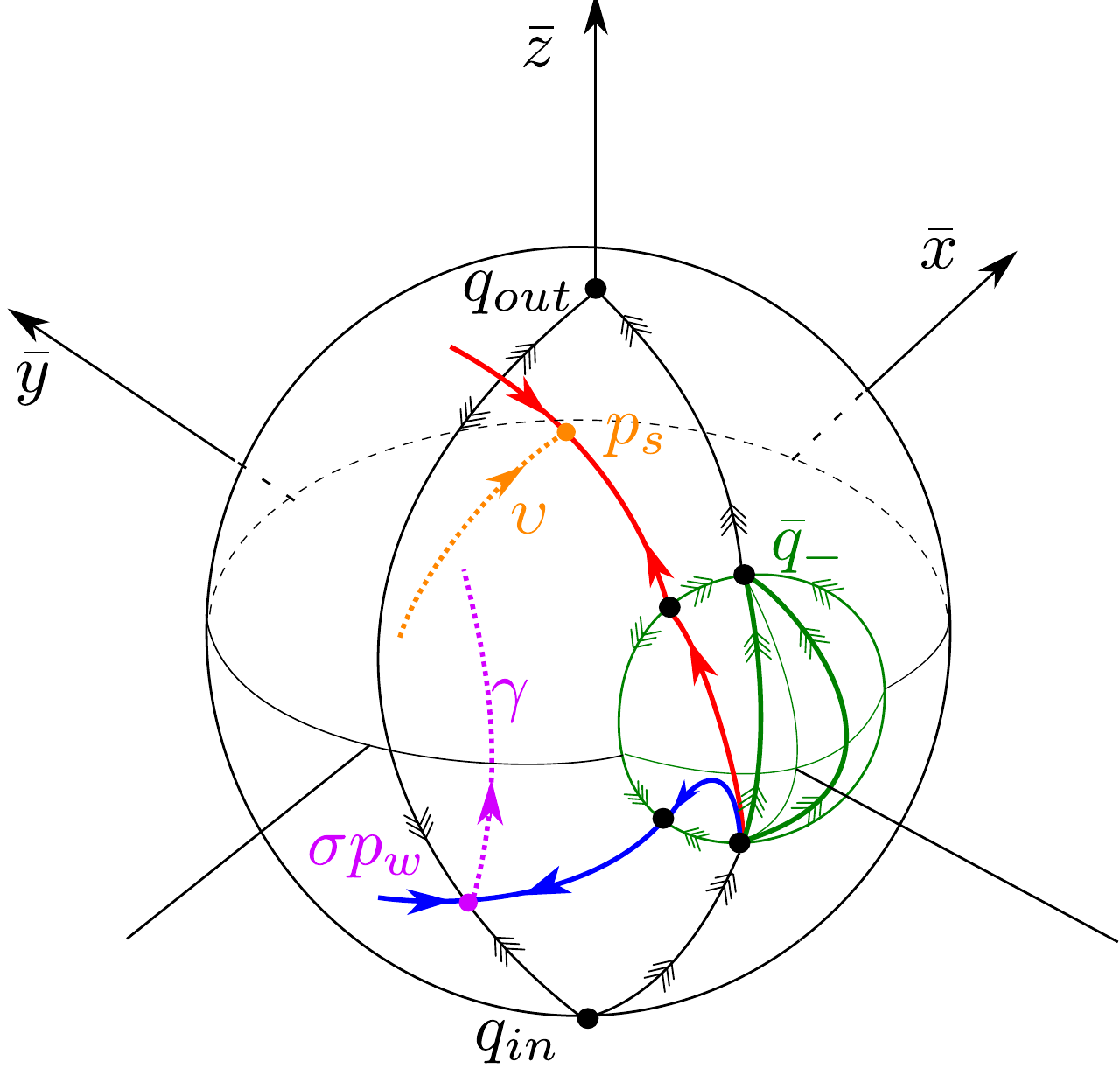}}
\end{center}
\caption{Illustration of the blowup of $q_-$, using the same viewpoint as in \figref{blowup1}. This blowup allows us to conclude that every point on the local manifold $W^{cu}$, close to the line of equilibria (in blue) and between $p_w$ and $q_-$, will be backwards asymptotic to $q_{\textnormal{in}}$.}
\figlab{blowup2}
\end{figure}

\begin{figure}[h!]
\begin{center}
{\includegraphics[width=.795\textwidth]{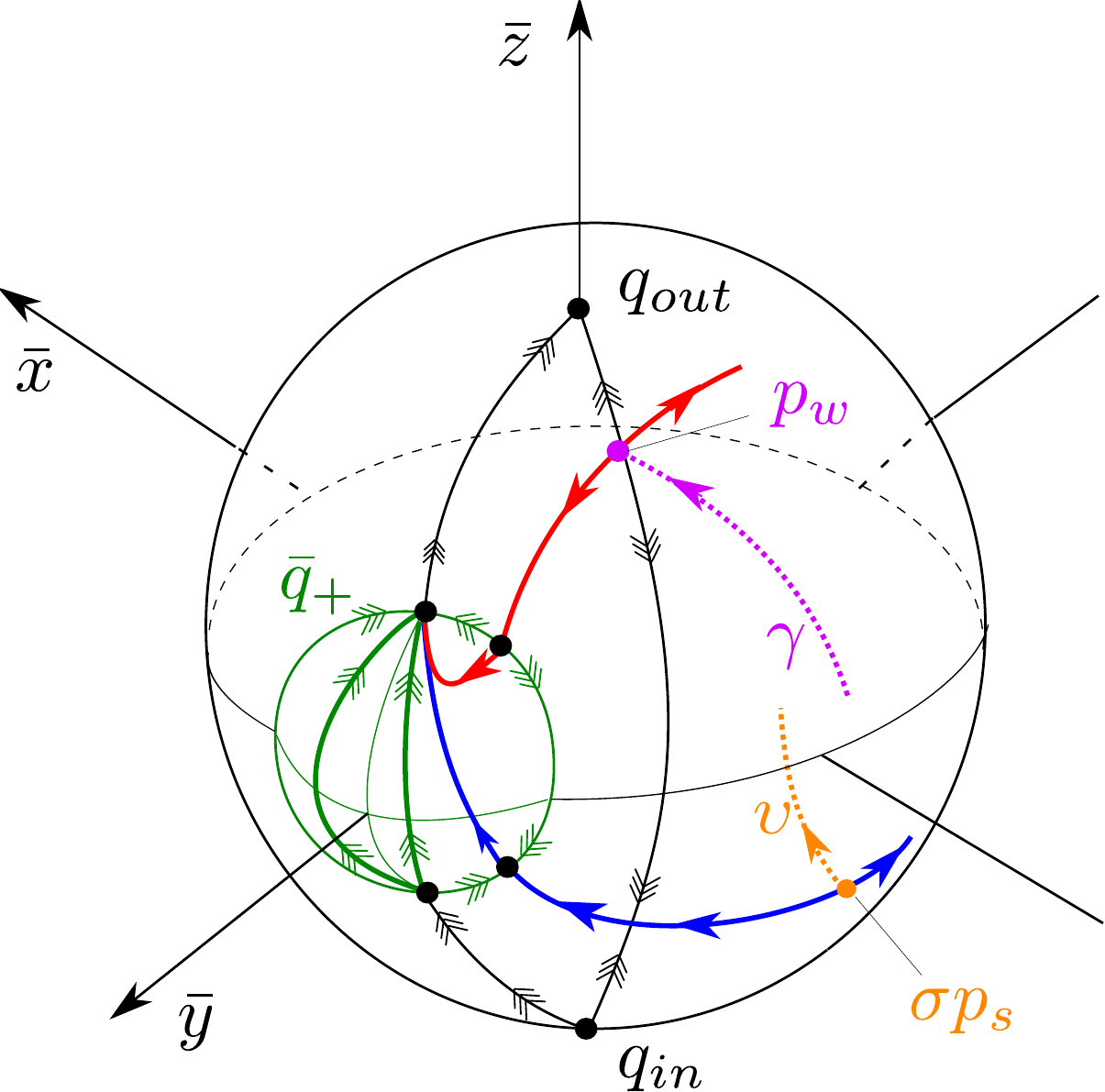}}
\end{center}
\caption{Illustration of the blowup of $q_+$, using the same viewpoint as in \figref{blowup1}, except now the positive $y$-axis is coming out of the diagram. This blowup allows us to conclude that every point on the local manifold $W^{cs}(\mu)$ close to the line of equilibria (in red) and between $p_w$ and $q_+$, will be forward asymptotic to $q_{\textnormal{out}}$.}
\figlab{blowup3}
\end{figure}




\section{Other examples of time-reversible systems}\seclab{more}
In this section, we present a few other examples: The two-fold, the Falkner-Skan equation and finally the Nos\'e equations, of nonhyperbolic connection problems where our time-reversible version of the Melnikov theory in \cite{wechselberger2002a} can be applied to study bifurcations. Whereas our analysis of the two-fold is brief -- postponing all of the details to future work -- we do include complete, self-contained descriptions of the bifurcations of periodic orbits in the Falkner-Skan equation and the Nos\'e equations.
\subsection{The two-fold: Bifurcations of `canards'} Singularly perturbed systems  in $\mathbb R^3$  that limit to the piecewise smooth two-fold singularity, see \cite{jeffrey_two-fold_2009}, also possess orbits that are reminiscent of weak and strong canards in the singular limit $\epsilon\rightarrow 0$. In particular, upon blowup, the two-fold $p$ corresponds to a `$0/0$'-singularity of a reduced problem on a critical manifold, having attracting and repelling parts on either side of $p$. Upon desingularization, in much the same way as in \eqref{yzDS}, $p$ becomes a stable node with eigenvalues $\lambda_s<\lambda_w<0$ for a subset of parameters. The essential geometry is shown in \figref{VISliding}, see further description in the figure caption. Notice the geometry is essentially different from the folded node, insofar that the attracting and repelling manifolds for the two-fold only `meet up' at the point $p$, whereas for the folded node they align along the fold line. 
Nevertheless, upon further blowup, \cite{krihog} showed, see also \cite{kristiansen2018a}, by working on a `normal form', that center manifold extensions of slow-like manifolds for the two-fold also intersect along a `strong canard' for all $0<\epsilon\ll 1$, as well as along a `weak' one provided that $\mu = \lambda_s/\lambda_w\notin \mathbb N$. The equations in the associated scaling chart for $\epsilon=0$ take the following form
\begin{equation}\eqlab{twofold}
\begin{aligned}
\dot x &=\beta^{-1} c(1+\phi(y))-(1-\phi(y)),\\
\dot{y} &=b {z}(1+\phi(y)) -\beta x(1-\phi( y)),\\
 \dot z&=1+\phi(y)+b^{-1} \tilde \gamma (1-\phi( y)).
\end{aligned}
\end{equation}
see \cite[Eq. (93)]{kristiansen2018a}. 
Here $\phi$ is any `regularization function' satisfying $\phi'>1$ and $\phi(y)\rightarrow \pm 1$ as $y\rightarrow \pm \infty$. Associated with the strong and weak eigenvectors of $p$, the system \eqref{twofold} has, under certain assumptions on the parameters of the system ($b,c,\tilde \gamma,\beta$ and $\phi$), two algebraic solutions $\upsilon$ and $\gamma$, respectively. These solutions each lie within sets of the form $\{y=\text{const}.\}$ and their projections onto the $(x,z)$-plane coincides with the strong and weak eigenspace, see further details in \cite{krihog,kristiansen2018a}. Moreover, $\upsilon$ and $\gamma$ are fixed by the time-reversible symmetry of \eqref{twofold} given by $\sigma=\textnormal{diag}\,(-1,1,-1)$ and correspond to `unbounded heteroclinic connections' upon the compactification provided by the blowup for $\epsilon=0$. Moreover, even though \eqref{twofold} does not fit our general framework in \secref{recipe}, the variational equations along $\gamma$ can also be reduced to the Weber equation, see \cite[Eq. (101)]{kristiansen2018a}. In particular, for each $\mu\in \mathbb N$, this equation has an algebraic solution resulting in a bifurcation scenario similar to folded node, where `secondary canards' (may) emerge. I expect that the details are very similar to the folded node above, see also the numerical exploration in \cite[Section 8]{krihog}. However, I also expect it to be slightly more involved due to the many parameters of the system. (In fact, it might be more advantageous to work with a different, further simplified, normal form, e.g. one derived from the `normal form' in \cite{jeffrey_two-fold_2009} based on the associated piecewise smooth Filippov system). I have therefore decided not pursue this problem further in the present manuscript. 

\begin{figure}[h!] 
\begin{center}
{\includegraphics[width=.95\textwidth]{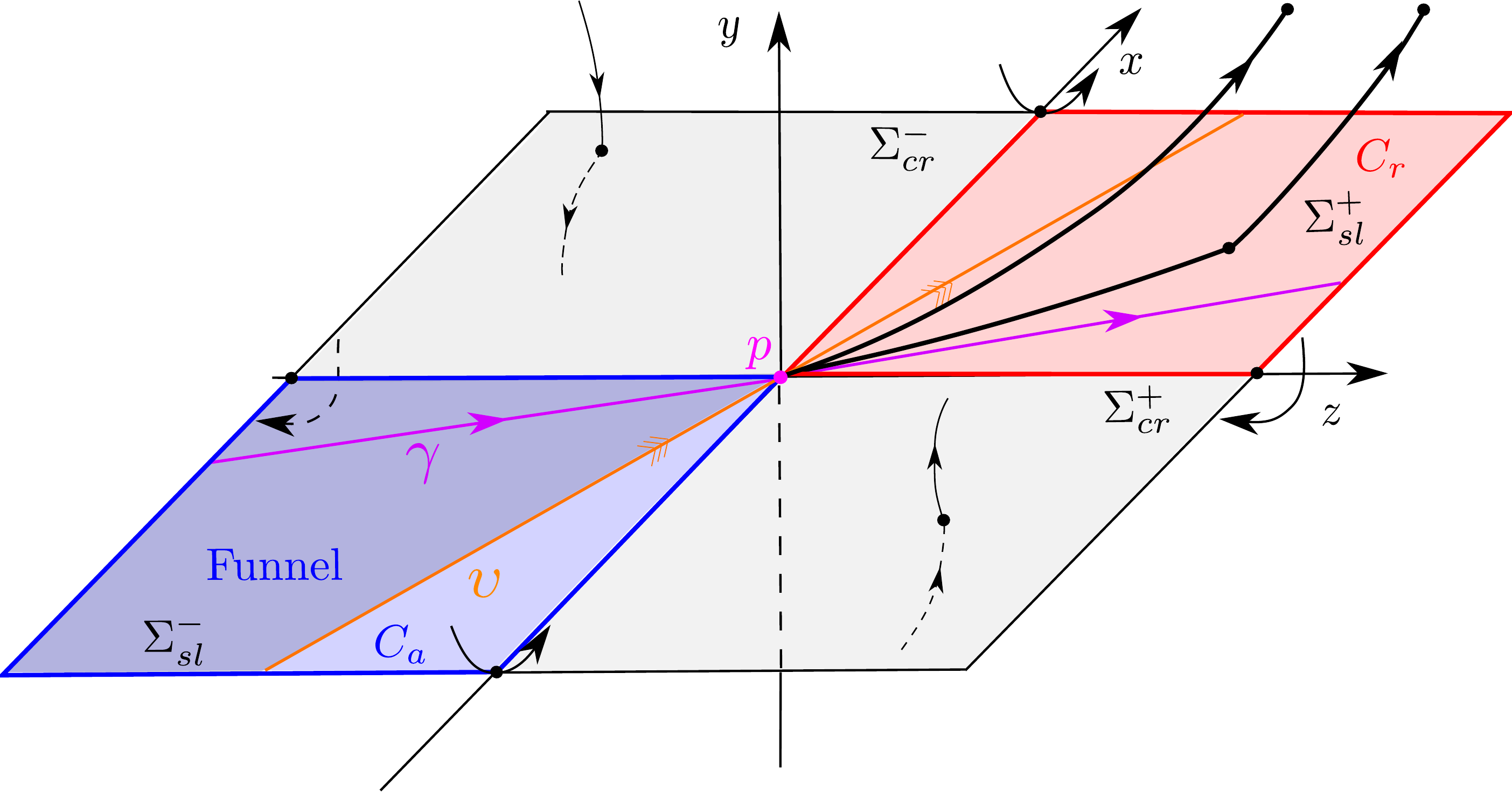}}
\end{center}
 \caption{ {Local geometry of the visible-invisible two-fold at $p=(0,0,0)$ (pink)}.  As a piecewise smooth system, the two-fold is the intersection of two fold lines on either side of a discontinuity set. In this (generic) `normal form' picture, the discontinuity set is $y=0$, while the $x$-axis is a visible fold line for the system defined for $y>0$ whereas the $y$-axis is an invisible fold-line for the system below $y<0$. The fold lines divide a neighborhood of $p$ on the discontinuity set into four quadrants: crossing downwards $\Sigma_{cr}^-$, crossing upwards $\Sigma_{cr}^+$, stable sliding $\Sigma_{sl}^-$ and unstable sliding $\Sigma_{sl}^+$. See \cite{jeffrey_two-fold_2009,Bernardo08} for further background on these piecewise smooth concepts. On the other hand, as a singular perturbed system, the system has, upon blowup of $y=\epsilon=0$, $\Sigma_{sl}^-$ as an attracting critical manifold $C_a$ (blue) and $\Sigma_{sl}^+$ as a repelling one $C_r$ (red). The point $p$ is a degenerate point (fully nonhyperbolic half-circle for the blowup system). However, the reduced problem on $C=C_a\cup C_r$ has a `$0/0$'-type of singularity where orbits, like canards, can pass from the attracting sheet to the repelling one. In fact, as for the folded node, one can apply desingularization so that $p$ becomes a stable node with eigenvalues $\lambda_s<\lambda_w<0$ for the reduced problem on $C_a$. The two orbits shown $\upsilon$ (orange) and $\gamma$ (purple) are `strong' and `weak' canards. Furthermore, bifurcations of $\gamma$ occur whenever $\mu=\lambda_s/\lambda_w\in \mathbb N$. }
\figlab{VISliding}
\end{figure}

\subsection{The Falkner-Skan equation: Bifurcations of unbounded periodic orbits}\seclab{FSsec} 
In \cite{swinnertondyer1995a} it was shown for the Falkner-Skan equation \eqref{falkerskan0} that periodic orbits bifurcate from each integer value of $\mu\in \mathbb N$. As noted in \cite{swinnerton-dyer2008a}, the proof is long, complicated and -- to a large extend -- not based upon dynamical systems theory. The aim of the following section, is therefore to give a simple proof using the Melnikov approach, in particular \thmref{Melnikov}, and the recipe in \secref{recipe}, which is based upon -- as is more standard in dynamical systems -- invariant manifolds. See also \cite{llibre2007a}, for a similar approach in this context. In this reference, however, periodic orbits are constructed through an analysis of a return mapping.

First we write the equation \eqref{falkerskan0} as a first order system
\begin{equation}\eqlab{falkerskan}
\begin{aligned}
 \dot x &=y,\\
 \dot y&=z,\\
 \dot z &=-xz-\mu(1-y^2),
\end{aligned}
\end{equation}
which possesses two special solutions:
\begin{align*}
 \gamma:\,(x,y,z)&= (- t,- 1,0),\\
 \upsilon:\,(x,y,z) &= (t,1,0)
\end{align*}
and a time-reversible symmetry given by
 \begin{align*}
  \sigma=\textnormal{diag}\,(-1,1,-1).
 \end{align*}
 Both $\gamma$ and $\upsilon$ are symmetric orbits. It is easy to see, upon rectifying $\gamma$ to the $x_3$-axis, setting $(x,y,z)=(-x_3,1+x_1,x_2)$, that \eqref{falkerskan} satisfies the assumptions in \secref{recipe}, recall (H5)-(H7), respectively. In particular, 
 \begin{align*}
  \beta = 2\mu-1.
 \end{align*}
 \thmref{recipe} therefore applies and we can evaluate the relevant integrals at bifurcations in closed form following the recipe outlined in \remref{procedure}.
To describe the global dynamics relevant for the bifurcations of periodic orbits, and obtain the invariant manifolds $W^{cs}$ and $W^{cu}$, we will compactify the system. 
For our purposes I find it useful to just compactify $(x,z)$, leaving $y$ untouched, by setting 
 \begin{align}\eqlab{barW1New}
 \begin{split}
  x &= \frac{\bar x}{\bar w},\\
  z &= \frac{\bar z}{\bar w},
  \end{split}
 \end{align}
for $(\bar x,\bar z,\bar w)\in S^2$. To describe the dynamics near the equator defined by $\bar w=0$, we consider the directional chart `$\bar x=-1$' defined by
\begin{align}\eqlab{barXNeg1New}
\begin{split}
 z_1&:=-\frac{\bar z}{\bar x},\\
 w_1&:=-\frac{\bar w}{\bar x}.
 \end{split}
\end{align}
The smooth change of coordinates between the `$\bar w=1$' chart, defined in \eqref{barW1New}, and the `$\bar x=-1$ chart, given by \eqref{barXNeg1New}, is determined by the following equations
\begin{align}\eqlab{coordinateNew}
\begin{split}
 w_1 &= -x^{-1},\\
 z_1 &= -z x^{-1},
 \end{split}
\end{align}
for $x<0$.
Using \eqref{coordinateNew} we obtain the following equations in the `$\bar x=-1$' chart:
\begin{align}\eqlab{barXNeg1NewEqns}
\begin{split}
\dot y &= z_1,\\
\dot z_1&= z_1+w_1^2 (yz_1-\mu(1-y^2)),\\
\dot w_1&=w_1^3 y,
\end{split}
\end{align}
upon multiplication of the right hand sides by $w_1$, to ensure that $w_1=0$ -- corresponding to $\bar w=0$ under the coordinate map -- is invariant. For this system, we notice that any point $(y,0,0)$ is a partially hyperbolic equilibrium of \eqref{barXNeg1NewEqns}, the linearization having $\lambda=1$ as a single nonzero eigenvalue. Therefore, by center manifold theory there exists a local repelling center manifold $W_{loc}^{cs}(\mu)$. A simple calculation, using the invariance of $\gamma$ and $\upsilon$, shows that it takes the following form 
\begin{align}
 W_{loc}^{cs}(\mu):\, z_1 = (1-y^2) w_1^2\left(\mu +w_1m_1(y,w_1)\right),\quad y\in I,\,w_1 \in [0,\delta],\eqlab{WCSLOCNew}
\end{align}
with $I$ a fixed sufficiently large interval and where $\delta>0$ is sufficiently small. Also, $m_1$ is a smooth function, also depending on $\mu$. In terms of $(x,y,z)$, $W_{loc}^{cs}(\mu)$ takes the following form
\begin{align*}
 W_{loc}^{cs}(\mu):\, z = -(1-y^2) x^{-1}\left(\mu -x^{-1} m_1(y,-x^{-1}))\right) ,\quad y\in I,\,
\end{align*}
valid for all $x$ sufficiently negative. 

Inserting \eqref{WCSLOCNew} into \eqref{barXNeg1NewEqns} gives the reduced problem on $W^{cs}_{loc}$:
\begin{align}\eqlab{yw1Eqns}
\begin{split}
 \dot y &= (1-y^2) \left(\mu +w_1m_1(y,w_1)\right),\\
 \dot w_1 &= w_1 y,
 \end{split}
\end{align}
upon desingularization through division of the right hand side by $w_1^2$. 
Notice that $\dot y>0$ for $y\in (-1,1)$ and all $w_1\ge 0$ sufficiently small. In particular, $(y,w_1)=(-1,0)$ and $(y,w_1)=(1,0)$ are saddles, with the orbit $y\in (-1,1)$, $w_1=0$ being a heteroclinic connection under the flow of \eqref{yw1Eqns}. For later reference, let $L$ be the invariant set defined by 
\begin{align}
z_1=w_1=0, \quad y\in [-1,1].\eqlab{LFS} 
\end{align} 
It becomes $(\bar x,\bar z,\bar w)=(-1,0,0)$, $y\in I$ on the cylinder.  

By applying the symmetry, we obtain a local manifold $W^{cu}_{loc}$ for all $x$ sufficiently large. Combining this information we obtain the diagram in \figref{FS}, see \cite[Fig. 1]{sparrow2002a} for a related figure.

The global manifolds $W^{cs}(\mu)$ and $W^{cu}(\mu)$ intersect along $\gamma$ and $\upsilon$. In particular, along $\gamma$ we have the following
 \begin{lemma}\lemmalab{FSlemma}
  The manifold $W^{cs}(\mu)$ and $W^{cu}(\mu)$ intersect transversally along $\gamma$ if and only if $2\mu \notin \mathbb N$. 
 \end{lemma}
 \begin{proof}
 We use \lemmaref{eqvH4} item (\ref{3iv}). Consider therefore the variational equations about $\gamma$:
\begin{align}\eqlab{varsFS}
\begin{split}
 \dot z_1 &= z_2,\\
 \dot z_2 &=z_3,\\
 \dot z_3 &=tz_3-2\mu z_2,
 \end{split}
\end{align}
which upon eliminating $z_1$ and $z_2$, can be written as a Weber equation
\begin{align}
L_{2\mu-1}z_3 = 0,\eqlab{weberFS}
\end{align}
recall \eqref{Lmexpr}. The result then follows from \lemmaref{LmHmHl}, see also proof of \thmref{recipe}.
In particular, for $n=2\mu$, we obtain the following algebraic solution of \eqref{varsFS}:
\begin{align}
 z = \begin{pmatrix}
      \frac{1}{2n(n+1)}H_{n+1}(t/\sqrt{2})\\
      \frac{1}{\sqrt{2} n}H_n(t/\sqrt{2})\\
      H_{n-1}(t/\sqrt{2})
     \end{pmatrix}.\eqlab{zpFShere}
\end{align}
%
%
%
%
 \end{proof}

 Next, fix any $n\in \mathbb N$ and define $\alpha$ by $$\mu = n/2+\alpha. $$ Then we can define a local Melnikov function $D(v,\alpha)$, the roots of which correspond to intersections of $W^{cs}(\mu)$ and $W^{cu}(\mu)$ near $\gamma$. Using \thmref{Melnikov}, and proceeding as in the proof of \thmref{main1}, we obtain the following.
 \begin{proposition}\proplab{FS}
 Let $k\in \mathbb N$ be so that
\begin{align*}
 n = \begin{cases}
      2k-1 & n=\textnormal{odd}\\
      2k & n=\textnormal{even}
     \end{cases}.
\end{align*}
Then 
\begin{enumerate}
 \item \label{nodd2} For $n=\textnormal{odd}$, $D(v,\alpha)=0$ is locally equivalent with a pitchfork bifurcation:
 \begin{align}
  \tilde v (\tilde \alpha +\tilde v^2) = 0.
 \end{align}
\item \label{neven2} For $n=\textnormal{even}$, $D(v,\alpha)=0$ is locally equivalent with the transcritical bifurcation:
\begin{align}
    \tilde v(\tilde \alpha+(-1)^{k+1}\tilde v)=0.\eqlab{transcritical2}
 \end{align}
\end{enumerate}
In each case, the local conjugacy $\phi:(v,\alpha)\mapsto (\tilde v,\tilde \alpha)$ satisfies $\phi(0,0)=(0,0)$ and
\begin{align}
 D\phi(0,0) = \textnormal{diag}\,(d_1(n),d_2(n)) \quad \mbox{with $d_i(n)>0$ for every $n$.}\nonumber
\end{align}
 \end{proposition}
\begin{proof}
See \appref{FS}.
 \end{proof}

  The bifurcations of $W^{cs}(\mu)$ and $W^{cu}(\mu)$, described in the previous result, produce new transverse intersection for $\mu \sim n/2$ and every $n\in \mathbb N$. Notice, however, that they will not always produce periodic orbits. Instead they may simply diverge (by following the invariant green curves in \figref{FS} defined by $(\bar x,\bar z)=0$). To give rise to periodic orbits, the intersections have to be symmetric (which rules out the pitchfork bifurcation, whose symmetrically related solutions diverge either as $t\rightarrow \pm \infty$ along the green curves in \figref{FS}) and they have to be on the `right' side so that they follow $L$ and $\upsilon$. 
  This is qualitatively very similar to the bifurcation of canards, recall \propref{gammasctranscrit} and \propref{secondarcanards2k}. However, whereas for canards, the `interesting' orbits, the true canards, appeared on unique sides of the invariant manifolds, we will see that for the Falkner-Skan equation and the Nos\'e equations, described below, the `interesting' periodic orbits now appear on the nonunique sides of the manifolds. To obtain closed orbits we will have to fix copies of the center manifolds. We do so by using the fix-point sets of the symmetries.   

In fact, we obtain a new proof of the following result on the bifurcation of periodic orbits from infinity, see \cite{swinnertondyer1995a}, now based upon bifurcation theory and invariant manifolds. 
\begin{theorem}\thmlab{pp}
Let $\Gamma$ be the (singular) heteroclinic cycle obtained from concatenating (a) $\gamma$, (b) the `segment' $L:\,(\bar x,\bar z,\bar w)=(-1,0,0)$, $y\in I$, recall \eqref{LFS} in the `$\bar x=-1$' chart, (c) $\upsilon$, and finally (d) $\sigma L$, i.e. the symmetrically related version of the segment $L$ defined in (b). Then symmetric periodic orbits bifurcate from $\Gamma$ for each $\mu\in \mathbb N$. 

In further, details let $\mu=k$ (so that $n=2k$). Then symmetric periodic orbits only exist (`locally' to $\Gamma$) for $\mu \succ k$.
\end{theorem}
\begin{proof}
The manifolds $W^{cs}$ and $W^{cu}=\sigma W^{cs}$ are (again) nonunique. We select unique copies as follows: Consider the strip $I$ defined by $(0,y,0)$ with $y\sim 1$. Notice that $\sigma I=I$. We then select a unique copy of $W^{cs}(\mu)$ on the $y\ge -1$ side of $\gamma$ by flowing this strip backwards (where $W^{cs}(\mu)$ becomes attracting). Obviously, we let $W^{cu}(\mu)$ be the symmetrically related version of this fixed manifold.

Now, the transcritical bifurcation \eqref{transcritical2} produce a secondary intersection $\gamma^{sc}(\mu)$ of $W^{cs}(\mu)$ and $W^{cu}(\mu)=\sigma W^{cs}(\mu)$ for all $\mu \sim k$ so that $\gamma^{sc}(k)=\gamma$. In particular, we first note -- following \eqref{VFSspace} in \appref{FS} -- that $\gamma^{sc}(\mu)$ intersects $\Sigma$ along the $y$-axis. Let $(0,y_0(\mu),0)$ denote this intersection point where $y_0(\mu)\sim -1$. Consider $\mu \succ k$ so that $\alpha\succ 0$. Then by \eqref{transcritical2} we have
\begin{align}
 \text{sign}(y_0(\mu)+1) = \text{sign}(-1)^k.\eqlab{signy}
\end{align}
Consider now the solution \eqref{zpFShere}$_{n=2k}$ of the variational equations \eqref{varsFS}, repeated here for convenience
\begin{align}
 z = \begin{pmatrix}
      \frac{1}{4k(2k+1)}H_{2k+1}(t/\sqrt{2})\\
      \frac{1}{2\sqrt{2} k}H_{2k}(t/\sqrt{2})\\
      H_{2k-1}(t/\sqrt{2})
     \end{pmatrix},\eqlab{zpFShere2}
\end{align}
with initial condition
\begin{align}
  \left(0,\frac{1}{2\sqrt{2} k} H_{2k}(0),0\right)^T.\eqlab{zpFShere20}
\end{align}
By \eqref{prop3} in \appref{Hermite}, we realise that the sign of the second component of \eqref{zpFShere2} coincides with the sign of \eqref{signy}. But then, since the second component of \eqref{zpFShere2} is positive for all sufficiently large $t$, we conclude that $\gamma^{sc}(\mu)$ for $\mu \succ k$ follows $L$ for $t$ large enough. 
 Subsequently, by following $\upsilon$, see \figref{FS}, we realise that $\gamma^{sc}(\mu)$ returns to $x=0$. Since we have fixed the manifolds to intersect $x=0$ in the strip $I$, this intersection is of the form $(0,y_1(\mu),0)$ with $y_1(\mu)\rightarrow 1$ as $\mu\rightarrow k^+$. But then upon applying the time-reversible symmetry, we obtain a closed orbit. The periodic orbit intersects the fix-point set of $\sigma$, defined by $(0,y,0)$ twice, once near $y\sim -1$ at $(0,y_0(\mu),0)$ and once near $y\sim 1$ at $(0,y_1(\mu),0)$. 
 
 In the remaining cases ($\mu\prec k$ and the $n=\textnormal{odd}$), the `secondary intersection' diverge as either $t\rightarrow \pm \infty$ by following the green curves in \figref{FS} defined by $(\bar x,\bar z)=(0,\pm 1)$.
\end{proof}
\begin{remark}
 Notice that the periodic orbits appearing for $\mu\succ k$ will rotate (or twist) around $\gamma$, the number of rotations, as for the folded node, being determined by $k$. See \figref{FSpo} for examples, and the figure caption for further description. As for the folded node, and the twisting of the secondary canards around the weak one, these rotations around $\gamma$ can be explained by the solutions \eqref{zpFS} of variational equations, see \cite{wechselberger_existence_2005} for further details. 
 \end{remark}

 
 \begin{figure}[h!]
\begin{center}
{\includegraphics[width=.795\textwidth]{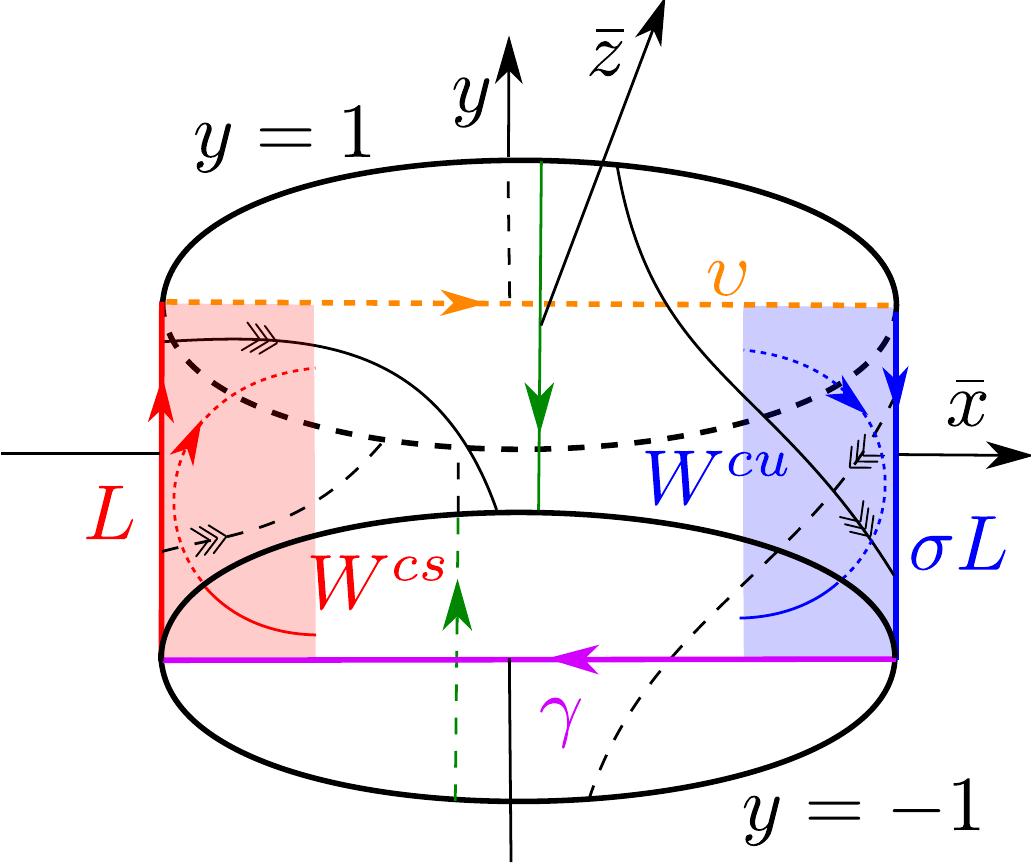}}
\end{center}
\caption{Illustration of the compactification of \eqref{falkerskan}. Our viewpoint is from the negative $y$-axis, seeing the disc at $y=-1$, containing $\gamma$ (purple), from below. Notice that the circles at $y=\pm 1$ are not invariant; they are just emphasized for illustrative purposes. We find invariant manifolds $W^{cs}(\mu)$ (red) and $W^{cu}(\mu)=\sigma W^{cs}(\mu)$ (blue) by application of center manifold theory to the partially hyperbolic lines $L$ and $\sigma L$. Together with $\gamma$ and $\upsilon$ (orange and dotted since it is on the disc at $y=1$), these lines produce to a (singular) cycle. We obtain a new proof of the bifurcation of periodic orbits by using our time-reversible version of the Melnikov theory to perturb away from this cycle. }
\figlab{FS}
\end{figure}

 \begin{figure}[h!]
\begin{center}
\subfigure[]{\includegraphics[width=.495\textwidth]{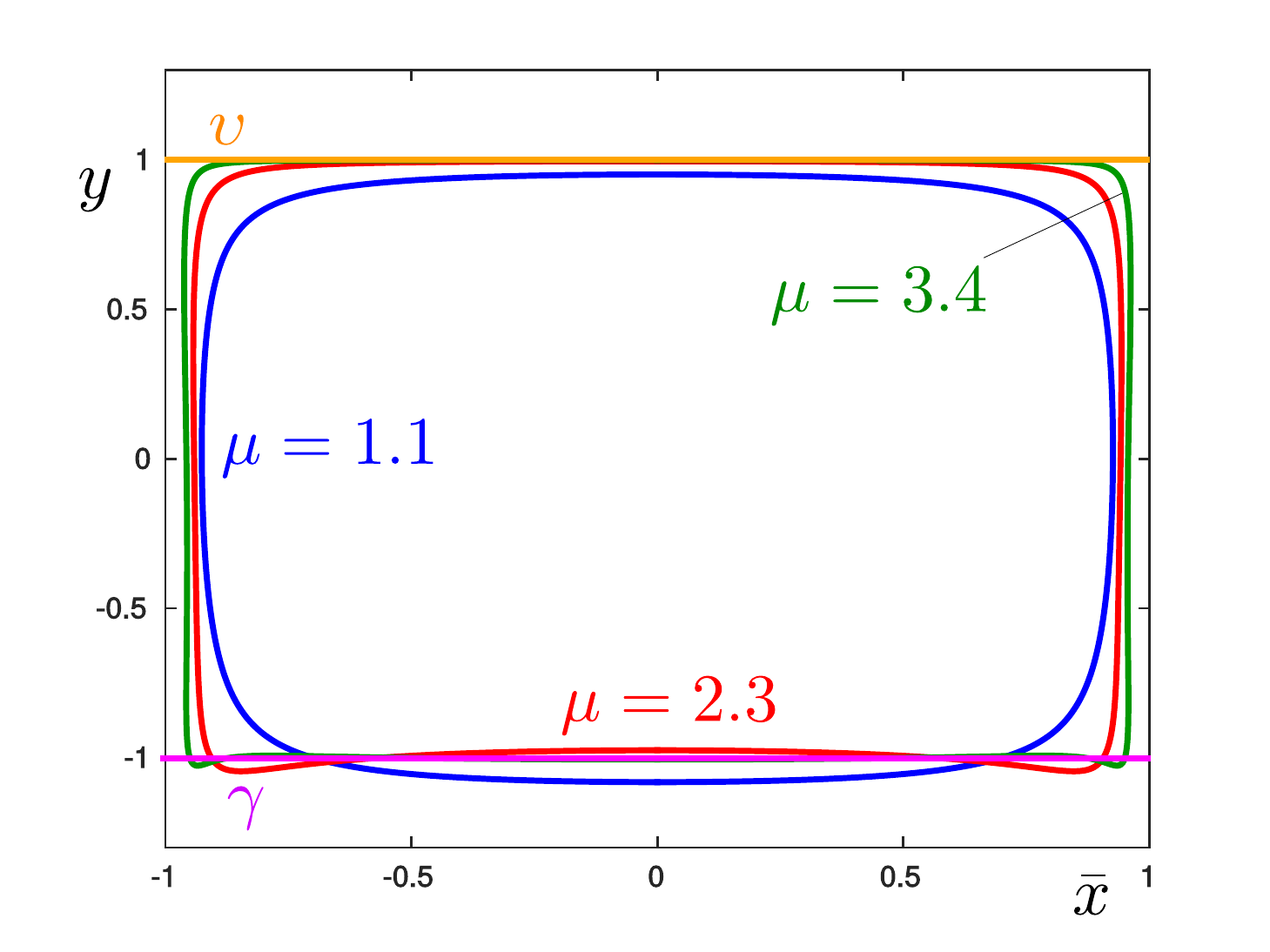}}
\subfigure[]{\includegraphics[width=.495\textwidth]{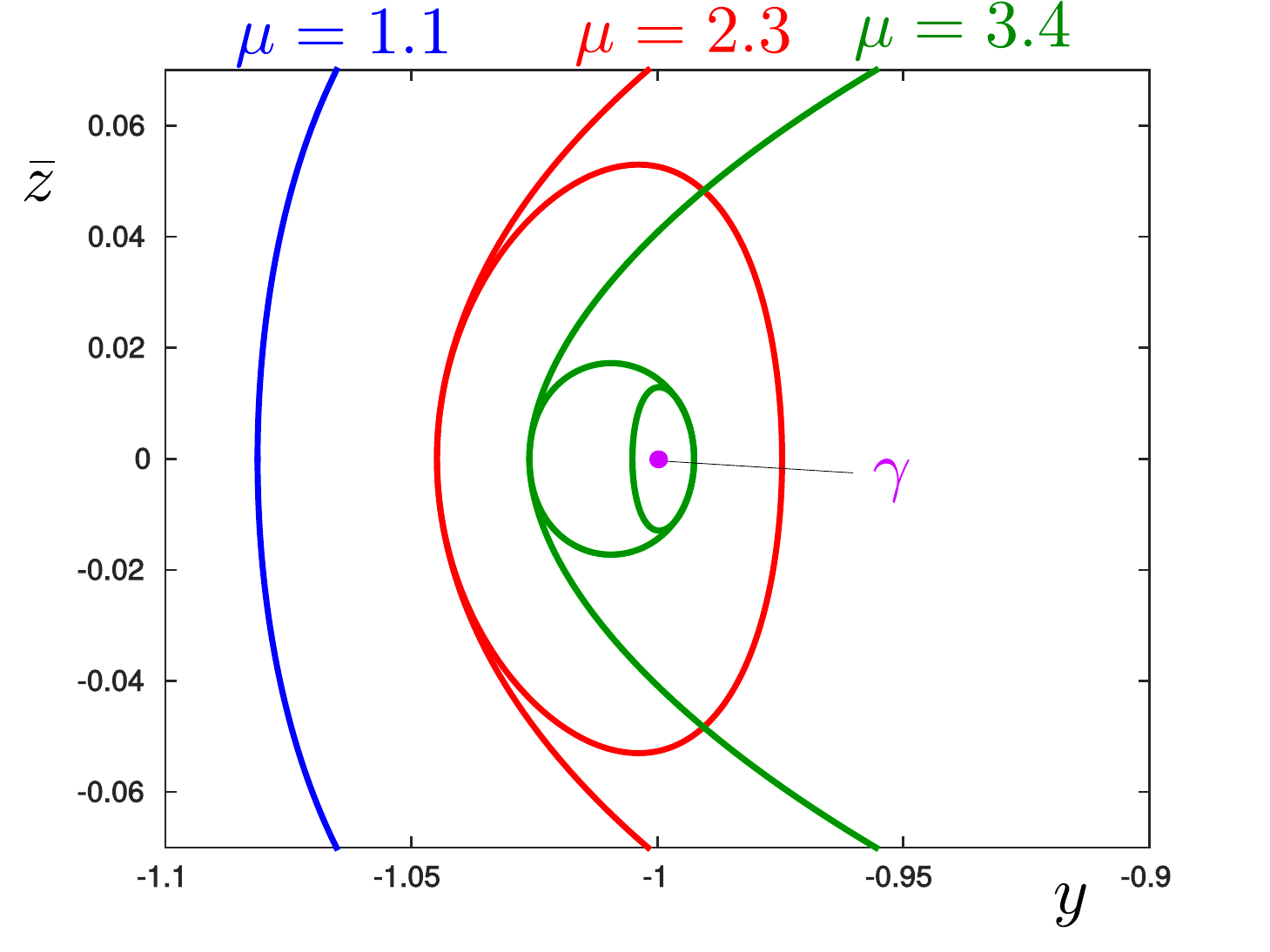}}
\end{center}
\caption{In (a): Three periodic orbits of the (compactified) Falkner-Skan equation, projected onto the $(\bar x,y)$-plane, for three different values of $\mu$ ($\mu=1.1$ in blue, $\mu=2.3$ in red, $\mu=3.4$ in green). These orbits are determined by appropriate backward integration from the set $I$, described in the proof of \thmref{pp}. In (b): The periodic orbits in (a) are now projected onto the $(y,\bar z)$-plane with a zoom near $\gamma$, appearing as a point $(-1,0)$ in this projection. Notice that the periodic orbits twist (one twist defined as one a $360^\circ$ complete rotation) around $\gamma$. For $\mu=1.1$ (blue) there is $1/2$ a twist, for $\mu=2.3$ (red) there are $3/2$ twists, and finally for $\mu=3.4$ (green) there are $5/2$ twists.  }
\figlab{FSpo}
\end{figure}
\subsection{The Nos\'e equation: Bifurcations of unbounded periodic orbits}

The system \eqref{nose} has two time-reversible symmetries given by
\begin{align*}
 \sigma^x:=\text{diag}\,(1,-1,-1)
\end{align*}
and 
\begin{align*}
 \sigma^y:=\text{diag}\,(-1,1,-1),
\end{align*}
as well as one symmetry given by $\sigma^z:(x,y,z)\mapsto (-1,-1,1)$, the superscripts $x$, $y$ and $z$ indicating the coordinates that are fixed by the transformation. Notice that $\sigma^z = \sigma^x\sigma^y$ and the group of symmetries of \eqref{nose} is therefore the group, consisting of $4$ elements, generated by $\sigma^x$ and $\sigma^z$. It is isomorphic to $\mathbb Z_2 \rtimes \mathbb Z_2$.

Furthermore, periodic solutions bifurcate from each integer value of $\mu^{-1} \in \mathbb N$, see \cite{swinnerton-dyer2008a}. For $\mu<1$, it is known that periodic solutions only bifurcate at these values. In \cite{swinnertondyer1995a} they also show that for $\mu>1$, (different) periodic solutions bifurcate for every $\mu\in \mathbb N$, but they do not prove whether periodic solutions bifurcate from other values of $\mu$, see remark following  \cite[Theorem 2]{swinnerton-dyer2008a}. In the following, we will prove this using our time-reversible version of the Melnikov theory:
\begin{theorem}\thmlab{Nose}
 For $\mu>1$ periodic solutions only bifurcate from `infinity' for $\mu\in \mathbb N$. In particular, periodic orbits only emerge for $\mu\prec n$ for each integer $n$.
\end{theorem}

We will prove this theorem in the remainder of this section. For this purpose, it will be convenient to scale \eqref{nose} in the following way: Let 
\begin{align*}
 \kappa:=\sqrt{1-\mu^{-1}},
\end{align*}
and define $\tilde x$ and $\tilde y$ by
\begin{align*}
 x&=\kappa \tilde x,\\
 y&=\kappa \tilde y.
\end{align*}
Then 
 \begin{equation}
\begin{aligned}
 \dot x &= -y-xz,\\
 \dot y &=x,\\
 \dot z &=-\mu +(\mu-1)x^2,
\end{aligned}\eqlab{NoseHere}
\end{equation}
upon dropping the tildes again. For this system, there exists three special solutions of \eqref{nose}
\begin{align*}
 \upsilon:\,(x,y,z)=(0,0,-t),
\end{align*}
as well as
\begin{align*}
 \gamma:&\,\,(x,y,z)=(1, t,-t),\\
 \sigma^z \gamma:&\,\, (x,y,z) =(-1,-t,-t).
\end{align*}
We introduce the Poincar\'e sphere $(\bar x,\bar y,\bar z,\bar w)\in S^3$ by setting
\begin{align*}
 x &= \frac{\bar x}{\bar w},\\
 y&=\frac{\bar y}{\bar w},\\
 z&=\frac{\bar z}{\bar w}.
\end{align*}
By working in the charts defined by `$\bar z=-1$', $\bar y=\pm 1$, and applying the symmetries defined by $\sigma^x$, $\sigma^y$ and $\sigma^z$ we obtain the diagram in \figref{Nose}. Here we adopt the same visualization technique (by projection) used above: The outer sphere corresponds to the equator sphere defined by $\bar w=0$, whereas everything inside is $\bar w>0$. Notice, in particular, that the invariant manifolds $W_{loc}^{cs}(\mu)$ and $W_{loc}^{cu}(\mu)$ are obtained as local center manifolds in the charts `$\bar z=\mp 1$', respectively. The reduced flow on these local manifolds, can be desingularized along $\bar w=0$ to produce the dynamics indicated in the figure. The associated global manifolds $W^{cs}(\mu)$ and $W^{cu}(\mu)$ contain $\gamma$, $\sigma^z\gamma$ and $\upsilon$. These solution curves are shown in purple and orange. (Notice that these lines are actually not straight lines in the projection used in \figref{Nose}. The figure is therefore `artistic'.) We leave out the simple details. In particular, using thick lines, we visualize a (singular) heteroclinic cycle $\Gamma$. It consists of (a) $\gamma$, (b) a segment $L$ (through the fully nonhyperbolic point $q_+$ at $(\bar x,\bar y,\bar z,\bar w)=(0,-1,0,0)$) connecting `the end of' $\gamma$ with `the beginning of' $\sigma^z \gamma$, (c) $\sigma^z \gamma$ and finally $\sigma^z L$. The periodic orbits of \thmref{Nose} will appear as bifurcations from this cycle through bifurcations of intersections of $W^{cs}(\mu)$ and $W^{cu}(\mu)$. But notice the following: 

\begin{figure}[h!]
\begin{center}
{\includegraphics[width=.795\textwidth]{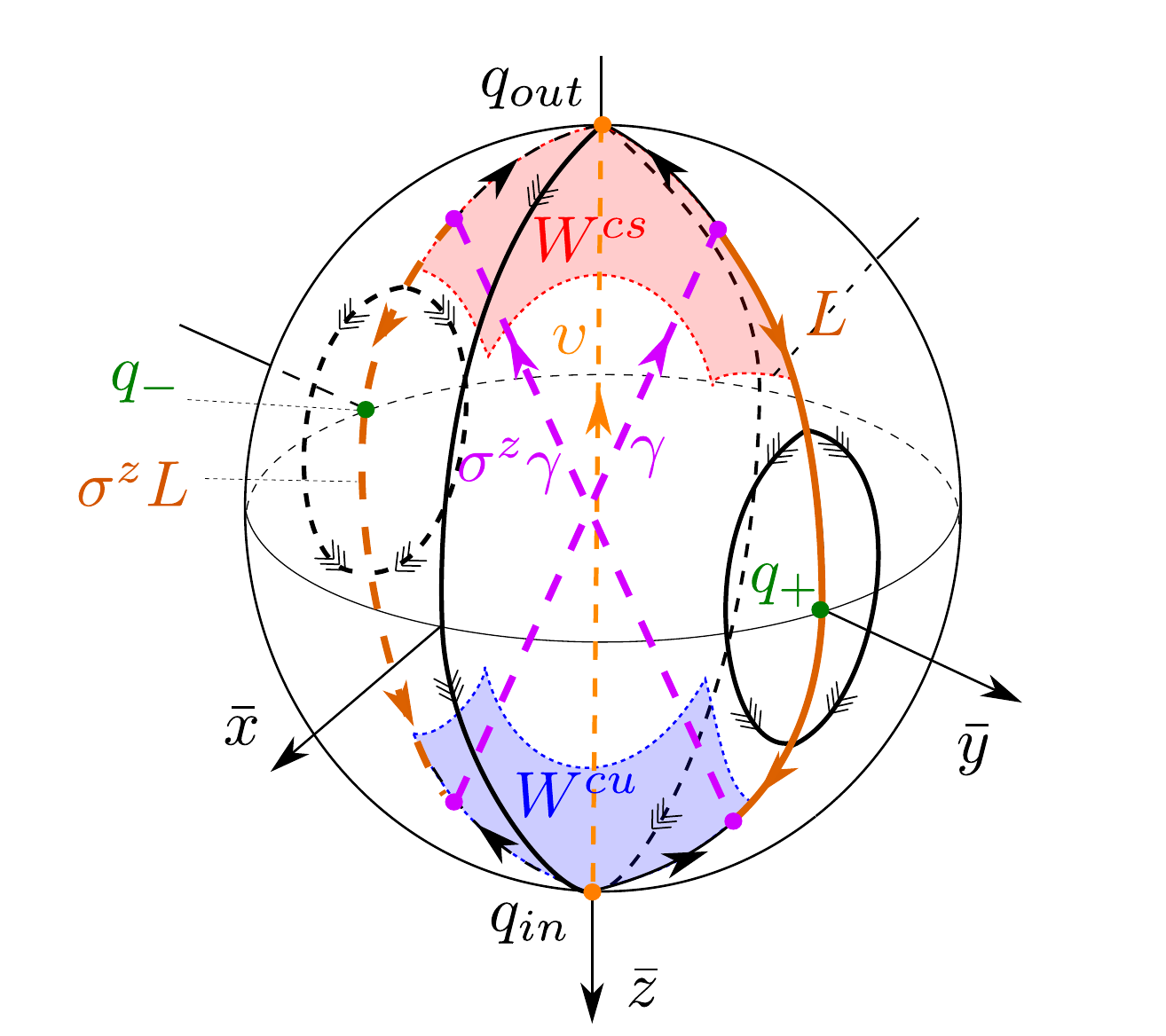}}
\end{center}
\caption{Poincar\'e compactification of the Nos\'e equations for $\mu>1$. There exists three special solutions $\nu$, $\gamma$ and $\sigma^z \gamma$, connecting partially hyperbolic points at infinity. In the figure, we indicate the invariant manifolds $W^{cs}(\mu)$ and $W^{cu}(\mu)$, obtained as center manifolds of these partially hyperbolic points. These points include the sets $L$ and $\sigma^z L$, shown in orange, which, together with the special solutions $\gamma$ and $\sigma^z \gamma$ make up a closed (singular) cycle. Our main result shows that periodic orbits bifurcate from these cycles for $\mu=k$ for every $k\ge 2$ integer only. }
\figlab{Nose}
\end{figure}

Firstly, `new' intersection of $W^{cs}(\mu)$ and $W^{cu}(\mu)$ may not produce periodic orbits. Similar to the bifurcation of secondary canards, these intersections may just converge to the points $q_{\text{in}}$: $(\bar x,\bar y,\bar z,\bar w)=(0,0,1,0)$,  $q_{\text{out}}$: $(\bar x,\bar y,\bar z,\bar w)=(0,0,-1,0)$ as $t\rightarrow \mp \infty$. Indeed, $q_{\text{in}}$ and $q_{\text{out}}$ are hyperbolic, being a source and a sink, for the desingularized slow flow on $W^{cs}_{loc}$ and $W^{cu}_{loc}$, respectively. Notice that these manifolds are unique on the corresponding side of $\gamma$ as stable and unstable sets of $q_{\text{out}}$ and $q_{\text{in}}$, respectively. Consequently, to produce periodic orbits, additional intersections of $W^{cu}(\mu)$ and $W^{cs}(\mu)$ have to be on the `nonunique side' of $\gamma$. 

Secondly, there are other (singular) heteroclinic cycles. For example: (a) $\gamma$, (b) a piece of $L$ before (c) jumping off on a `fast' connecting (shown in black near the green point $q_+$), (d) follow a separate piece on $L$, (e) $\sigma^z \gamma$, etc. However, these cycles do not produce periodic orbits since they are not symmetric. 
The cycle $\Gamma$ is the only symmetric cycle. 

To prove \thmref{Nose}, we first realise, upon rectifying $\gamma$ to the $x_3$-axis, setting $(x,y,z)=(1+x_1,x_3,-x_3+x_2)$, that the system \eqref{NoseHere} satisfies the assumptions of \thmref{recipe}. In particular,
\begin{align*}
 \beta=2(\mu-1).
\end{align*}
Consequently:
\begin{lemma}
 The manifolds $W^{cs}(\mu)$ and $W^{cu}(\mu)$ intersect transversally along $\gamma$ if and only if $2(\mu-1) \notin \mathbb N$.
\end{lemma}
\begin{proof}
Follows from \thmref{recipe}, but for completeness notice the following:
 The variational equations about $\gamma$ takes the following form:
 \begin{equation}\eqlab{NoseVar}
 \begin{aligned}
  \dot z_1 &= t z_1 - z_2-z_3,\\
  \dot z_2 &=z_1,\\
  \dot z_3 &=2(\mu-1) z_1,
 \end{aligned}
 \end{equation}
which upon eliminating $z_2$ and $z_3$, can be written as a Weber equation:
\begin{align}
L_{2(\mu-1)}z_1=0,\eqlab{NoseWeber}
\end{align}
recall \eqref{Lmexpr}. The result therefore follows from \lemmaref{LmHmHl} and \lemmaref{eqvH4}. Notice, in particular, that for $n:=2(\mu-1)\in \mathbb N$ 
 we obtain, using \eqref{prop2} in \appref{Hermite}, the following algebraic solution of \eqref{NoseVar}
\begin{align}
 z = \begin{pmatrix}
      H_n(t/\sqrt{2})\\
      \frac{1}{\sqrt{2} (n+1)} H_{n+1}(t/\sqrt{2})\\
      \frac{n}{\sqrt{2} (n+1)} H_{n+1}(t/\sqrt{2})
     \end{pmatrix}.\eqlab{NoseWeber2}
\end{align}
\end{proof}
Completely analogously to \secref{FSsec} for the Falkner-Skan equation, we fix any $n\in \mathbb N$ and define $\alpha$ by 
\begin{align}
 \mu = \frac{n}{2}+1+\alpha,\eqlab{alphaNose}
\end{align}
and let $D(v,\alpha)$ denote the resulting Melnikov function, the roots of which correspond to intersections of $W^{cs}(\mu)$ and $W^{cu}(\mu)$ near $\gamma$. Following \thmref{recipe} we can therefore evaluate the appropriate Melnikov integrals in closed form by using the recipe in \remref{procedure}. In this way, we obtain the following result.
\begin{proposition}\proplab{NoseBif}
 Let $k\in \mathbb N$ be so that 
\begin{align*}
 n = \begin{cases}
      2k-1 & n=\textnormal{odd}\\
      2k & n=\textnormal{even}
     \end{cases}.
     \end{align*}
Then 
\begin{enumerate}
 \item \label{nodd3} For $n=\textnormal{odd}$, $D(v,\alpha)=0$ is locally equivalent with a pitchfork bifurcation:
 \begin{align}
  \tilde v (\tilde \alpha +\tilde v^2) = 0.
 \end{align}
\item \label{neven3} For $n=\textnormal{even}$, $D(v,\alpha)=0$ is locally equivalent with the transcritical bifurcation:
\begin{align}
    \tilde v(\tilde \alpha+(-1)^{k}\tilde v)=0.\eqlab{transcriticalNose2}
 \end{align}
\end{enumerate}
In each case, the local conjugacy $\phi:(v,\alpha)\mapsto (\tilde v,\tilde \alpha)$ satisfies $\phi(0,0)=(0,0)$ and
\begin{align}
 D\phi(0,0) = \textnormal{diag}\,(d_1(n),d_2(n)) \quad \mbox{with $d_i(n)>0$ for every $n$.}\nonumber
\end{align}
\end{proposition}
\begin{proof}
 See \appref{Nose}.
\end{proof}

To complete the proof of \thmref{Nose}, we first realise that when $n=2k-1=\textnormal{odd}$ -- in which case $\mu\notin \mathbb N$ -- then the roots of $D$ produce additional symmetrically related solutions, coexisting for $\mu \prec \frac{n}{2}+1=k+\frac12$. Working from the diagram in \figref{Nose}, we conclude, see also \propref{secondarcanards2k}, that one of these solutions is asymptotic to $q_{\text{out}}$ whereas the other one is backwards asymptotic to $q_{\text{in}}$. Hence no periodic bifurcate from $n=\textnormal{odd}$. On the other hand, for $n=2k=\textnormal{even}$, then the transcritical bifurcation produce for $\mu\sim k+1$ a `secondary' intersection $\gamma^{sc}(\mu)$, with $\gamma^{sc}(k)=\gamma$, of $W^{cs}(\mu)$ and $W^{cu}(\mu)$. Since each $\gamma^{sc}(\mu)$ is symmetric, it will be asymptotic to $q_{\text{out}}$ and $q_{\text{in}}$ for $t\rightarrow \pm \infty$ on one side of $\alpha =0$. On the other side, however, it will follow $L$, recall \figref{Nose} for $t$ large enough. To distinguish the two cases, we proceed as in the proofs of \propref{gammasctranscrit} and \thmref{pp}. In particular, we first note -- following \eqref{VNosespace} in \appref{Nose} -- that $\gamma^{sc}(\mu)$ intersects $\Sigma$ along the $x$-axis. Denote the intersection point by $(x(\mu),0,0)$ and suppose that $\mu\prec k+1$. Then $\alpha \prec 0$ and by \eqref{transcriticalNose2}, 
\begin{align}
 \text{sign}\,(x(\mu)-1) = \text{sign}\,(-1)^{k}.\eqlab{signXmu}
\end{align}
Consider the solution \eqref{NoseWeber2} of \eqref{NoseVar}, repeated here for convenience
\begin{align}
 z = \begin{pmatrix}
      H_{2k}(t/\sqrt{2})\\
      \frac{1}{\sqrt{2} (2k+1)} H_{2k+1}(t/\sqrt{2})\\
      \frac{n}{\sqrt{2} (2k+1)} H_{2k+1}(t/\sqrt{2})
     \end{pmatrix},\eqlab{NoseWeber3}
\end{align}
with initial condition 
\begin{align}
 (H_{2k}(0),0,0)^T.\eqlab{NoseWeber30}
\end{align}
By \eqref{prop3} in \appref{Hermite}, we realise that the sign of the first component in \eqref{NoseWeber30} coincides with the sign of \eqref{signXmu}. But then, since the second component of \eqref{NoseWeber3} is positive for all $t$ sufficiently large, we conclude that $\gamma^{sc}(\mu)$ for $\mu\prec k+1$ follows $L$ for $t$ large enough. To construct periodic orbits, we fix $W^{cs}(\mu)$ by flowing the points near $L$, of the form $(0,y,0)$ for $y$ large enough,  backwards. This fixes a copy of $W^{cs}(\mu)$. In this way, $\gamma^{sc}(\mu)$ intersects $z=0$ for the first time in forward time in a point $(0,y(\mu),0)$, where $y(\mu)\rightarrow \infty$ as $\mu\rightarrow k+1^-$. Since $\gamma^{sc}(\mu)$ is symmetric with respect to the time-reversible symmetry $\sigma^x$, the first intersection in backwards time is at the point $(0,-y(\mu),0)$. But then upon applying the symmetry $\sigma^z$, we obtain a closed orbit that approaches the singular heteroclinic cycle $\Gamma$ as $\mu\rightarrow k+1^-$. The periodic orbit intersects $z=0$ four times: Once near $\gamma \cap \Sigma$ at $(x,y,z)=(x(\mu),0,0)$, once at  $(0,y(\mu),0)$, then near $\sigma^z\gamma\cap \Sigma$ at $(x,y,z)=(-x(\mu),0,0)$, and then finally at $(0,-y(\mu),0)$. 

Periodic orbits therefore only bifurcate from infinity for $\mu>1$ when $\mu\in \mathbb N$, appearing for $\mu\prec n$ for each integer $n\ge 2$. 

\begin{remark}
 For $\mu<1$, $\kappa \in i \mathbb R$ and therefore only $\upsilon$ exists. In fact, $\upsilon$ bifurcates in a pitchfork-like bifurcation at $\mu=1$ in such a way that $q_{\text{out}}$ becomes a saddle for the reduced flow on $W^{cs}_{loc}$. By following \thmref{recipe}, and reducing the variational equations of \eqref{nose} along $\upsilon$ to a Weber equation, it is again straightforward to show that $W^{cs}(\mu)$ and $W^{cu}(\mu)$ intersect transversally along $\upsilon$ if and only if $\mu^{-1}\notin \mathbb N$.   
\end{remark}
\section{Conclusion}\seclab{conclusion}
In this paper, we have applied a time-reversible version of the Melnikov theory for nonhyperbolic unbounded connection problems in \cite{wechselberger2002a} to the bifurcations of canards in the folded node normal form. In particular, we proved -- for the first time -- the existence of a pitchfork bifurcation for $\mu=\text{even}$. Our time-reversible setting also allowed for a new description of the `secondary canards' emerging from the bifurcations at $\mu\in \mathbb N$, see \secref{blowup}. The connection to the Weber equation as well as properties of the Hermite polynomials were essential to our proof of \thmref{main1}. But the results in \secref{recipe}, specifically see \thmref{recipe} and \remref{procedure}, highlight that the Weber equation is `synonymous' with quadratic, time-reversible systems satisfying (H4) and (H2) with the unbounded symmetric orbit $\gamma$ linear in $t$ and independent of $\alpha$. This is also expected, as noted by \cite{swinnerton-dyer2008a}, since the Weber equation is the `simplest' non-autonomous equation with a non-trivial time-reversible symmetry. 
I believe that it is possible to obtain closed-form expressions for the Melnikov integrals in \cite{mitry2017a} related to the bifurcations of faux canards for the folded saddle singularity. Although these problems do not fit our general setting in \secref{recipe}, the Weber equation also appears naturally for these problems.


In \secref{more}, we also applied our approach to the Falkner-Skan equation and the Nos\'e equations. In particular, we provided a new proof of the emergence of periodic orbits, bifurcating from heteroclinic cycles at infinity, in these systems using more standard methods of dynamical systems theory. In particular, we showed that for the Nos\'e equations periodic orbits only bifurcate from $\mu\in \mathbb N$, a result that had escaped \cite{swinnerton-dyer2008a}. In future work, it would be natural to use the geometric framework provided by this theory to study the emergence of chaos in these two systems. 


\section*{Acknowledgement}
I would like to thank Martin Wechselberger for his encouragement and for providing valuable feedback on an earlier version of this manuscript. 
\appendix
\section{Properties on the Hermite polynomials}\applab{Hermite}

The following properties of the ``physicist'' Hermite polynomials:
\begin{align*}
 H_n(x) = \left(2x-\frac{d}{dx}\right)^n \cdot 1,
\end{align*}
is standard, see e.g. \cite{trifonov2011a}. 
\begin{lemma}
For every $n\in \mathbb N$
 \begin{align}
  H_{n+1}(t) &= 2s H_{n}(s)-H_{n}'(s),\eqlab{prop1}\\
  H_{n}'(s) &=2n H_{n-1}(s),\eqlab{prop2}\\
  H_n(0) &=\left\{\begin{array}{cc}
                  0 & \text{$n=$ odd}\\
                  (-2)^{n/2} (n-1)!!& \text{$n=$ even}
                  \end{array}\eqlab{prop3}
\right.
 \end{align}
 and
 \begin{align}
  \int_{-\infty}^\infty e^{-t^2/2} H_{n}(t/\sqrt{2})H_{m}(t/\sqrt{2}) =                                                                               \sqrt{2\pi} 2^n n!\delta_{nm},\eqlab{HOrtho}
 \end{align}
where $\delta_{nm}$ is the Kronecker delta.

Furthermore, for every $n$, $m\in \mathbb N$:
\begin{align}
 H_n(s) H_m(s) = \sum_{j=0}^{\text{min}(n,m)} \begin{pmatrix}
                                               m\\
                                               j
                                              \end{pmatrix}\begin{pmatrix}
                                               n\\
                                               j
                                              \end{pmatrix}2^j j! H_{n+m-2j}(s),\eqlab{productRuleHnm}
\end{align}
Finally, for every $(n,m,l)\in \mathbb N^3$ that satisfies the triangle property and for which $s=(n+m+l)/2\in \mathbb N$:
\begin{align}
 \int_{-\infty}^\infty e^{-t^2/2} H_n(t/\sqrt{2}) H_m(t/\sqrt{2})H_l(t/\sqrt{2}) dt = \sqrt{2\pi} 2^{s} \frac{n!m!l!}{(s-n)!(s-m)!(s-l)!}.\eqlab{intHnHmHl}
\end{align}
If $(n,m,l)$ does not satisfy the triangle inequality or if $s\notin \mathbb N$, then the integral in \eqref{intHnHmHl} is $0$.
\end{lemma}

\section{Comparison with \cite{mitry2017a}}\applab{compare}
In \cite[App. A, p. 595]{mitry2017a} the third order Melnikov integral \eqref{Dv3Final} is evaluated numerical for $k=1,\ldots,10$. However, we cannot compare the results directly since this reference considers $\mu\in (0,1)$; in this case $\upsilon$ is the weak canard while $\gamma$ is the strong one. Nevertheless, these two cases are obviously equivalent and we can map the $\mu \in(0,1)$ system into the $\mu>1$ system, considered in the present paper, through the following transformation:
\begin{align*}
( x,
 y,
 z,
 t,
 \mu)\mapsto \begin{cases}
\tilde x= \frac{1}{\mu}x,\\
\tilde y= \sqrt{\mu} y,\\
\tilde z=\frac{1}{\sqrt{\mu}} z,\\
\tilde t=\sqrt{\mu}t,\\
\tilde \mu =\mu^{-1},
\end{cases}
\end{align*}
upon dropping the tildes. Specifically, this transformation maps $\gamma$ and $\upsilon$ for $\mu\in (0,1)$ into $\upsilon$ and $\gamma$ for $\mu^{-1}>1$, respectively. But \cite{mitry2017a} also rectifies $\gamma$ (which is $\upsilon$ for $\mu \in (0,1)$) in a slightly different way. A simple computation shows that $(\tilde x,\tilde y,\tilde z)$ in \cite[Eq. (15)]{mitry2017a} is related to $(x_1,x_2,x_3)$ in \eqref{rectify} as follows:
\begin{equation}\eqlab{ddd}
\begin{aligned}
 \tilde x &=\frac{1}{\mu} x_1,\\
 \tilde y &=\sqrt{\mu} (x_2+x_3),\\
 \tilde z &=-\frac{1}{2\sqrt{\mu}}x_2,
\end{aligned}
\end{equation}
upon also replacing $\mu$ by $\mu^{-1}$. Let $D_{MW}(v,\mu^{-1})$ be the Melnikov function in \cite[Proposition 29]{mitry2017a} for $\rho=v$ and $r=0$. Then from \eqref{ddd} and \cite[Eq. (89)]{mitry2017a} it follows that:
\begin{align*}
 D(v,\alpha) = (n+\alpha)D_{MW}\left(-\frac{1}{2\sqrt{n+\alpha}} v,\frac{1}{n+\alpha}\right),
\end{align*}
where $D$ is the Melnikov function used in the present paper. 
Hence, 
\begin{align}
\frac{\partial^3 D}{\partial v^3}(0,0) &= -\frac{1}{8\sqrt{n}}\frac{\partial^3 D_{MW}}{\partial v^3}(0,0).\eqlab{finalEquations}
\end{align}
In \tabref{tbl1}, we compare each side of this equation, using our analytical expression \eqref{Dv3Final} on the left hand side, whereas on the right hand side we use the numerical values in the table on \cite[p. 595]{mitry2017a}. See further explanation in the table caption. We conclude that the results are in agreement (and attribute the tiny differences, indicated in red, to round off errors).  
   \begin{table}[h]
    \renewcommand\arraystretch{2}
\begin{tabular}{|c|c|c|c|}
\hline
            $n=2k$ & $\frac{\partial^3 D_{MW}}{\partial v^3}(0,0)$ & $-\frac{1}{8\sqrt{n}}\frac{\partial^3 D_{MW}}{\partial v^3}(0,0)$ & $\frac{\partial^3 D}{\partial v^3}(0,0)$ \\
    \hline
     $2$ & $-4.0837336724863...\times 10^3$ & $360.9544714...$ & $360.9544714...$\\
     \hline
     $4$ & $-9.1263550336787...\times 10^5 $ & $57039.71895...$ & $57039.7189$\textcolor{red}{$6$}$...$\\
     \hline
     $6$ & $-1.2403985652051...\times 10^8 $ & $6.329882421...\times 10^6$ & $6.32988242$\textcolor{red}{$0$}$...\times 10^6$\\
     \hline
     $8$ & $-1.3867566218372... \times 10^{10}$ & $6.128656321... \times 10^8$ & $6.12865632$\textcolor{red}{$18$}$... \times 10^8$\\
     \hline
     $10$ & $-1.3996176586682...\times 10^{12}$ & $5.532474570...\times 10^{10}$ & $5.5324745$\textcolor{red}{$68$}$...\times 10^{10}$\\
     \hline
     $12$ & $-1.3282386742790... \times 10^{14}$ & $4.792868474...\times 10^{12}$ & $4.792868474...\times 10^{12}$\\
     \hline
        $14$ & $-1.2108610331032... \times 10^{16}$ & $4.045202792...\times 10^{14}$ & $4.045202792...\times 10^{14}$\\
     \hline
     $16$ & $-1.0738223745005... \times 10^{18}$ & $3.355694922...\times 10^{16}$ & $3.35569492$\textcolor{red}{$0$}$...\times 10^{16}$\\
     \hline
     $18$ & $-9.3381989535112... \times 10^{19}$ & $2.751293251...\times 10^{18}$ & $2.751293251...\times 10^{18}$\\
     \hline
     $20$ & $-8.0059501510523 ... \times 10^{21}$ & $2.237731095...\times 10^{20}$ & $2.237731095...\times 10^{20}$\\
     \hline
\end{tabular}
\caption{Comparison of our closed-form Melnikov integral \eqref{Dv3Final} with the values in \cite{mitry2017a}. The first and second column show all of the even values considered in \cite{mitry2017a}. The third column shows the values of the right hand side of \eqref{finalEquations}, using the values in the first two columns, whereas the final column uses the expression in \eqref{Dv3Final}. In red we indicate the slight deviations between the last two columns. We attribute these tiny differences to round off errors.  }
\tablab{tbl1}
    \end{table}

\section{The Falkner-Skan equation: Proof of \propref{FS}}\applab{FS}
Let $\tilde z$ be defined as $(x,y,z)=\gamma(t)+\tilde z$. Then we have 
\begin{align}
 \dot z_1 &= z_2,\nonumber\\
 \dot z_2 &=z_3,\nonumber\\
 \dot z_3 &=tz_3-n z_2+g(z,\alpha),\nonumber
\end{align}
where
\begin{align*}
 g(z,\alpha):= -2\alpha z_2-z_1z_3+\left(\frac{n}{2}+\alpha\right)z_2^2,
\end{align*}
upon dropping the tildes. Then by \eqref{zpFShere}, we obtain the following regarding the state transition matrix:
\begin{align}
 \Phi(t,0) &= \begin{pmatrix}
             1 & * & \frac{1}{2n(n+1)H_{n-1}(0)}\left(H_{n+1}(t/\sqrt{2})-H_{n+1}(0)\right)\\
              0 & *&             \frac{1}{\sqrt{2} n H_{n-1}(0)}H_n(t/\sqrt{2})\\
              0 & *&             \frac{1}{H_{n-1}(0)}H_{n-1}(t/\sqrt{2})
             \end{pmatrix},\quad n=\textnormal{odd},\nonumber\\
             \Phi(t,0)& = \begin{pmatrix}
              1 & \frac{1}{\sqrt{2}(n+1)H_n(0)}H_{n+1}(t/\sqrt{2}) &*\\
             0 & \frac{1}{H_n(0)} H_n(t/\sqrt{2})& *\\
              0 & \frac{\sqrt{2} n}{H_n(0)}H_{n-1}(t/\sqrt{2})&*
             \end{pmatrix},\quad n=\textnormal{even}. \eqlab{PhiEven2}
\end{align}
Consequently,
\begin{align}
 V& = \textnormal{span}\, e_v,\quad \left\{\begin{array}{cc}
             e_v = (0,0,1)^T  & n=\textnormal{odd}\\
            e_v = (0,1,0)^T & n=\textnormal{even}
            \end{array}\right., \eqlab{VFSspace}\\
W&= \textnormal{span}\, e_w,\quad \left\{\begin{array}{cc}
             e_w=(0,1,0)^T & n=\textnormal{odd}\\
             e_w=  (0,0,1)^T & n=\textnormal{even}
            \end{array}\right.,\nonumber
\end{align}
recall (H4) and \eqref{WDefinition}. Also $U=\textnormal{span}(1,0,0)^T$ for all $n\in \mathbb N$. Therefore by \eqref{MelnikovDNew}:
\begin{align*}
\sigma_v &= \begin{cases} 
         -1 &n=\textnormal{odd}\\
         1 &n=\textnormal{even}
           \end{cases},\\
\sigma_w &= \begin{cases} 
         1 &n=\textnormal{odd}\\
         -1 &n=\textnormal{even}
           \end{cases}
\end{align*}
and hence
\begin{align}
 D(v,\alpha) = \left\{\begin{array}{cc }
 h_{cs}(-v,\alpha)-h_{cs}(v,\alpha)&n=\textnormal{odd}\\
         -2h_{cs}(v,\alpha) &n=\textnormal{even}
                \end{array}\right..\nonumber
\end{align}
Consequently, for $n=\textnormal{odd}$, $v\mapsto D(v,\alpha)$ is an odd function for every $\alpha$. On the other hand, for $n=\textnormal{even}$ roots of $D(\cdot,\alpha)$ correspond to symmetric solutions, being fixed with respect to the symmetry $\sigma$. Furthermore, 
using
\begin{align*}
 \psi_*(t)& = \begin{pmatrix}
              0\\
              \frac{1}{H_{n-1}(0)}e^{-t^2/2} H_{n-1}(t/\sqrt{2})\\
              -\frac{1}{\sqrt{2} H_{n-1}(0))}e^{-t^2/2} H_n(t/\sqrt{2})
             \end{pmatrix},n=\textnormal{odd}\\
             \psi_*(t) &= \begin{pmatrix}
              0\\
              -\frac{\sqrt{2}}{H_n(0)} e^{-t^2/2} H_{n-1}(t/\sqrt{2})\\
             \frac{1}{H_n(0)} e^{-t^2/2} H_n(t/\sqrt{2}
             \end{pmatrix},n=\textnormal{even}.
\end{align*}
which follows from a simple calculation, 
we obtain
\begin{align}
 D(v,\alpha) &= 2\int_0^\infty e^{-t^2/2} \times\nonumber \\
 &\left\{\begin{array}{cc} 
                                        \frac{1}{2\sqrt{2} n H_{n-1}(0)}H_n(t/\sqrt{2})  \left(g(z_*(-v,\alpha)(t),\alpha)-g(z_*(v,\alpha)(t),\alpha)\right) & n=\textnormal{odd}\\
                                        \frac{1}{H_{n}(0)}H_{n}(t/\sqrt{2})  g(z_*(v,\alpha)(t),\alpha)& n=\textnormal{even}
                      \end{array}\right. dt.\eqlab{Dformula2}
\end{align}
 We now focus on $n=\textnormal{even}$, which is easier,  and prove the transcritical case. The details of $n=\textnormal{odd}$ and the pitchfork are lengthier, but similar to the details of the proof of \thmref{main1} item (\ref{neven}), see also \remref{procedure}, and therefore left out. 
 
 Let therefore $n=2k$, such that $\mu=k+\alpha$, and write $z':=\frac{\partial}{\partial v}z_*(0,0)$. Following \eqref{PhiEven2}, we have
 \begin{align}
  z' =  \begin{pmatrix}
               \frac{1}{\sqrt{2}(n+1)H_n(0)}H_{n+1}(t/\sqrt{2})\\
             \frac{1}{H_n(0)}H_n(t/\sqrt{2})\\
              \frac{\sqrt{2} n}{H_n(0)} H_{n-1}(t/\sqrt{2})
             \end{pmatrix},\eqlab{zpFS}
 \end{align}

Then upon differentiating \eqref{Dformula2}$_{n=\textnormal{even}}$  with respect to $\alpha$ and $v$ we have 
 \begin{align*}
  \frac{\partial^2 D}{\partial v\partial \alpha}(0,0)&= \frac{2}{H_{2k}(0)} \int_0^\infty e^{-t^2/2} H_{2k}(t/\sqrt{2})(-2z_2') dt\\
  &= -\frac{2}{H_{2k}(0)^2} \int_{-\infty}^\infty e^{-t^2/2} {H_{2k}(t/\sqrt{2})}^2 dt\\
  &= -\frac{2\sqrt{2\pi}(2k)!}{(2k-1)!!^2}=-\frac{2\sqrt{2\pi}(2k)!!}{(2k-1)!!},
 \end{align*}
using \eqref{PhiEven2} as well as \eqref{prop3} and \eqref{HOrtho} in \appref{Hermite}. Similarly, by differentiating \eqref{Dformula2}$_{n=\textnormal{even}}$ twice with respect to $v$ we have 
\begin{align*}
 \frac{\partial^2 D}{\partial v^2}(0,0) &=\frac{2}{H_{2k}(0)} \int_0^\infty e^{-t^2/2} H_{2k}(t/\sqrt{2})(-2z_1'z_3'+2k(z_2')^2) dt\\
 &:=I_1+I_2,
\end{align*}
where
\begin{align*}
 I_1 &= -\frac{4k}{(2k+1)H_{2k}(0)^3} \int_{-\infty}^\infty e^{-t^2/2} H_{2k}(t/\sqrt{2}) H_{2k+1}(t/\sqrt{2})H_{2k-1}(t/\sqrt{2})dt,\\
 I_2&=\frac{2k}{H_{2k}(0)^3} \int_{-\infty}^\infty e^{-t^2/2} H_{2k}(t/\sqrt{2})^3dt,
\end{align*}
using \eqref{zpFS}.
To compute these integrals we use \eqref{intHnHmHl} in \appref{Hermite} and obtain the following expression
\begin{align*}
 I_2 &= \frac{2k}{H_{2k}(0)^3}\sqrt{2\pi} 2^{3k} \frac{(2k)!^3}{k!^3} =(-1)^k\frac{2k\sqrt{2\pi} (2k)!^3}{(2k-1)!!^3 k!^3}=(-1)^k \frac{2k\sqrt{2\pi} (2k)!!^3}{k!^3},
\end{align*}
using \eqref{prop3}, and, after some simple calculations,
\begin{align*}
 I_1 & = -\frac{1}{k+1}I_2.
\end{align*}
Consequently, 
\begin{align*}
 \frac{\partial^2 D}{\partial v^2}(0,0) = \frac{k}{k+1} I_2 = (-1)^k \frac{2k^2\sqrt{2\pi} (2k)!!^3}{(k+1)k!^3}
\end{align*}
By singularity theory \cite{golubitsky1988a} this proves the transcritical bifurcation and the local equivalence (upon replacing $D$ by $-D$) with the normal form \eqref{transcritical2}. 

\section{The Nos\'e equations: Proof of \propref{NoseBif}}\applab{Nose}

 Let $(x,y,z)=\gamma(t)+\tilde z$. Then
 \begin{align*}
  \dot z_1 &= t z_1 - z_2-z_3+g_1(z,\alpha),\\
  \dot z_2 &=z_1,\\
  \dot z_3 &=n z_1+g_3(z,\alpha),
 \end{align*}
where
\begin{equation}\eqlab{Noseg}
\begin{aligned}
 g_1(z,\alpha)&:=-z_1z_3,\\
 g_3(z,\alpha)&:=2\alpha z_1+\left(\frac12 n+\alpha\right) z_1^2,
\end{aligned}
\end{equation}
upon dropping the tildes. By \eqref{NoseWeber2}, we can compute the following relevant quantities:
\begin{align}
 \Phi(t,0) &= \begin{pmatrix}
              \frac{\sqrt{2} (n+1) }{\sqrt{1+n^2} H_{n+1}(0)}H_n(t/\sqrt{2}) & 0 &*\\
              \frac{1}{\sqrt{1+n^2} H_{n+1}(0)}H_{n+1}(t/\sqrt{2}) & \frac{1}{\sqrt{2}} &* \\
              \frac{n }{\sqrt{1+n^2} H_{n+1}(0)}H_{n+1}(t/\sqrt{2}) &-\frac{1}{\sqrt{2}} & *
             \end{pmatrix}V_n,\quad n=\textnormal{odd},\eqlab{NosePhiOdd}\\
             \Phi(t,0)& = \begin{pmatrix}
              \frac{1}{H_n(0)}H_{n}(t/\sqrt{2}) & 0&*\\
             \frac{1}{\sqrt{2}(n+1)H_n(0)} H_{n+1}(t/\sqrt{2})& \frac{1}{\sqrt{2}} & *\\
             \frac{n}{\sqrt{2}(n+1)H_n(0)} H_{n+1}(t/\sqrt{2}) & -\frac{1}{\sqrt{2}} &*
             \end{pmatrix}V_n,\quad n=\textnormal{even}, \eqlab{NosePhiEven}
\end{align}
for some (unspecified) constant matrix $V_n$ (to ensure that $\Phi(0,0)=\textnormal{id}$), which will not be important in the following, 
and consequently
\begin{align}
 V& = \textnormal{span}\, e_v,\quad \left\{\begin{array}{cc}
             e_v = \left(0,\frac{1}{\sqrt{1+n^2}},\frac{n}{\sqrt{1+n^2}}\right)^T  & n=\textnormal{odd}\\
            e_v = (1,0,0)^T & n=\textnormal{even}
            \end{array}\right., \eqlab{VNosespace}\\
W&= \textnormal{span}\, e_w,\quad \left\{\begin{array}{cc}
             e_w=(1,0,0)^T & n=\textnormal{odd}\\
             e_w=  (0,1,1)^T & n=\textnormal{even}
            \end{array}\right.,\nonumber
\end{align}
recall (H4) and \eqref{WDefinition}. Also $U=\textnormal{span}(0,1,-1)^T$ for all $n\in \mathbb N$. Therefore by \eqref{MelnikovDNew}:
\begin{align*}
\sigma_v &= \begin{cases} 
         -1 &n=\textnormal{odd}\\
         1 &n=\textnormal{even}
           \end{cases},\\
\sigma_w &= \begin{cases} 
         1 &n=\textnormal{odd}\\
         -1 &n=\textnormal{even}
           \end{cases}
\end{align*}
with respect to the symmetry $\sigma^x$,
and hence
\begin{align}
 D(v,\alpha) = \left\{\begin{array}{cc }
 h_{cs}(-v,\alpha)-h_{cs}(v,\alpha)&n=\textnormal{odd}\\
         -2h_{cs}(v,\alpha) &n=\textnormal{even}
                \end{array}\right..\nonumber
\end{align}
Consequently, for $n=\textnormal{odd}$, $v\mapsto D(v,\alpha)$ is an odd function for every $\alpha$. On the other hand, for $n=\textnormal{even}$ roots of $D(\cdot,\alpha)$ correspond to symmetric solutions, being fixed with respect to the symmetry $\sigma^x$. Furthermore, 
using
\begin{align*}
 \psi_*(t)& = \begin{pmatrix}
              \frac{1}{H_{n+1}(0)}e^{-t^2/2} {H_{n+1}(t/\sqrt{2})}\\
              -\frac{\sqrt{2}}{{H_{n+1}(0)}} e^{-t^2/2} {H_n(t/\sqrt{2})}\\
              -\frac{\sqrt{2}}{H_{n+1}(0)} e^{-t^2/2} {H_n(t/\sqrt{2})}
             \end{pmatrix},n=\textnormal{odd}\\
             \psi_*(t) &= \begin{pmatrix}
              -\frac{1}{{\sqrt{2} H_{n}(0)}}e^{-t^2/2} {H_{n+1}(t/\sqrt{2})}\\
              \frac{1}{{H_{n}(0)}}e^{-t^2/2} {H_n(t/\sqrt{2})}\\
              \frac{1}{H_{n}(0)}e^{-t^2/2} {H_n(t/\sqrt{2})}
             \end{pmatrix},n=\textnormal{even}.
\end{align*}
which follows from a simple calculation,
we obtain $D(v,\alpha)=D_1(v,\alpha)+D_3(v,\alpha)$ where
\begin{align}
 D_1(v,\alpha) &= \int_0^\infty e^{-t^2/2} \times\nonumber \\
 &\left\{\begin{array}{cc} 
                                        \frac{1}{H_{n+1}(0)}H_{n+1}(t/\sqrt{2}) \left(g_1(z_*(v,\alpha)(t),\alpha)-g_1(z_*(-v,\alpha)(t),\alpha)\right) & n=\textnormal{odd}\\
                                         -\frac{1}{\sqrt{2} H_{n}(0)} H_{n+1}(t/\sqrt{2}) g_1(z_*(v,\alpha)(t),\alpha)& n=\textnormal{even}\\
                                        \end{array}\right. dt,\eqlab{NoseD1}\\
 D_3(v,\alpha) &= \int_0^\infty e^{-t^2/2} \times\nonumber \\
 &\left\{\begin{array}{cc} 
                                        -\frac{\sqrt{2} }{H_{n+1}(0)}H_{n}(t/\sqrt{2})  \left(g_3(z_*(v,\alpha)(t),\alpha)-g_3(z_*(-v,\alpha)(t),\alpha)\right) & n=\textnormal{odd}\\
                                         \frac{1}{H_{n}(0)}H_n(t/\sqrt{2})  g_3(z_*(v,\alpha)(t),\alpha)& n=\textnormal{even}\\
                                        \end{array}\right. dt.\eqlab{NoseD3}
                                        \end{align}
As described in \thmref{recipe}, we are able to evaluate these integrals by following the procedure in \remref{procedure}:
 \begin{lemma} 
 Let $k\in \mathbb N$ be so that 
\begin{align*}
 n = \begin{cases}
      2k-1 & n=\textnormal{odd}\\
      2k & n=\textnormal{even}
     \end{cases}.
     \end{align*}
Then 
\begin{enumerate}
\item For $n=\textnormal{odd}$, the following holds 
\begin{align*}
 \frac{\partial^2 D}{\partial v\partial \alpha}(0,0) &= -\frac{2\sqrt{2\pi} (2k-1)(2k)!!}{\sqrt{1+(2k-1)^2}}.
\end{align*}
Furthermore, let 
\begin{align*}
 c_{kj} &:= \frac{(2(2k-1)-2j)!}{j!(j+1)!(2k-1-j)!^4}\frac{2k-1}{2k-1-2j}\times \\
 &\left(5-\frac{2}{2k-j}\right)\left(2k\left(1+\frac{j+1}{2k-j}\right)+1+j\right),\\
 d_{kj} &:=\frac{(2(2k-1)-2j)!}{j!(j+1)!(2k-1-j)!^4}\frac{2k(j+1)}{2k-j},
\end{align*}
for $j=0,\ldots,2k-1$. Then
\begin{align*}
 D'''_{vvv}(0,0)&= -\frac{3(2k)!!^4 \sqrt{2\pi} (2k-1)}{2k(1+(2k-1)^2)^{3/2}}\sum_{j=0}^{2k-1} \left(c_{kj}+d_{kj}\right)
\end{align*}
\item For $n=\textnormal{even}$, then 
\begin{align*}
 \frac{\partial^2 D}{\partial v\partial \alpha}(0,0) &=2k\sqrt{2\pi} (2k-1)!!,
\end{align*}
and
\begin{align*}
 D''_{vv}(0,0)&=\frac{8(-1)^k k^4 \sqrt{2\pi} (2k-1)!!^3}{(4k)!^3} \left(1+16k^2(2k+1)\right).
\end{align*}
\end{enumerate}
\end{lemma}
\begin{proof}
We simply differentiate the expressions \eqref{NoseD1} and \eqref{NoseD3} and use \eqref{Noseg},
\begin{align}
 z'&=\begin{pmatrix}
              \frac{\sqrt{2} (n+1) }{\sqrt{1+n^2} H_{n+1}(0)}H_n(t/\sqrt{2}) \\
              \frac{1}{\sqrt{1+n^2} H_{n+1}(0)}H_{n+1}(t/\sqrt{2}) \\
              \frac{n }{\sqrt{1+n^2} H_{n+1}(0)}H_{n+1}(t/\sqrt{2})
             \end{pmatrix},\quad n=\textnormal{odd},\nonumber\\
   z'&=\begin{pmatrix}
              \frac{1}{H_n(0)}H_{n}(t/\sqrt{2}) \\
             \frac{1}{\sqrt{2}(n+1)H_n(0)}H_{n+1}(t/\sqrt{2}) \\
             \frac{n}{\sqrt{2}(n+1)H_n(0)}H_{n+1}(t/\sqrt{2})
             \end{pmatrix},\quad n=\textnormal{even},\eqlab{zpNoseEven}          
\end{align}
by \eqref{NosePhiEven} and \eqref{NosePhiOdd}, where  $z':=\frac{\partial z_*}{\partial v}(0,0)$. Following \remref{procedure}, see \textbf{step (a)}, we then characterize $z{''}=\frac{\partial^2 z_*}{\partial v^2}(0,0)$ using the higher variational equations:
\begin{align*}
 \dot z_1'' &=tz_1''-z_2''-z_3''-2z_1'z_3',\\
 \dot z_2''&=z_1'',\\
 \dot z_3''&=nz_1''+n(z_1')^2.
\end{align*}
By the remaining \textbf{steps b,c,d} we obtain the results. The details are identical to \lemmaref{main1} and therefore left out. 
\end{proof}
For $n=\textnormal{odd}$, notice that whereas all $d_{kj}>0$, the sign of $c_{kj}$ -- due to the factor $2k-1-2j$ in the denominator -- changes from $j=k-1$ to $j=k$, in such a way that
\begin{align*}
 c_{kj}\begin{cases}
        <0 & \text{for all}\quad j\ge k\\
        >0 & \text{for all}\quad j\le k-1
       \end{cases}.
\end{align*}
However, a simple calculation shows that
\begin{align*}
 \left| \frac{c_{k(k-1-l)}}{c_{k(k+l)}}\right|=\left| \frac{5(k+l)+3}{5(k-l)-2}\times \frac{k+l+1}{k-l}\times \frac{(2k+2l)\cdots (2k-2l-1)}{(k+l+1)^2\cdots (k+1-l)^2}\right|>1\times 1\times 1,
\end{align*}
for all $l=0,\ldots,k-1$. But then $$\sum_{j=0}^{2k-1} c_{kj} = \sum_{l=0}^k\vert c_{k(k+l)} \vert\left( \left| \frac{c_{k(k-1-l)}}{c_{k(k+l)}}\right|-1\right)>0,$$ and hence  $D'''_{vvv}(0,0)<0$ for all $k$. 
%
 By singularity theory \cite{golubitsky1988a}, these expressions therefore complete the proof of \propref{NoseBif}. 

\bibliography{refs}
\bibliographystyle{plain}
\end{document}